\documentclass[11pt,a4paper]{article}
\usepackage{epsf,epsfig,amsfonts,amsgen,amsmath,amstext,amsbsy,amsopn,amsthm}
\usepackage{amsmath}
\usepackage{enumerate}
\usepackage{hhline}
\usepackage{multirow}
\usepackage{multicol}
\usepackage{booktabs}
\usepackage{tabularx,array}
\usepackage{lineno}
\usepackage{amsfonts,amsthm,amssymb,bm}
\usepackage{amsfonts}
\usepackage{enumitem}
\usepackage{graphics}
\usepackage{latexsym,bm}
\usepackage{amsfonts,amsthm,amssymb,bbding}
\usepackage{indentfirst}
\usepackage{graphicx}
\usepackage{authblk}

\usepackage{xcolor}

\definecolor{paperblue}{HTML}{174A7E}
\definecolor{paperred}{HTML}{A33A2B}

\usepackage{color}
\usepackage[colorlinks=true,anchorcolor=blue,filecolor=blue,linkcolor=red,urlcolor=blue,citecolor=blue]{hyperref}
\usepackage{float}
\usepackage{subcaption}
\usepackage{tikz,enumerate}

\usepackage{xcolor}
\usetikzlibrary{backgrounds,positioning,shapes.geometric}

\usetikzlibrary{calc}
\usepackage{geometry}
\newtheorem{theorem}{Theorem}[section]

\newtheorem{lemma}[theorem]{Lemma}
\newtheorem{claim}{Claim}[theorem]

\newtheorem{corollary}[theorem]{Corollary}
\newtheorem{conjecture}[theorem]{Conjecture}

\newtheorem{proposition}[theorem]{Proposition}

\usepackage[numbers,sort&compress]{natbib}

\begin{document}

\title{\bf The chromatic number of $\{\mathrm{ISK}_4, \textup{diamond}\}$-free graphs}

\author[1]{Qing Cui\footnote{Email: cui@nuaa.edu.cn.}}
\author[2]{Kaiyang Lan\footnote{Email: kylan95@126.com (Corresponding author).}}

\affil[1]{School of Mathematics, Nanjing University of Aeronautics and Astronautics, Nanjing, 211106, Jiangsu, China}
\affil[2]{School of Mathematics and Statistics, Minnan Normal University, Zhangzhou, 363000, Fujian, China}

\date{}
\maketitle



\begin{abstract}
An $\mathrm{ISK}_4$ is a graph that is isomorphic to a subdivision of $K_4$; a \textit{diamond} is $K_4$ minus one edge; and a \textit{bowtie} consists of two triangles sharing exactly one vertex.
L\'ev\^eque, Maffray, and Trotignon~(2012) conjectured that every $\mathrm{ISK}_4$-free graph is $4$-colorable.
Chen, Chen, Cui, Feng, and Liu~(2021) proved this for $\{\textup{diamond},\textup{bowtie}\}$-free graphs. 
We prove this conjecture for diamond-free graphs.
Moreover, the bound is optimal.
\end{abstract}



\section{Introduction}

We consider only finite simple graphs throughout this paper.
For a positive integer $t$, let $K_t$, $P_t$ and $C_t$ denote the complete graph, the path and the cycle on $t$ vertices, respectively (with \(t\geq 3\) for \(C_t\)).
For graphs $G$ and $H$, we say that $G$ \textit{contains} $H$ if $H$ is isomorphic to an induced subgraph of $G$, and otherwise, $G$ is $H$-\textit{free}. 
For a family $\mathcal{H}$, the graph $G$ is $\mathcal{H}$-\textit{free} if $G$ is $H$-free
for every graph $H\in \mathcal{H}$.
A graph class $\mathcal{G}$ is $\chi$-\textit{bounded} if there is a function $f:\mathbb{N}\to\mathbb{N}$ (a $\chi$-\textit{binding function}) for which $\chi(G)\le f(\omega(G))$ holds for every induced subgraph \(G\) of a graph in $\mathcal{G}$, where $\chi(G)$ and $\omega(G)$ denote the chromatic number and clique number, respectively.
A $\chi$-bounded class is \textit{polynomially $\chi$-bounded} if such an $f$ can be chosen to be a polynomial.
The study of $\chi$-boundedness originates with Gy{\'a}rf{\'a}s \cite{GST87}, and has since attracted extensive research on classes defined by forbidden induced subgraphs.
For background and recent developments, we refer the reader to \cite{CSSS2023,Nguyen2025,BR2004,IS2019,AS2020}.

A \textit{subdivision} of a graph \(G\) is a graph obtained from \(G\) by replacing the edges of \(G\) with internally vertex-disjoint paths of length at least one between their end vertices, with all internal vertices outside \(V (G)\).
Scott~\cite{scott1997induced} conjectured that for every graph $H$, the class of graphs containing no induced subdivision of $H$ is $\chi$-bounded. This conjecture was refuted in general by Pawlik, Kozik, Krawczyk, Laso{\'n}, Micek, Trotter, and Walczak~\cite{pawlik2014triangle}; nevertheless, it is known to be true for several choices of $H$, including all graphs with at most four vertices~\cite{chudnovsky2012excluding}.
Recently, Nguyen, Scott and Seymour~\cite{nguyen2025subdivisions} proved that, for every fixed integer \(t \geq 3\), every \(n\)-vertex graph with \(n \geq 2\) containing no induced subdivision of \(K_t\) has a stable set of size at least $n/\operatorname{polylog}n$, where the polynomial depends on \(t\).
In this paper, we focus on the case $H=K_4$.

An $\mathrm{ISK}_4$ is a graph that is isomorphic to a subdivision of $K_4$.
L\'ev\^eque, Maffray and Trotignon \cite{LMT2012} studied the structure
of $\mathrm{ISK}_4$-free graphs and obtained several decomposition theorems.
Since no $\mathrm{ISK}_4$-free graphs whose chromatic number is five or more are
known, they proposed the following conjecture:

\begin{conjecture}[\cite{LMT2012}]\label{ISK4C}
	Every $\mathrm{ISK}_4$-free graph $G$ satisfies $\chi(G)\leq 4$.
\end{conjecture}


Subsequent work progressively strengthened the available coloring bounds.
For example, Le~\cite{Le2017} obtained the explicit bound $24$ in general; Chen, Chen, Cui, Feng and Liu~\cite{ChenEtAl2020} improved the general bound to $8$.
Chudnovsky, Liu, Schaudt, Spirkl, Trotignon, and Vu\v{s}kovi\'c~\cite{ChudnovskyEtAl2019} proved the optimal statement that every $\{\mathrm{ISK}_4,K_3\}$-free graph is $3$-colorable.

The role of triangles was isolated further by Chen, Chen, Cui, Feng, and Liu~\cite{ChenEtAl2021},
who proved that every $\{\mathrm{ISK}_4,\textup{diamond},\textup{bowtie}\}$-free graph is
$4$-colorable.
Here a \textit{diamond} is $K_4$ minus one edge, and a \textit{bowtie} consists of two triangles sharing exactly one vertex (see~\autoref{fig:diamond-bowtie}).
Most recently, Liu, Sun, and Wang~\cite{LiuSunWang2026} strengthened that result to $3$-colorability and obtained a detailed decomposition of the same class.
However, Conjecture~\ref{ISK4C} remains open in general.
\begin{figure}[htbp]
	\centering
	\begin{tikzpicture}[
		v/.style={circle,draw,fill=black,inner sep=0pt,minimum size=1.5mm}]
		\begin{scope}
			\node[v] (a) at (-.65,0) {};
			\node[v] (b) at (.65,0) {};
			\node[v] (c) at (0,.85) {};
			\node[v] (d) at (0,-.85) {};
			\draw[blue, line width=1pt] (a)--(b) (a)--(c)--(b) (a)--(d)--(b);
			\node at (0,-1.35) {diamond};
		\end{scope}
		\begin{scope}[xshift=4.2cm]
			\node[v] (z) at (0,0) {};
			\node[v] (p) at (-1,.72) {};
			\node[v] (q) at (-1,-.72) {};
			\node[v] (r) at (1,.72) {};
			\node[v] (s) at (1,-.72) {};
			\draw[blue, line width=1pt] (z)--(p)--(q)--(z) (z)--(r)--(s)--(z);
			\node at (0,-1.35) {bowtie};
		\end{scope}
	\end{tikzpicture}
	\caption{The two triangle configurations.}
	\label{fig:diamond-bowtie}
\end{figure}
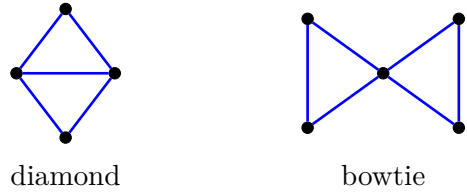

In this paper, we improve the result of Chen et al.~\cite{ChenEtAl2021} by removing the bowtie-free assumption.

\begin{theorem}\label{thm:main-intro}
	Every $\{\mathrm{ISK}_4,\textup{diamond}\}$-free graph is $4$-colorable.
\end{theorem}

This result is best possible.

\begin{theorem}\label{thm:sharp-intro}
	There exists an $\{\mathrm{ISK}_4,\text{diamond}\}$-free graph of chromatic number $4$.
\end{theorem}

%
%
%

A \textit{prism} consists of two disjoint triangles joined by three pairwise internally vertex-disjoint paths with distinct ends in each triangle, with no additional edges among the paths.

The proof has three parts. First, we adapt the wheel arguments of Liu--Sun--Wang and replace every use of the bowtie-free hypothesis, which yields the sector-separation theorem (Theorem~\ref{thm:sector-separation}).
Second, for the core class of \(\{\mathrm{ISK}_4,\text{diamond},\text{prism},K_{3,3}\}\)-free graphs, we find a stable set hitting every triangle (Corollary~\ref{thm:core-transversal}).
Third, we use the L\'ev\^eque--Maffray--Trotignon decompositions (\cite[Lemma~2.2, 3.3 and 7.2]{LMT2012}).


\section{Preliminaries and the proof outline}
\label{sec:preliminaries}

For a vertex $v$, we write $N(v)$ and $N[v]$ for its open and closed neighborhoods, respectively.
For $X\subseteq V(G)$, put \(N(X)=\{u\in V(G)\setminus X: u\text{ has a neighbor in }X\}\).
If $D\subseteq V(G)$, then $N_G(D)$ denotes this same external
neighborhood $N(D)$.  A stable set meeting every
triangle will be called a \textit{stable triangle transversal}.

A \textit{hole} in a graph is an induced cycle of length at least four.  An orientation
of a hole is a choice of one of its two cyclic directions.  For a subgraph
$J$ and a set $X\subseteq V(G)$, write $N_J(X)=N(X)\cap V(J)$; for a
single vertex $v$, abbreviate $N_J(\{v\})$ to $N_J(v)$.  If $R$ is a path,
$R^*$ denotes its interior.
If $X$ is disjoint from a subgraph $J$, an \textit{attachment of $X$ on
	$J$} is a vertex of $J$ with a neighbor in $X$; equivalently, it is the
$J$-end of an $X$--$J$ edge.
A \textit{wheel} is a pair $W=(C,x)$, where $C$ is a hole and $x\notin V(C)$
has at least three neighbors on $C$.  The vertex $x$ is its \textit{centre}, the
hole $C$ is its \textit{rim}, the vertices of $N(x)\cap V(C)$ are its \textit{spokes}, and the subpaths of $C$ between
cyclically consecutive spokes are its \textit{sectors}.  A sector consisting of one
edge is short; all other sectors are long.  The interior of a sector is the
sector with its two spokes deleted.  Distinct sectors are \textit{consecutive}
when they share a spoke; in a wheel with exactly four spokes, two sectors that do
not share a spoke are \textit{opposite}.  The order of a path is its number of
vertices.  If $R$ is a path and $p,q\in V(R)$,
then $pRq$ denotes the unique $pq$-subpath of $R$.  A \textit{tail of $R$ at
	$p$} is either end-subpath of $R$ having ends $p$ and one end of $R$.
For $v\notin V(C)\cup\{x\}$, we say that $v$ is \textit{proper for $W$} if
all neighbors of $v$ on $C$ lie in one sector and
\[
|N(v)\cap V(C)|>2 \quad\Longrightarrow\quad vx\in E(G).
\]
The wheel is \textit{proper} if every vertex outside $V(C)\cup\{x\}$ is
proper for it.
We write
\[
\sigma(W)=|N(x)\cap V(C)|.
\]
We also put $V(W)=V(C)\cup\{x\}$.  A wheel $W$ with $\sigma(W)=4$ is
called an \textit{exact four-spoke wheel}, or simply an \textit{exact-four wheel} (see \autoref{fig:exact-four-wheel}).


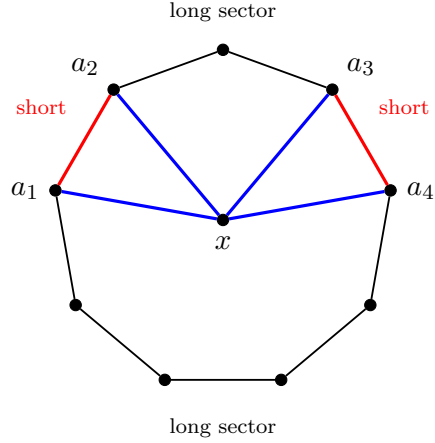
\begin{figure}[htbp]
	\centering
	\begin{tikzpicture}[
		rv/.style={circle,draw,fill=black,inner sep=0pt,minimum size=1.5mm},
		spoke/.style={draw=blue,line width=1.2pt},
		short/.style={draw=red,line width=1.2pt},
		rim/.style={draw,line width=.7pt}]
		
		\def\R{2.25} 
		
		\node[rv] (p)  at (90:\R) {};
		\node[rv,label=above right:$a_3$] (a3) at (50:\R) {};
		\node[rv,label=right:$a_4$] (a4) at (10:\R) {};
		\node[rv] (q1) at (-30:\R) {};
		\node[rv] (q2) at (-70:\R) {};
		\node[rv] (q3) at (-110:\R) {};
		\node[rv] (q4) at (-150:\R) {};
		\node[rv,label=left:$a_1$] (a1) at (170:\R) {};
		\node[rv,label=above left:$a_2$] (a2) at (130:\R) {};
		
		\node[rv,label=below:$x$] (x) at (0,0) {};
		
		\draw[rim] (a2)--(p)--(a3);
		\draw[rim] (a4)--(q1)--(q2)--(q3)--(q4)--(a1);
		
		\draw[short] (a1)--(a2) (a3)--(a4);
		
		\foreach \z in {a1,a2,a3,a4} \draw[spoke] (x)--(\z);
		
		\node[font=\scriptsize] at (0, 2.75) {long sector};
		\node[font=\scriptsize] at (0, -2.75) {long sector};
		\node[font=\scriptsize,text=red] at (-2.4, 1.5) {short};
		\node[font=\scriptsize,text=red] at (2.4, 1.5) {short};
		
	\end{tikzpicture}
	\caption{An exact four-spoke wheel with two opposite short sectors.  The
		normalisation lemma preserves the centre, all four spokes, and the two red
		rim edges.}
	\label{fig:exact-four-wheel}
\end{figure}

A \textit{clique cutset} is a clique whose deletion disconnects the graph.  A
pair of nonadjacent vertices $\{a,b\}$ is a \textit{proper $2$-cutset} if there
is a partition
\begin{equation}\label{eq:proper-cutset-split}
	V(G)=X\mathbin{\dot\cup}\{a,b\}\mathbin{\dot\cup}Y
\end{equation}
such that $X,Y\ne\emptyset$, $X$ is anticomplete to $Y$, both
$G[X\cup\{a,b\}]$ and $G[Y\cup\{a,b\}]$ contain an $a$--$b$ path, and neither
of these two blocks is itself an $a$--$b$ path.  We call
$(X,Y,\{a,b\})$ a split of the cutset.  An \textit{oriented split} is a
split in which the first side $X$ is designated.  A graph is
\textit{series--parallel}
if it contains no subdivision of $K_4$ as a (not necessarily induced)
subgraph.

For completeness, we recall the basic graph called a \textit{rich square}.
Let $S=s_1s_2s_3s_4s_1$ be an induced square.  A \textit{long link of $S$}
is an induced path $P$ with distinct ends $p,p'$ such that either
\[
N_S(p)=\{s_1,s_2\},\quad N_S(p')=\{s_3,s_4\},
\]
or
\[
N_S(p)=\{s_1,s_4\},\quad N_S(p')=\{s_2,s_3\},
\]
and no internal vertex of $P$ has a neighbor in $S$.  A \textit{short link}
is one vertex $p$ with $N_S(p)=V(S)$.  A rich square is a graph $R$
containing such an $S$ for which $R-S$ has at least two components and
every component of $R-S$ is a long or short link of $S$.  Distinct
components, of course, are anticomplete.  This is the convention of
L\'ev\^eque--Maffray--Trotignon \cite{LMT2012}.
A rich square is \textit{long} if all its links are long; equivalently,
it is wheel-free \cite[Section~1]{LMT2012}.

Following the same convention, a complete bipartite or complete tripartite
graph is called \textit{thick} if it contains an induced $K_{3,3}$.

We use the following published results.

\begin{enumerate}[label=(L\arabic*)]
	\item If an ISK$_4$-free graph is not series--parallel, then it contains a
	prism, a wheel, or an induced $K_{3,3}$ \cite[Lemma~2.2]{LMT2012}.\label{L1}
	\item An ISK$_4$-free graph containing an induced $K_{3,3}$ is either a
	thick complete bipartite or complete tripartite graph, or it has a
	clique cutset of size at most three \cite[Lemma~3.3]{LMT2012}.\label{L2}
	\item An ISK$_4$-free graph containing a prism is either the line graph of a
	simple graph of maximum degree at most three, a rich square, or it has
	a clique cutset of size at most three or a proper $2$-cutset
	\cite[Lemma~7.2]{LMT2012}.\label{L3}
	
	\item Let $G$ be a rich square that contains no wheel. Then $G$ is 3-colorable \cite[Lemma~11.3]{LMT2012}.\label{L4.5}
	\item Every $\{\mathrm{ISK}_4,\text{triangle}\}$-free graph is
	$3$-colorable \cite[Theorem 3]{ChudnovskyEtAl2019}.\label{L4}
	\item Every simple graph of maximum degree at most three has a proper
	$4$-edge-coloring, by Vizing's theorem \cite{Vizing1964}.\label{L5}
\end{enumerate}

We shall also use that the class of series--parallel graphs is minor-closed,
and that every nonempty induced subgraph of a series--parallel graph has a
vertex of degree at most two.  In particular, every series--parallel graph is
$2$-degenerate and hence $3$-colorable.

We use the following elementary extension of the linked-hole principle
\cite[Lemma~2.3]{LiuSunWang2026}.  Let $D$ be an induced cycle, possibly
a triangle.  A vertex $v\notin V(D)$ is \textit{linked} to $D$ if there are three induced paths $P_1,P_2,P_3$ such that
\begin{itemize}
	\item $P_1^*\cup P_2^*\cup P_3^*\cup \{v\}$ is disjoint from $D$;
	\item each $P_i$ has one end $v$ and the other end in $D$, and there are no other edges between $\bigcup_i (V(P_i)\setminus V(D))$ and $D$;
	\item for $i,j\in\{1,2,3\}$ with $i\ne j$, $V(P_i)\cap V(P_j)=\{v\}$;
	\item if \(i\neq j\) and $x\in P_i$ is adjacent to $y\in P_j$, then either $v\in\{x,y\}$ or $\{x,y\}\subseteq V(D)$; and
	\item if $v$ has a neighbor $d\in D$, then $d\in V(P_i)$ for some $i$.
\end{itemize}
These conditions ensure that \(D\cup P_1\cup P_2\cup P_3\) is an induced subdivision of \(K_4\), with branch vertices
\(v\) and the three distinct ends on \(D\).
Consequently an ISK$_4$-free graph has no vertex
linked to an induced cycle.  We call this the \textit{linked-cycle principle}.

\subsection*{Proof outline}

The remainder of the paper is organized as follows. 
Section~\ref{sec:local-wheels} proves proper-wheel existence and four-spoke normalization; Section~\ref{sec:local-attachments} establishes the local attachment lemmas. Sections~\ref{sec:wheel-repairs} and~\ref{sec:path-attachments} address 
wheel repairs and the attachments of induced paths, including the complete 
four-spoke base case.
Section~\ref{sec:sector-separation} assembles these results and proves the sector separation theorem. Section~\ref{sec:core} constructs rooted triangle transversals, Section~\ref{sec:proper-two-cut} treats proper 2-cutsets, and Sections~\ref{sec:main-proof} and~\ref{sec:sharpness} complete the coloring proof and establish sharpness.

\section{Local structure of wheels}
\label{sec:local-wheels}

We first record the consequences of sparse neighborhoods and then prove
an external-connector lemma, the existence of a proper wheel, and four-spoke
normalization.  We use the definitions of holes, wheels, sectors,
properness, path interiors, and tails from Section~\ref{sec:preliminaries}.
For vertices $a,b$ on a hole $C$, $aCb$ denotes the subpath specified by
the context.

We begin with an elementary consequence of the two forbidden configurations.

\begin{lemma}\label{lem:sparse-neighborhood}
	If $G$ is $\{\mathrm{ISK}_4,\textup{diamond}\}$-free, then \(\Delta(G[N(v)])\le 1\) for every \(v\in V(G)\).
	Moreover, let $P$ be an induced path and let $y\notin V(P)$ have at least
	three neighbors on $P$, with extreme neighbors $p,q$.  Then $pPq$ has
	at least three edges, and replacing it by $p-y-q$ gives a strictly
	shorter induced path with the same ends as $P$.
\end{lemma}
\begin{proof}
	Suppose that \(y,z,w\in N(v)\) and \(yz,yw\in E(G)\). If \(zw\notin E(G)\),
	then \(G[\{v,y,z,w\}]\) is a diamond.
	If \(zw\in E(G)\), then those four vertices induce \(K_4\), which is itself an \(\mathrm{ISK}_4\). Both alternatives are impossible.
	For the second assertion, at least three neighbors of \(y\) lie on \(pPq\), so this
	subpath has at least two edges. If it were \(p-r-q\), then \(yr\in E(G)\)
	and \(\{y,p,r,q\}\) would induce a diamond. Thus it has at least three edges.
	By the extremal choice of \(p\) and \(q\), the vertex \(y\) has no neighbor
	among the retained vertices of \(P\) other than \(p\) and \(q\). Hence replacing \(pPq\) by \(p-y-q\)
	yields an induced path with the same ends as \(P\), proving the assertion.
\end{proof}

A three-spoke wheel is itself an induced subdivision of $K_4$.  Consequently
every wheel used below has at least four spokes.  Lemma
\ref{lem:sparse-neighborhood} also says that the short sectors form a
matching in the cyclic sequence of sectors.  Thus a $k$-spoke wheel has at
least $\lceil k/2\rceil$ long sectors.  Every wheel therefore has two vertex-disjoint long sectors.  Indeed, if there
are exactly two long sectors, there are four spokes and the two long
sectors are opposite.  If there are at least three, some two are disjoint,
since the sector-intersection graph is a cycle of length at least four
and contains no triangle.

\subsection{An external connector to a hole}

\begin{lemma}\label{lem:external-connector-obstruction}
	Let $G$ be any graph, let $R$ be a hole, let $a,b\in V(R)$ be distinct, and let $A,B$ be
	its two $a$--$b$ subpaths.  Let $Q$ be an induced $x$--$w$ path
	disjoint from $R$, where $x\ne w$.  Suppose that
	\[
	N_R(x)=\{a,b\},\quad N_R(Q^*)=\emptyset,\quad
	\emptyset\ne N_R(w)\cap V(B^*),\quad
	N_R(w)\subseteq V(B)\setminus\{b\}.
	\]
	Then $G[V(R)\cup V(Q)]$ contains an induced subdivision of $K_4$
	or an induced prism, unless there is a vertex $t\in V(B^*)$ such that
	\[
	ab,at\in E(R),\quad N_R(w)=\{a,t\}.
	\]
	In this exceptional case, $R^*=x-bBt-wQx$ is a hole and $(R^*,a)$
	is an exact-four wheel with cyclically ordered spokes $x,b,t,w$ and
	opposite short sectors $xb,tw$.
\end{lemma}

\begin{proof}
	If $wa\in E(G)$, let $t$ be the $w$-neighbor on $B$ closest to $b$.
	If $at\notin E(G)$, the cycle $w-tBbAa-w$ is a hole, and
	$xa,xb,Q$ link $x$ to it.  Extremality of $t$ excludes every further
	$w$-edge on $tBb$; the nonedge $at$ and the other attachment assumptions
	exclude all remaining chords and cross-edges.  If $at\in E(G)$ but
	$ab\notin E(G)$, the six paths
	\[
	Q,\quad xa,\quad wa,\quad wt,\quad at,\quad x-bBt
	\]
	induce a subdivision of $K_4$ with branch vertices $x,w,a,t$.
	If both $at,ab\in E(G)$, then $t$ is the first vertex after $a$ on
	$B$ and the last $w$-neighbor before $b$, so $N_R(w)=\{a,t\}$.  The cycle $x-bBt-wQx$ is a hole:
	$R$ being a hole gives at least two edges on $tBb$, extremality of
	$t$ excludes extra $w$-edges, and $Q^*$ is anticomplete to $R$.
	With center $a$, this hole has exactly the four spokes $x,b,t,w$,
	with opposite short sectors $xb,tw$.
	
	Hence assume $wa\notin E(G)$, and let $r,t$ be the
	first and last $w$-neighbors on $B$, ordered from $a$ to $b$.
	If $r=t$, then $xa,xb,xQw-r$ link $x$ to $R$.  If $r\ne t$ and
	$rt\notin E(G)$, replace $rBt$ in $R$ by $r-w-t$.  Extremality of
	$r,t$ makes the result a hole, to which $xa,xb,Q$ link $x$.
	
	Finally suppose $rt\in E(G)$.  Then $r,t$ are consecutive on $B$
	and are the only neighbors of $w$ on $R$.  If $ab\notin E(G)$, the
	six paths
	\[
	Q,\quad wr,\quad wt,\quad rt,\quad x-aBr,\quad x-bBt
	\]
	induce a subdivision of $K_4$ with branch vertices $x,w,r,t$.
	If $ab\in E(G)$, then $A=ab$, and the triangles $xab$ and $wrt$,
	joined by $Q,aBr,bBt$, form an induced prism.  Chordlessness of $R$,
	inducedness of $Q$, and the displayed attachment assumptions exclude
	every additional edge in both constructions.
\end{proof}

\subsection{Existence of a proper wheel}

\begin{lemma}\label{lem:proper-wheel-no-bowtie}\label{lem:proper-wheel-exists}
	Let $G$ be an \(\{\mathrm{ISK}_4,\text{diamond},
	\text{prism}\}\)-free graph.
	If $G$ contains a wheel, then $G$ contains a proper wheel.
\end{lemma}

\begin{proof}
	Among all wheels $W=(C,x)$ in $G$, choose one with $|V(C)|$ minimum.  We
	prove that $W$ is proper.  A three-spoke wheel is itself an induced
	subdivision of $K_4$, so $W$ has at least four spokes.  Suppose to the contrary that
	$v\in V(G)\setminus(V(C)\cup\{x\})$ is not proper for $W$.
	If, for some sector $S$, the vertex $v$ has at least three neighbors on the
	cycle $xSx$, then either $xSx$ is a hole and $(xSx,v)$ is a wheel with a rim
	strictly shorter than $C$, or $xSx$ is a triangle and the four relevant
	vertices induce $K_4$.  Both alternatives are impossible.  Consequently
	$v$ has at most two neighbors in every sector; if $vx\in E(G)$, it has at
	most one neighbor in $V(S)$ for every sector $S$.
	
	\begin{claim}
		\label{clm:minimal-wheel-gap}
		Let $P$ be a subpath of $C$ with ends $v_1,v_2\in N_C(v)$ and with
		$P^*\cap N(v)=\emptyset$.  Then $P$ contains at most two spokes of
		$W$.
	\end{claim}
	
	\begin{proof}[Proof of Claim~\ref{clm:minimal-wheel-gap}]
		Suppose that $P$ contains at least three spokes, and let $P'$ be the other
		$v_1v_2$-subpath of $C$.  If $P'=v_1v_2$, then
		$N_C(v)=\{v_1,v_2\}$ is contained in the sector containing that rim edge,
		so $v$ is proper, a contradiction.  Hence $P'$ is not an edge, and the
		cycle $vPv$ is a hole on which $x$ has at least three neighbors.
		Minimality of $C$ therefore implies that $P'$ has at most three vertices,
		and consequently $P'=v_1wv_2$.
		
		We must have $wx\in E(G)$.  Indeed, if $wx\notin E(G)$, then $P'$ lies in
		one sector.  If $vw\in E(G)$, the vertex $v$ has three neighbors in that
		sector, while if $vw\notin E(G)$, then
		$N_C(v)=\{v_1,v_2\}$; either conclusion is impossible.  Since $G$ is
		diamond-free, \(vw\notin E(G)\) and \(|N(x)\cap\{v_1,v_2\}|\le 1\).
		After interchanging $v_1$ and $v_2$ if necessary, assume
		$xv_1\notin E(G)$.  Let $x_1$ and $x_2$ be, respectively, the spokes on
		$P$ closest to $v_1$ and to $v_2$.
		
		If $xv\in E(G)$, then $x$ is linked to the hole
		$vv_1wv_2v$ by
		\[
		xv,\quad xw,\quad xx_2Pv_2,
		\]
		a contradiction.  Thus $xv\notin E(G)$.  Now \(D=vv_1Px_1xx_2Pv_2v\) is a hole, and $w$ has precisely the three neighbors $v_1,x,v_2$ on
		$D$.  Hence $(D,w)$ is a three-spoke wheel, equivalently an induced
		subdivision of $K_4$, again a contradiction.  This proves the claim.
	\end{proof}

	\begin{claim}
		\label{clm:minimal-wheel-sector}
		For every sector $S$, either $v$ has a neighbor in $S^*$, or $v$ has
		at most one neighbor in $V(S)$.
	\end{claim}
	
	\begin{proof}[Proof of Claim~\ref{clm:minimal-wheel-sector}]
		Suppose that $S=a_1Cb_1$ satisfies
		$N(v)\cap S^*=\emptyset$ while $v$ has at least two neighbors in
		$V(S)$.  Then $va_1,vb_1\in E(G)$ and $vx\notin E(G)$.  The sector is not
		short, since otherwise $G[\{x,v,a_1,b_1\}]$ is a diamond.
		
		Let $Q=C-S^*$.  Since $v$ is not proper, it has a neighbor in
		$Q^*$.  Starting at $b_1$ along $Q$, let $v^*$ be the first such
		neighbor.  Let $S_1=b_1Ca_2$ be the sector of $Q$ incident with $b_1$,
		and let $S_2=a_2Cb_2$ be the next sector.  Applying
		Claim~\ref{clm:minimal-wheel-gap} to $b_1Qv^*$ gives
		\[
		v^*\in V(b_1Qb_2)\setminus\{b_1,b_2\}.
		\]
		If $v^*\in V(S_2)$, the choice of $v^*$ makes \(D=vb_1Qv^*v\) an induced cycle, possibly a triangle.  The three paths
		\[
		xb_1,\quad xa_2,\quad xa_1v
		\]
		link $x$ to $D$: their ends $b_1,a_2,v$ are distinct, the sector
		$a_1Cb_1$ is long, and $a_1,a_2$ are nonconsecutive spokes.
		This contradicts the linked-cycle principle.
		Thus $v^*\in S_1^*$.  In this case $v$ is linked
		to the hole $xS_1x$ by
		\[
		vb_1,\quad vv^*,\quad va_1x.
		\]
		Here $v$ cannot also be adjacent to $a_2$: it already has the two
		neighbors $b_1,v^*$ in $S_1$, and every sector contains at most two
		neighbors of $v$.  This is again a contradiction, and proves the claim.
	\end{proof}
	
	Put $A=N_C(v)$.  Since $v$ is not proper and has at most two neighbors in
	every sector, $A$ is not contained in one sector: otherwise $|A|\le2$, and
	$v$ would satisfy both conditions in the definition of a proper vertex.
	Let us call two distinct sectors $T,T'$ a \textit{witnessing pair} when
	\[
	A\cap\bigl(V(T)\setminus V(T')\bigr)\ne\emptyset
	\quad\text{and}\quad
	A\cap\bigl(V(T')\setminus V(T)\bigr)\ne\emptyset.
	\]
	To see that a witnessing pair exists, suppose first that $A$ contains a
	vertex in the interior of a sector $T$ and a vertex outside $V(T)$.  A sector
	containing the latter vertex, together with \(T\), forms a witnessing pair.  If instead
	every member of $A$ is a spoke, the fact that $A$ is not contained in one
	sector gives two nonconsecutive spokes; choosing incident sectors directed
	away from the other spoke again gives exclusive witnesses.  Thus a
	witnessing pair exists.
	We distinguish whether a witnessing pair can consist of consecutive sectors.
	
	\medskip
	\noindent\textbf{Case 1: no witnessing pair is consecutive.}
	We first show that $A$ contains no spoke.  Suppose that a spoke $q$ belongs
	to $A$.  Orient $C$, let $T^+=qCq^+$ be the sector immediately following
	$q$, and let $p^+$ be the first vertex of $A$ encountered after $q$.  The
	vertices $q,p^+$ are consecutive in the cyclic order of $A$, so
	Claim~\ref{clm:minimal-wheel-gap}, applied to $qCp^+$, says that this path
	contains at most two spokes.
	
	If $p^+=q^+$, then $A\cap(T^+)^*=\emptyset$ while both ends of $T^+$
	belong to $A$, contrary to Claim~\ref{clm:minimal-wheel-sector}.  If $p^+$
	occurs after $q^+$, the gap bound places $p^+$ in the interior of the sector
	$U^+$ immediately following $T^+$; reaching the other end of $U^+$ would
	already put three spokes on $qCp^+$.  Then $q$ and $p^+$ are exclusive
	attachments to the consecutive sectors $T^+,U^+$, a contradiction.  Hence
	$p^+\in(T^+)^*$.
	
	The reverse orientation similarly puts the first vertex $p^-\in A$
	preceding $q$ in the interior of the sector $T^-$ immediately preceding
	$q$.  If $p^-=p^+$, that vertex would lie in the disjoint interiors of
	$T^-$ and $T^+$.  Otherwise $p^-,p^+$ are exclusive attachments to the
	consecutive sectors $T^-,T^+$.  Both alternatives are impossible.  Thus
	\begin{equation}
		\label{eq:case1-no-spoke-neighbor}
		A\cap N_C(x)=\emptyset.
	\end{equation}
	
	Every member of $A$ now lies in the interior of a unique sector.  Since $A$
	is not contained in one sector, choose two cyclically consecutive vertices
	$u_1,u_2\in A$ whose sectors $S_1,S_2$ are distinct, and orient $u_1Cu_2$
	from $u_1$ to $u_2$.  These sectors are not consecutive, for otherwise
	$u_1,u_2$ would witness them.  The interior of $u_1Cu_2$ is disjoint from
	$A$, so Claim~\ref{clm:minimal-wheel-gap} puts at most two spokes on this
	path.  Distinct nonconsecutive sectors require at least two; hence there are
	exactly two.  Name them $b_1,a_2$ in order, and name the other sector ends so
	that
	\[
	S_1=a_1Cb_1,\quad S_3=b_1Ca_2,\quad S_2=a_2Cb_2.
	\]
	Now $N(v)\cap V(S_3)=\emptyset$.
	Let $s,t$ be the neighbors of $v$ in $S_1$ closest to $a_1,b_1$,
	respectively, and let $z$ be the neighbor of $v$ in $S_2$ closest to
	$a_2$.
	
	The cycle \(D=vtS_1b_1S_3a_2S_2zv\) is a hole.  If $xv\in E(G)$, then $x$ is linked to $D$ by
	$xv,xb_1,xa_2$; hence $xv\notin E(G)$.  If $s=t$, then $x$ is linked to
	$D$ by
	\[
	xb_1,\quad xa_2,\quad xa_1S_1t.
	\]
	Thus $s\ne t$.  Apply Lemma~\ref{lem:external-connector-obstruction}
	with rim $R=D$, distinguished vertices $a=b_1$, $b=a_2$, and
	\[
	A=S_3,\quad B=b_1S_1t-v-zS_2a_2,\quad Q=x-a_1S_1s.
	\]
	The path $Q$ is induced and disjoint from $D$.  Its internal vertices
	are anticomplete to $D$: the choice of $s$ and
	\eqref{eq:case1-no-spoke-neighbor} exclude edges to $v$, while
	chordlessness of $C$ excludes edges to the retained rim segments.
	Here $s\ne t$ separates the $S_1$-parts, and the four distinct spokes
	$a_1,b_1,a_2,b_2$ ensure that $a_1$ has no neighbor on the retained
	$a_2S_2z$ segment.  Moreover,
	\[
	N_D(x)=\{b_1,a_2\},\quad
	v\in N_D(s)\subseteq\{v,t\}\subseteq V(B)\setminus\{a_2\}.
	\]
	Thus every hypothesis of the lemma holds.  Its exceptional outcome
	would require $sb_1\in E(G)$, which is impossible because $t$ lies
	strictly between $s$ and $b_1$ on $S_1$.  Hence $G$ contains an induced
	subdivision of $K_4$ or a prism, a contradiction.
	
	\medskip
	\noindent\textbf{Case 2: a consecutive witnessing pair exists.}
	Write
	\[
	S_1=a_1Cb_1,\quad S_2=b_1Cb_2.
	\]
	Since $W$ has at least four spokes, the complementary $b_2$--$a_1$ arc
	contains another spoke; hence $a_1,b_2$ are distinct and nonadjacent.
	Let $s,t$ be the neighbors of $v$ in $S_1$ closest to $a_1,b_1$,
	respectively, and let $z,w$ be the neighbors of $v$ in $S_2$ closest to
	$b_1,b_2$, respectively.
	
	Suppose first that $t=b_1$.  Claim~\ref{clm:minimal-wheel-sector} gives
	$s\in S_1^*$ and $w\in S_2^*$, and the per-sector neighbor bound
	gives $xv\notin E(G)$.  Diamond-freeness implies that $sb_1$ and $b_1w$
	cannot both be edges.  By symmetry assume $sb_1\notin E(G)$.  Then $v$ is
	linked to the hole $xS_2x$ by
	\[
	vb_1,\quad vw,\quad vsS_1a_1x,
	\]
	a contradiction.  Thus $t\ne b_1$, and symmetrically $z\ne b_1$.
	
	If $vx\in E(G)$, the per-sector neighbor bound gives $s=t$ and $z=w$.
	Diamond-freeness gives $s\ne a_1$ or $w\ne b_2$; otherwise
	$G[\{x,v,a_1,b_2\}]$ is a diamond.  By symmetry assume $s\ne a_1$.
	The three paths
	\[
	vx,\quad vz,\quad vsS_1b_1
	\]
	link $v$ to the induced cycle $xS_2x$, possibly a triangle.
	Here $s\in S_1^*$, $z\ne b_1$, and the per-sector neighbor bound
	and the extreme choices of $s,z$ exclude every extra attachment.
	This is a contradiction.  Therefore $vx\notin E(G)$.
	
	Suppose that $s\ne t$.  If $z\ne b_2$, then $v$ is linked to $xS_1x$ by
	$vs,vt,vzS_2b_1$; if $z=b_2$ but $b_1b_2\notin E(G)$, use instead the
	third path $vzx$.  Consequently
	\[
	z=b_2,\quad b_1b_2\in E(G).
	\]
	If $st\notin E(G)$, then $b_2$ is linked to the hole \(xa_1S_1svtS_1b_1x\) by $b_2b_1,b_2v,b_2x$.  Thus $st\in E(G)$.  The triangles
	$vst$ and $xb_1b_2$, together with
	\[
	sS_1a_1x,\quad tS_1b_1,\quad vb_2,
	\]
	now form an induced prism, a contradiction.  Hence $s=t$.
	By symmetry $z=w$.  Finally, $x$ is linked to the hole
	$vtS_1b_1S_2zv$ by
	\[
	xb_1,\quad xa_1S_1t,\quad xb_2S_2z,
	\]
	again a contradiction.
	
	Both cases are impossible.  Hence the minimum-rim wheel $W$ is proper.
\end{proof}

\subsection{Four-spoke normalization at a prescribed centre}

\begin{lemma}\label{lem:fixed-centre-four-spoke}\label{lem:n4}
	Let $G$ be an \(\{\mathrm{ISK}_4,\text{diamond},
	\text{prism}\}\)-free graph.
	Suppose that $W=(C,x)$ has exactly four spokes
	$a_1,a_2,a_3,a_4$, in this cyclic order on $C$, and that \(a_1a_2,\ a_3a_4\in E(C)\).
	Then $G$ contains a proper wheel centred at $x$, with exactly the same four
	spokes, in which $a_1a_2$ and $a_3a_4$ are still opposite short sectors.
\end{lemma}

\begin{proof}
	\setcounter{claim}{0}
	Let $\mathcal R$ be the family of oriented holes $R$ satisfying
	\begin{enumerate}
		\item $x\notin V(R)$ and $N_R(x)=\{a_1,a_2,a_3,a_4\}$;
		\item the orientation encounters $a_1,a_2,a_3,a_4$ in this cyclic order;
		\item $a_1a_2,a_3a_4\in E(R)$.
	\end{enumerate}
	The original oriented rim $C$ belongs to $\mathcal R$.  For
	$R\in\mathcal R$, let
	\[
	P_R=a_2Ra_3,\quad Q_R=a_4Ra_1
	\]
	be the two sectors other than the specified short sectors.  Neither is
	short.  For if $a_2a_3\in E(R)$, then
	$G[\{x,a_1,a_2,a_3\}]$ is a diamond; the argument for $Q_R$ is symmetric.
	
	We prove a sector-local statement uniformly over this family.
	
	\begin{claim}
		\label{clm:four-spoke-local}
		For every $R\in\mathcal R$ and every $y\notin V(R)\cup\{x\}$, the set
		$N_R(y)$ is contained in one sector of $(R,x)$.
	\end{claim}
	
	\begin{proof}[Proof of Claim~\ref{clm:four-spoke-local}]
		Fix an arbitrary $R\in\mathcal R$ and, for this proof only, write
		\[
		C:=R,\quad P:=P_R,\quad Q:=Q_R.
		\]
		We distinguish whether $y$ is adjacent to the centre.
		
		First suppose that $xy\in E(G)$.
		Each spoke already has a neighbor among the other spokes.  Since
		$y\in N(x)$, Lemma~\ref{lem:sparse-neighborhood} implies that $y$ is
		adjacent to none of $a_1,a_2,a_3,a_4$.
		
		If $y$ sees both $P^*$ and $Q^*$, choose
		$p\in N(y)\cap P^*$ closest to $a_3$, and choose
		$q\in N(y)\cap Q^*$ closest to $a_4$.  The triangle $xa_3a_4$ and the
		three paths
		\[
		xy,\quad ypPa_3,\quad yqQa_4
		\]
		form an induced subdivision of $K_4$.  The extremal choices of $p,q$,
		together with $ya_3,ya_4\notin E(G)$, give
		\[
		N(y)\cap V(pPa_3)=\{p\},\quad
		N(y)\cap V(qQa_4)=\{q\}.
		\]
		The sector interiors are anticomplete to $x$, and $C$ is chordless.  Thus $y$
		cannot see both long sectors.  Since it sees no spoke, all of $N_C(y)$ is
		contained in one sector.
		
		Now suppose that $xy\notin E(G)$.
		Assume for a contradiction that both
		\[
		B_P=N(y)\cap V(P),\quad B_Q=N(y)\cap V(Q)
		\]
		are nonempty.
		
		Suppose first that $|B_P|\ge2$.  In the order from $a_2$ to $a_3$, let
		$p$ and $p'$ be the first and last vertices of $B_P$.
		
		If $pp'\notin E(G)$, reverse the cyclic labelling
		$(a_1,a_2,a_3,a_4)$ to $(a_4,a_3,a_2,a_1)$ if necessary so that the
		member $q$ of $B_Q$ closest to $a_4$ is different from $a_1$.  This
		choice is possible: the only obstruction in one orientation is
		$B_Q=\{a_1\}$, which disappears after reversal.  The triangle
		$xa_3a_4$, with fourth branch vertex $y$ and paths
		\begin{equation}
			\label{eq:n4-nonadjacent-extremes}
			ypPa_2x,\quad yp'Pa_3,\quad yqQa_4,
		\end{equation}
		is an induced subdivision of $K_4$.  The extreme choices exclude extra
		edges at $y$, the nonedge $pp'$ separates the two $P$-tails, and
		$q\ne a_1$ excludes the possible chord $a_1a_2$.
		
		Suppose instead that $pp'\in E(G)$.  Then $pp'\in E(C)$ and
		$B_P=\{p,p'\}$.  By the same reversal, assume $p'\ne a_3$; otherwise
		reverse, using $p\ne a_2$, which follows because the sector $P$ is
		long.  Choose $q\in B_Q$ closest to $a_4$.  If $q\ne a_1$, the
		triangles $xa_3a_4$ and $ypp'$, joined by
		\begin{equation}
			\label{eq:n4-first-prism}
			xa_2Pp,\quad a_3Pp',\quad a_4Qqy,
		\end{equation}
		form an induced prism.  If $q=a_1$, then $B_Q=\{a_1\}$.  The case
		$p=a_2$ gives the diamond $G[\{x,y,a_1,a_2\}]$; otherwise the
		triangles $xa_1a_2$ and $ypp'$, joined by
		\[
		xa_3Pp',\quad a_2Pp,\quad a_1y,
		\]
		form an induced prism.  In both prism constructions the two triangles
		are disjoint by the displayed endpoint inequalities, and the extreme
		choices and chordlessness of $C$ exclude every other cross-edge.
		Thus $|B_P|\ge2$ is impossible; by symmetry so is $|B_Q|\ge2$.
		
		Consequently $B_P=\{p\}$ and $B_Q=\{q\}$.  The pairs
		$(p,q)=(a_2,a_1)$ and $(a_3,a_4)$ give diamonds with $x,y$.
		Excluding them, the same reversal lets us assume simultaneously
		$p\ne a_3$ and $q\ne a_1$: failure of these inequalities in both
		orientations would give one of the two excluded pairs.  The triangle
		$xa_3a_4$, with fourth branch vertex $p$ and paths
		\[
		pPa_2x,\quad pPa_3,\quad pyqQa_4,
		\]
		is now an induced subdivision of $K_4$.  The singleton attachment
		sets and the two endpoint inequalities exclude additional edges.
		
		Thus $B_P$ and $B_Q$ cannot both be nonempty.  Hence all neighbors of
		$y$ on $R$ lie in a single sector of $(R,x)$.  Since $R\in\mathcal R$ was
		arbitrary, the claim follows.
	\end{proof}
	
	Choose $R\in\mathcal R$ with $|V(R)|$ minimum, and put $P=P_R$, $Q=Q_R$.
	If $(R,x)$ is not proper, Claim~\ref{clm:four-spoke-local}, applied to this
	particular $R$, yields a vertex $y\notin V(R)\cup\{x\}$ whose $R$-neighbors
	all lie in one sector.  Its
	only possible violation of properness is
	\begin{equation}
		\label{eq:n4-improper-vertex}
		xy\notin E(G),\quad |N_R(y)|\ge3.
	\end{equation}
	Since a short sector contains only two vertices, all these neighbors lie
	in $P$ or all lie in $Q$.  By symmetry, assume that they lie in $P$.
	
	Let $p,q$ be the two outermost neighbors of $y$ on $P$, with
	$pRq\subseteq P$, and replace $pRq$ by the two-edge path $pyq$.  Denote the
	resulting cycle by $R'$.  By the choice of $p,q$, the vertex $y$ has no
	neighbor other than $p,q$ on the retained $qRp$-arc.  Since $R$ is a hole,
	$R'$ is a hole.  Lemma~\ref{lem:sparse-neighborhood} and
	\eqref{eq:n4-improper-vertex} give $|E(pRq)|\ge3$, so this replacement
	strictly shortens the rim.
	
	The replacement takes place inside a long sector and $xy\notin E(G)$.
	Consequently all four old spokes remain, no new spoke is created, their
	cyclic order is unchanged, and the two prescribed short edges
	$a_1a_2,a_3a_4$ remain.  Hence $R'\in\mathcal R$ but
	$|V(R')|<|V(R)|$, a contradiction.  The wheel $(R,x)$ is proper.
\end{proof}

\paragraph{Normalization and global minimality.}
We write \textbf{(N4)} for the fixed-centre four-spoke normalization of
Lemma~\ref{lem:fixed-centre-four-spoke}.  A proper wheel $W$ in $G$ is
called \textit{globally spoke-minimal} if
\begin{equation*}
	W \text{ has the minimum number of spokes among all proper wheels of }G.
	\tag{GM}\label{eq:global-minimal-wheel}
\end{equation*}
The centre is allowed to vary.  This quantifier is essential: several
replacement wheels have a centre different from that of the original
wheel.  Minimising only among wheels with a prescribed centre does not
exclude these replacements.  Whenever a globally spoke-minimal wheel has
at least five spokes, any exact-four output to which \textbf{(N4)} applies
therefore gives a contradiction, regardless of its centre.

\section{Local attachment lemmas}
\label{sec:local-attachments}

Throughout this section,
\[
G \text{ is }\{\mathrm{ISK}_4, \text{diamond}, \text{prism}\}
\text{-free}.
\]
All local arguments needed later are proved here; no portion of a cited proof is
silently imported.  The terminology of wheels, spokes, sectors, proper
vertices, and proper wheels is that fixed in Section~\ref{sec:preliminaries}.
The abbreviations \textbf{(N4)} and~\eqref{eq:global-minimal-wheel}
refer to the normalization and global minimality conventions fixed at the
end of Section~\ref{sec:local-wheels}.

\subsection{An exact-four replacement}

\begin{lemma}\label{lem:modified-lsw-32}
	Let $W=(C,x)$ be a proper wheel, let $u,v\in N(x)$ with $uv\in E(G)$,
	and let $W'=(C',x)$ be an exact-four wheel containing $u$ and $v$.
	Let $S$ be the sector of $W'$ that contains $u$ but not $v$, and suppose
	that
	\begin{equation}
		V(C')\setminus V(S)\subseteq V(C).
		\label{eq:32-containment}
	\end{equation}
	Then $G$ contains a proper exact-four wheel.  Its center need not be $x$.
	In particular, if $W$ satisfies~\eqref{eq:global-minimal-wheel} and has at
	least five spokes, the conclusion is impossible.
\end{lemma}

\begin{proof}
	With $W,u,v$ fixed, consider every pair $(\widehat W,\widehat S)$ in which
	$\widehat W=(\widehat C,x)$ is an exact-four wheel containing $u,v$,
	$\widehat S$ is its sector containing $u$ but not $v$, and
	$V(\widehat C)\setminus V(\widehat S)\subseteq V(C)$.  Choose such a pair
	with $|V(\widehat C)|$ minimum and rename it $(W',S)$, where
	$W'=(C',x)$.  Since $uv\in E(G)$ and $C'$ is a hole, $uv$ is a short
	sector.  Write the four spokes in cyclic order as \(w_1,u,v,w_2\).
	If $w_1w_2\in E(G)$, then $uv$ and $w_1w_2$ are opposite short
	sectors, so (N4) gives the required proper exact-four wheel.
	Hence assume $w_1w_2\notin E(G)$, and put
	\[
	T_1=w_1C'u,\quad T_2=vC'w_2,\quad T_3=w_2C'w_1.
	\]
	Thus $T_1=S$, and all three $T_i$ are long: this is immediate for $T_3$,
	while $T_1,T_2$ are long by Lemma~\ref{lem:sparse-neighborhood}, since
	$uv\in E(G[N(x)])$.  If $W'$ is proper, there is nothing to prove.  Suppose
	that $y\notin V(W')$ is not proper for $W'$.
	
	We first show that
	\begin{equation}
		N(y)\cap V(T_1)\ne\emptyset.
		\label{eq:b2-y-meets-t1}
	\end{equation}
	Let $r_0$ be the neighbor of $w_1$ on $T_3$, and put
	$P=vT_2w_2T_3r_0$.  By~\eqref{eq:32-containment}, $V(P)\subseteq V(C)$;
	since $C$ is a hole, every edge of $C'$ with both ends in $V(C)$ is a rim
	edge of $C$.  Hence $P$ is a subpath of $C$.  The only
	vertices of $P$ adjacent to $x$ are $v,w_2$.  Suppose that $y$ has no neighbor on
	$T_1$.  If $y\notin V(C)$, properness of $y$ for $W$ puts
	$N(y)\cap V(P)$ in one sector of $W$.  The two parts of $P$ on opposite
	sides of the spoke $w_2$ lie in distinct sectors, so
	$N(y)\cap V(C')$ lies in one sector of $W'$; properness for $W$ also
	gives $xy\in E(G)$ whenever this set has more than two vertices.  Thus
	$y$ is proper for $W'$, a contradiction.  If $y\in V(C)$,
	chordlessness of $C$ gives $|N(y)\cap V(P)|\leq 2$.  With at most one
	such neighbor, $y$ is again proper for $W'$.  With two, these are the two
	cyclic neighbors of $y$ on $C$.  Since $P\subseteq C-y$ contains both,
	it follows that $P=C-y$.  Then $W$ has at most the three spokes
	$v,w_2,y$, contrary to the fact that every wheel in $G$ has at least
	four spokes.  This proves~\eqref{eq:b2-y-meets-t1}.
	
	Let $y_1$ be the neighbor of $y$ on $T_1$ closest to $u$.  Delete the edge $uv$ from $C'$ and orient the resulting path as \(Q=vT_2w_2T_3w_1T_1u\).
	Let $y'_1$ be the first neighbor of $y$ encountered on $Q$ from $v$.
	Whenever below we close a displayed subpath of $C'$ through $y$, the
	result is a hole: the choices of the extreme neighbors exclude further
	edges from $y$ to the retained subpath, and the chordlessness of $C'$
	excludes all remaining chords.
	
	Suppose first that $xy\in E(G)$.  By
	Lemma~\ref{lem:sparse-neighborhood}, $|N(x)\cap N(y)|\leq1$;
	also $yu,yv\notin E(G)$, since $uv$ is already an edge in $G[N(x)]$.
	Put $p=y_1$ and $r=y'_1$.  Since $y$ is not proper for
	$W'$ but $xy\in E(G)$, it has a neighbor outside $T_1$; hence
	$r\in (V(T_2)\cup V(T_3))\setminus\{w_1\}$.
	
	If $r\in V(T_2)$, consider the hole \(D=y-pT_1u-vT_2r-y\).
	Its spokes with centre $x$ are $y,u,v$, together with $w_1$ if $p=w_1$
	and $w_2$ if $r=w_2$.  The common-neighbor bound permits at most one
	of these extra spokes.  With none we obtain a forbidden three-spoke wheel;
	with one, the opposite short sectors are $uv,yw_i$, so (N4) supplies the
	required proper exact-four wheel.
	
	Thus assume $r\in T_3^*$.  We also have $p\ne w_1$: otherwise
	$w_1$ is the only neighbor of $y$ on $T_1$ and all its rim neighbors
	lie in $T_3$, making $y$ proper for $W'$.  The cycle \(D=y-pT_1u-vT_2w_2T_3r-y\) is therefore an exact-four wheel rim with centre $x$ and spokes
	$y,u,v,w_2$.  Its sector containing $u$ but not $v$ is $y-pT_1u$,
	and all vertices outside that sector lie on $C$.  Thus it satisfies
	\eqref{eq:32-containment}.  The replaced arc $pT_1w_1T_3r$ has at least
	two edges; if it has at least three, $D$ is shorter than $C'$, a
	contradiction.  In the remaining case this arc is $p-w_1-r$.
	If $yw_1\in E(G)$, then $\{y,p,w_1,r\}$ induces a diamond.  Otherwise
	$y-p-w_1-r-y$ is a hole.  The choice of $r$ as the first neighbor
	from $w_2$, together with $rw_1\in E(C')$, makes $r$ the only neighbor
	of $y$ on $T_3$.  Hence $x$ is linked to this hole by
	\[
	xy,\quad xw_1,\quad xw_2T_3r.
	\]
	The sector interiors are anticomplete to $x$, and chordlessness of
	$C'$ and the stated neighborhood of $y$ on $T_3$ exclude all extra
	edges.  This final contradiction proves that
	\begin{equation}
		xy\notin E(G).\label{eq:b2-xy-nonedge}
	\end{equation}
	
	Suppose that $y_1=w_1$.  Then $w_1$ is the only neighbor of $y$ on
	$T_1$.  If $y'_1\in V(T_2)\setminus\{w_2\}$, the hole \(y-w_1T_1u-vT_2y'_1-y\) has exactly the three neighbors $w_1,u,v$ of $x$, a contradiction.
	Hence $y'_1\in V(T_3)$, and every neighbor of $y$ on $C'$ lies on
	$T_3$.  Since $T_3-w_1\subseteq C$ and $xy\notin E(G)$, properness for $W$
	(or chordlessness of $C$ if $y\in V(C)$) gives $|N_{T_3}(y)|\le3$.
	Equality makes $(xT_3x,y)$ a forbidden three-spoke wheel, while at most
	two neighbors make $y$ proper for $W'$, again a contradiction.
	Thus
	\begin{equation}
		y_1\ne w_1.
		\label{eq:b2-y1-not-w1}
	\end{equation}
	
	Suppose next that $y'_1\in V(T_2)$.  If $yv\in E(G)$, then
	$yu\notin E(G)$, since otherwise $\{x,y,u,v\}$ induces a diamond.
	Apply Lemma~\ref{lem:external-connector-obstruction} with
	\[
	R=x-uT_1w_1-x,\quad a=u,\quad b=x,\quad
	A=ux,\quad B=uT_1w_1-x,\quad Q=v-y.
	\]
	The path $Q$ lies outside the hole $R$, and $N_R(v)=\{u,x\}$.
	By \eqref{eq:b2-y-meets-t1}, $y$ has a neighbor on $T_1$; since
	$yu,yx\notin E(G)$, all its $R$-neighbors lie in $B^*$.
	The exceptional outcome would require $yu\in E(G)$, so every outcome
	is impossible.
	
	If $yv\notin E(G)$, apply the same lemma with
	\[
	R=x-vT_2w_2-x,\quad a=v,\quad b=x,\quad
	A=vx,\quad B=vT_2w_2-x,\quad Q=uT_1y_1-y.
	\]
	Here $N_R(u)=\{v,x\}$.  By \eqref{eq:b2-y1-not-w1}, the internal
	vertices of $Q$ lie in $T_1^*$ and are anticomplete to $R$; the
	extreme choice of $y_1$ makes $Q$ induced.  Moreover,
	$y'_1\in T_2-\{v\}$ gives a neighbor of $y$ in $B^*$, and
	$yv,yx\notin E(G)$ put all its $R$-neighbors there.
	The exceptional outcome requires $yv\in E(G)$, again impossible.  Hence
	\begin{equation}
		y'_1\notin V(T_2).
		\label{eq:b2-y1prime-not-t2}
	\end{equation}
	
	If $y'_1\in V(T_3)\setminus\{w_1\}$, then \(y-y_1T_1u-vT_2w_2T_3y'_1-y\) is a hole on which $x$ has exactly the three neighbors $u,v,w_2$,
	again an induced subdivision of $K_4$.  It follows from the definition
	of $y'_1$ and~\eqref{eq:b2-y1prime-not-t2} that \(N(y)\cap V(C')\subseteq V(T_1)\).
	Since $xy\notin E(G)$ and $y$ is not proper for $W'$, it has at least
	three neighbors on $T_1$.  Replace the segment
	$y'_1T_1y_1$ of $C'$ by $y'_1-y-y_1$.  Explicitly, the resulting hole is \(D=u-vT_2w_2T_3w_1T_1y'_1-y-y_1T_1u\).
	The wheel $(D,x)$ is exact-four, with spokes $u,v,w_2,w_1$, and its
	sector containing $u$ but not $v$ is
	$w_1T_1y'_1-y-y_1T_1u$; hence it still satisfies
	\eqref{eq:32-containment}.  By the shortcut assertion of
	Lemma~\ref{lem:sparse-neighborhood}, the old segment $y'_1T_1y_1$ has
	at least three edges, so the replacement strictly shortens $C'$, a
	contradiction.
	
	No nonproper vertex $y$ exists.  Hence $W'$ is proper unless one of the
	preceding N4 configurations has already supplied a proper exact-four
	wheel, and the lemma follows.  Under~\eqref{eq:global-minimal-wheel}, a
	proper exact-four wheel, with any centre, contradicts the global spoke-minimality of \(W\) when \(\sigma(W) \geq 5\), proving the final assertion.
\end{proof}

\subsection{Replacing a rim sector}

\begin{lemma}\label{lem:modified-lsw-33}
	Let $W=(C,x)$ be a proper wheel, let $u,v\in N(x)$ with $uv\in E(G)$,
	and suppose that a wheel $W'=(C',x)$ contains $u,v$ and satisfies \(V(C')\setminus\{u,v\}\subseteq V(C)\).
	Then there is a proper wheel with center $x$ and with no more spokes than
	$W'$.
\end{lemma}

\begin{proof}
	Write \(\sigma(R,x)=|N(x)\cap V(R)|\) for the number of spokes of a wheel with rim $R$.  Put $k=\sigma(C',x)$
	and suppose, for a contradiction, that no proper wheel with centre $x$
	has at most $k$ spokes.  Since $W=(C,x)$ is proper, we have
	\begin{equation}
		\sigma(C,x)>k.                                      \label{eq:B3-spoke-gap}
	\end{equation}
	If both $u$ and $v$ belonged to $C$, then
	$V(C')\subseteq V(C)$.  The graph induced by $V(C)$ contains only the
	edges of the single cycle $C$, and a proper vertex subset of a cycle cannot
	contain a cycle.  Consequently $C'=C$, and $W'=W$ would already be proper.
	Thus at least one
	of $u,v$ is a new spoke outside $C$.
	
	We first record the elementary rim-transfer observation that will be used
	several times.  If
	\[
	t\notin V(C')\cup\{x\},\quad
	N(t)\cap\{u,v\}=\emptyset,
	\]
	and either $t\in V(C)$ or $t$ is proper for $W$, then $t$ is proper for
	$W'$.  To see this,
	note first that $C'-\{u,v\}$ is a path contained in $C$ (an edge between
	two of its vertices is necessarily an edge of the hole $C$).  If
	$t\notin V(C)$, all the neighbors of $t$ on $C$ lie in one sector $S$
	of $W$.  The trace of $S$ on $C'-\{u,v\}$ lies in one sector of $W'$,
	unless the complementary subpath of $C$ is contained in $S$.  In the
	latter case that complementary subpath has no spoke of $W$ in its
	interior.  Thus every old spoke outside $\{u,v\}$ occurs on
	$C'-\{u,v\}$, every member of $\{u,v\}\cap V(C)$ remains a spoke on
	$C'$, and at least one of $u,v$ is a new spoke not on $C$.  Hence
	$\sigma(C',x)\geq\sigma(C,x)$, contrary to
	\eqref{eq:B3-spoke-gap}.  Thus $N_{C'}(t)$ lies in one sector of $W'$.
	If it has more than two members, properness for $W$ gives $tx\in E(G)$,
	so the second condition in the definition of a proper vertex also holds.
	If $t\in V(C)\setminus V(C')$, chordlessness of $C$ gives at most two
	neighbors of $t$ on $C'$.  They can fail to lie in one sector of $W'$
	only when $V(C)\setminus V(C')=\{t\}$; but then passing from $C$ to
	$C'$ removes at most the one spoke $t$ and introduces at least one new
	spoke among $u,v$, again giving
	$\sigma(C',x)\geq\sigma(C,x)$.  This proves the observation.
	
	Let $T_1,T_2$ be the sectors of $W'$ incident, respectively, with $u$
	and $v$ but different from the short sector $uv$.  Let $w_i$ be the end
	of $T_i$ different from $u$ or $v$, respectively, and orient $C'$ so that \(w_1,u,v,w_2\) occur in this cyclic order.  If $W'$ is proper, it is the required wheel,
	so choose a vertex $y\notin V(C')\cup\{x\}$ that is not proper for $W'$.
	The rim-transfer observation shows that $y$ is adjacent to $u$ or $v$;
	after reversing the notation if necessary, assume \(yu\in E(G)\).
	The edge $xv$ in $G[N(u)]$ and Lemma~\ref{lem:sparse-neighborhood} give
	\begin{equation}
		xy,yv\notin E(G).                                  \label{eq:B3-y-nonedges}
	\end{equation}
	
	We next determine the remaining neighbors of $y$ on $C'$.  If
	$y\notin V(C)$, properness for $W$ and $xy\notin E(G)$ give at most two
	neighbors on $C'-\{u\}$; if $y\in V(C)\setminus V(C')$, the same bound
	follows from the chordlessness of $C$.  If the bound were attained, then,
	together with $u$, the vertex $y$ would have exactly three neighbors on
	the hole $C'$, and $(C',y)$ would be a three-spoke wheel, hence an induced
	subdivision of $K_4$.  On the other hand, if \(y\) had no neighbor in
	$C'-\{u\}$, then it would be proper for $W'$.  Consequently there is a
	unique vertex $y'\in V(C')\setminus\{u\}$ adjacent to $y$:
	\begin{equation}
		N(y)\cap V(C')=\{u,y'\}.                            \label{eq:B3-y-neighborhood}
	\end{equation}
	The two vertices $u,y'$ do not lie in a common sector of $W'$.  Among all
	choices of a nonproper vertex and, if necessary, the symmetric choice of
	$u$ or $v$, choose $y$ so that \(P_{u,y}=uC'y'-y\), where the $u$--$y'$ subpath avoids $v$, has as few vertices as possible.
	
	We now locate $y'$.  If $y'\in V(T_2)$, then \(y-u-vT_2y'-y\) is a hole and $x$ is linked to it by
	\[
	xv,\quad xu,\quad xw_2T_2y',
	\]
	contrary to the linked-cycle principle.  Let $T_3$ be the sector incident
	with $w_1$ on the side opposite $T_1$, and let $w_3$ be its other end
	(possibly $w_3=w_2$).  If $y'\in V(T_3)$, then $y'\neq w_1$ because
	$u,y'$ do not lie in a common sector, and $x$ is linked to the hole \(y-uT_1w_1T_3y'-y\) by
	\[
	xu,\quad xw_1,\quad xw_3T_3y'.
	\]
	Thus neither possibility occurs.
	In particular, \(w_3\neq w_2\), since otherwise the excluded sectors cover \(C'\).
	Let $Q$ be the $w_3$--$w_2$ subpath of
	$C'$ whose interior avoids $u$.  Then $y'\in Q^*$.
	
	Put
	\begin{equation}
		D=y-uT_1w_1T_3w_3Qy'-y                            \label{eq:B3-D}
	\end{equation}
	and $W_1=(D,x)$.  Equation~\eqref{eq:B3-y-neighborhood} and the
	chordlessness of $C'$ show that $D$ is a hole.
	It retains the three distinct spokes \(u, w_1, w_3\), so \((D, x)\) is a wheel.
	Its spokes are precisely the spokes of $W'$ on the $u$--$y'$ subpath avoiding $v$: the vertex $y$
	is not a spoke, while at least the spoke $v$ has been deleted.  Hence
	\begin{equation}
		\sigma(D,x)<k.                           \label{eq:B3-fewer-spokes}
	\end{equation}
	If $W_1$ is proper, it is the required wheel.  We may therefore choose
	$z\notin V(D)\cup\{x\}$ that is not proper for $W_1$.
	No vertex of $V(C')\setminus V(D)$ can serve as $z$.
	The deleted $u$--$y'$ arc contains the distinct internal vertices
	$v,w_2$, since $y'\notin V(T_2)$.  By chordlessness of $C'$, each
	deleted rim vertex has at most one neighbor on the retained arc;
	none is adjacent to $y$ by~\eqref{eq:B3-y-neighborhood}.
	It is therefore proper for $W_1$.
	Thus
	\begin{equation}
		z\notin V(C')\setminus V(D).             \label{eq:B3-z-not-deleted}
	\end{equation}
	Let $S_y$ denote the sector of $W_1$ that contains the whole two-edge
	subpath $u-y-y'$, and let $h$ be its end different from $u$.
	
	Suppose that $zu\in E(G)$.  The triangle $xuv$ and the edge $uz$
	force $zx,zv\notin E(G)$: exactly one of these edges would give a
	diamond on $\{x,u,v,z\}$, and both would give a $K_4$.  We have \(|N(z)\cap(V(C')\setminus\{u\})|\leq2\).
	Indeed, $zv\notin E(G)$ and $V(C')\setminus\{u,v\}\subseteq V(C)$;
	if $z\notin V(C)$, the bound follows from properness for $W$ and
	$zx\notin E(G)$, while if $z\in V(C)$ it follows from chordlessness
	of $C$.  Equality would make $(C',z)$ a three-spoke wheel.  Therefore
	$|N_{C'}(z)|\leq2$.
	
	If $zy\in E(G)$, then $N_D(z)$ contains $u,y$ and has at most three
	vertices.  Since $u,y$ lie in the sector $S_y$, two neighbors would
	make $z$ proper for $W_1$, whereas three would make $(D,z)$ a
	three-spoke wheel.  Thus $zy\notin E(G)$.
	Now $N_D(z)$ has at most two vertices and contains $u$.  The two
	sectors of $W_1$ incident with $u$ are $T_1$ and $S_y$.  Because $z$
	is not proper for $W_1$, it follows that
	\[
	N_D(z)=\{u,q\},\quad q\notin V(T_1)\cup V(S_y).
	\]
	The vertex $q$ lies in the interior of the $u$--$y'$ subpath of $C'$
	avoiding $v$.  Moreover, $u,q$ do not lie in a common sector of $W'$,
	whose two sectors at $u$ are $T_1$ and $uv$.  Hence $z$ is an admissible
	nonproper vertex for $W'$, and $uC'q-z$ is shorter than $P_{u,y}$,
	contrary to its choice.  We conclude that
	\begin{equation}
		zu\notin E(G).\label{eq:B3-zu-zero}
	\end{equation}
	
	Suppose that $zv\in E(G)$.  By~\eqref{eq:B3-zu-zero}, the edge $zx$
	would make $\{x,u,v,z\}$ a diamond, so $zx\notin E(G)$.  Properness
	for $W$, or chordlessness of $C$ when $z\in V(C)$, gives
	$|N(z)\cap(V(C')\setminus\{v\})|\leq2$.  Equality would make
	$(C',z)$ a three-spoke wheel.  If this set were empty, or its unique
	member $z'$ did not lie on $D$, then $N_D(z)\subseteq\{y\}$ and $z$
	would be proper for $W_1$.  Thus $z'$ exists and lies on $D$.  If
	$zy\notin E(G)$, again $|N_D(z)|=1$, a contradiction.  Otherwise $z$
	is linked to $D$ by $zy$, $zz'$, and $z-v-u$.  These have distinct
	ends because $zu\notin E(G)$ and $z'\in V(C')$; the only neighbor
	of $v$ on $D$ is $u$, by the location of $y'$ and chordlessness of
	$C'$, while $vy\notin E(G)$.  Thus $zv\notin E(G)$.
	
	The rim-transfer observation now makes $z$ proper for $W'$.  Put
	$P=uC'y'$, where $P$ avoids $v$, and $A=N_{C'}(z)$, and choose a
	sector $S$ of $W'$ containing $A$.  The following rim-trace facts will
	be used.  Every sector of $W'$ has its intersection with $P$ in one
	sector of $W_1$.  Moreover, if a sector is not wholly contained in
	$P$, its intersection with $P$ lies in $S_y$: it can meet $P$ only at
	$u$, or in the terminal path $hPy'$, where $h$ is the last spoke on
	$P$.  This includes $h=y'$ when $y'$ is a spoke.  These statements
	follow by traversing $P$ from the spoke $u$: each successive pair of
	retained spokes determines an unchanged sector, and the final
	segment $hPy'$ belongs to the new sector $S_y$.
	
	If $zy\notin E(G)$, the first rim-trace fact and properness for $W'$
	make $z$ proper for $W_1$: no rim neighbor is added, and any retained
	set of more than two neighbors still forces $zx\in E(G)$.  Hence
	$zy\in E(G)$ and $N_D(z)=(A\cap V(P))\cup\{y\}$.
	
	We claim that $A\cap V(P)$ is not contained in $V(S_y)$.  Otherwise
	all neighbors of $z$ on $D$ lie in $S_y$, and its failure to be proper
	would force $zx\notin E(G)$ and $|N_D(z)|\geq3$.  Properness for $W'$
	then gives $|A|\leq2$, so $|N_D(z)|=3$, an excluded three-spoke wheel.
	The claim and the second rim-trace fact show that $V(S)\subseteq V(P)$;
	in particular, $A\subseteq V(P)$ and $A\ne\emptyset$.
	
	If $A=\{a\}$, the claim gives $a\notin V(S_y)$, so $a,u,y'$ are
	distinct.  The paths $yu$, $yy'$, and $y-z-a$ link $y$ to $C'$:
	$N_{C'}(y)=\{u,y'\}$ by~\eqref{eq:B3-y-neighborhood}, and
	$N_{C'}(z)=\{a\}$.  If $|A|=2$, then $z$ has exactly three neighbors
	on $D$.  Both cases are impossible, and therefore $|A|\geq3$.
	
	Let $z_1,z_2$ be the first and last members of $A$ on $S$, and put
	$R=z_1Sz_2$.  There is another member of $A$ in $R^*$, so
	$z_1z_2\notin E(G)$: $R$ lies in one sector of a wheel, and its
	complementary rim arc contains other spokes.  Since $R\subseteq P$,
	neither endpoint $u,y'$ of $P$ lies in $R^*$.  Replace $R$ in
	$C'$ by $z_1-z-z_2$, obtaining $C_z$.  The complementary
	$z_1$--$z_2$ arc is induced, and the extremal choices leave $z$ with
	no neighbor in its interior; thus $C_z$ is a hole.  It retains
	$u,y'$, and $y$ has exactly the three neighbors $u,y',z$ on it,
	by~\eqref{eq:B3-y-neighborhood}.  This final three-spoke wheel is a
	contradiction, proving the lemma.
\end{proof}

\subsection{The adjacent-pair lemma}

\begin{lemma}\label{lem:modified-lsw-34}
	Let $W=(C,x)$ satisfy~\eqref{eq:global-minimal-wheel} and have at least
	five spokes.  If \(u,v\in N(x)\setminus V(C)\) and \(uv\in E(G)\), and both $u$ and $v$ have a neighbor on $C$, then \(N_C(u)\cup N_C(v)\) is contained in one sector of $W$.
\end{lemma}

\begin{proof}
	By Lemma~\ref{lem:sparse-neighborhood},
	\begin{equation}
		\{u,v\}\text{ is anticomplete to }N_C(x),
		\quad N(u)\cap N(v)=\{x\}.
		\label{eq:34-sparse-pair}
	\end{equation}
	For the first assertion, a spoke adjacent to $u$, for example, would give
	$u$ two neighbors in $G[N(x)]$, namely $v$ and that spoke.  For the
	second, a further common neighbor $z$ of $u,v$ would make
	$G[\{u,v,x,z\}]$ a diamond when $xz\notin E(G)$ and a $K_4$ when
	$xz\in E(G)$.
	
	Suppose for a contradiction that $N_C(u)\cup N_C(v)$ is not contained in
	one sector.  Properness of $W$ gives sectors
	\[
	S_1=a_1Cb_1,\quad S_2=a_2Cb_2
	\]
	such that $N_C(u)\subseteq V(S_1)$ and
	$N_C(v)\subseteq V(S_2)$.  Necessarily $S_1\ne S_2$.  Orient the
	sectors so that $a_1,b_1,a_2,b_2$ occur in this cyclic order, allowing
	$b_1=a_2$ or $b_2=a_1$ when the sectors are consecutive.
	By~\eqref{eq:34-sparse-pair}, every attachment of $u$ or $v$ lies
	in the interior of its sector.  Choose
	\begin{align*}
		u^+&\in N_C(u) &&\text{closest to $b_1$ on $S_1$},&
		u^-&\in N_C(u) &&\text{closest to $a_1$ on $S_1$},\\
		v^-&\in N_C(v) &&\text{closest to $a_2$ on $S_2$},&
		v^+&\in N_C(v) &&\text{closest to $b_2$ on $S_2$}.
	\end{align*}
	
	Let $R_1$ be the $u^+$--$v^-$ subpath of $C$ that contains the
	$b_1$--$a_2$ arc between the two sectors, and let $R_2$ be the
	$u^-$--$v^+$ subpath that contains the other inter-sector arc, from
	$a_1$ to $b_2$.  Then
	\begin{equation}
		D_1=u-u^+R_1v^--v-u,
		\quad
		D_2=u-u^-R_2v^+-v-u
		\label{eq:34-two-holes}
	\end{equation}
	are holes.  Indeed, the four extremal choices exclude further edges from
	$u$ or $v$ to the interiors of the displayed paths; all remaining
	vertices of those paths lie on the hole $C$; and
	\eqref{eq:34-sparse-pair} excludes the two possible cross-edges from an
	end chosen for one of $u,v$ to the other vertex.
	
	For $i\in\{1,2\}$, let $r_i$ be the number of spokes of $W$ on $R_i$.
	Each inter-sector arc contains at least one old spoke, including
	when its two ends coincide.  Thus $(D_i,x)$ is a wheel, with precisely
	$u,v$ and the old spokes on $R_i$ as its spokes, and \(\sigma(D_i,x)=r_i+2\).
	The interiors of $S_1$ and $S_2$ contain no old spoke, and the two
	inter-sector arcs partition the old spokes of $W$.  Consequently
	\[
	r_1+r_2=\sigma(W),
	\quad
	\sigma(D_1,x)+\sigma(D_2,x)=\sigma(W)+4<2\sigma(W).
	\]
	One of the two wheels therefore has fewer spokes than $W$.
	Lemma~\ref{lem:modified-lsw-33}, applied to that wheel, produces a proper
	wheel with centre $x$ and no more spokes, contradicting the global
	minimality of $W$.
\end{proof}


\subsection{The three-leaf attachment lemma}

We first state the cyclic attachment hypothesis explicitly.  Let
$H$ be an induced subdivision of $K_{1,3}$, with arms $P_a,P_b,P_c$,
common end $s$, and leaves $a,b,c$.  Suppose that
$V(H)\setminus\{a,b,c\}$ is anticomplete to the hole $C$.  Put
\[
A=N_C(a),\quad B=N_C(b),\quad K=N_C(c).
\]
These assumptions already imply $V(H)\cap V(C)=\emptyset$.  If
$z\in V(H)\setminus\{a,b,c\}$ also lay on $C$, either rim neighbor of $z$
would contradict anticompleteness; if a leaf, say $a$, lay on $C$, then its
neighbor on $P_a$ would give the same contradiction.
We say that these sets form \textit{noncrossing cyclic attachment blocks}
if there are three closed subpaths $I_a,I_b,I_c$ of $C$, possibly consisting
of one vertex, occurring in cyclic order such that
\[
A\subseteq V(I_a),\quad B\subseteq V(I_b),\quad
K\subseteq V(I_c),
\]
and any two block paths are vertex-disjoint or share only a common end.

\begin{lemma}\label{lem:modified-lsw-35}
	Let $C$ be a hole and let $H,P_a,P_b,P_c,s,a,b,c$ be as above.  Assume
	that $A,B,K$ form noncrossing cyclic attachment blocks and that, for
	distinct $r,t\in\{a,b,c\}$,
	\begin{equation}
		N_C(r)\neq N_C(t),\quad
		|N_C(r)\cap N_C(t)|\leq 1.
		\label{eq:35-distinct-attachments}
	\end{equation}
	Then at least one of the following holds:
	\begin{enumerate}[label=(\roman*)]
		\item one of $a,b,c$ is anticomplete to $C$;\label{lsw-35-1}
		\item $G$ contains a proper exact-four wheel.\label{lsw-35-2}
	\end{enumerate}
	Consequently, if a globally spoke-minimal proper wheel with at least five
	spokes has been fixed, outcome~\ref{lsw-35-2} is impossible, and at least one of
	$a,b,c$ is anticomplete to $C$.
\end{lemma}

\begin{proof}
	\setcounter{claim}{0}
	Suppose that all three leaves have neighbors on $C$.  We first establish
	the following local assertion:
	\begin{claim}\label{modified-lsw-35claim1}
		If a leaf has at least two neighbors on $C$, then either it has two
		nonadjacent neighbors on $C$, or outcome~\ref{lsw-35-2} of the lemma holds.
	\end{claim}
	\begin{proof}[Proof of Claim~\ref{modified-lsw-35claim1}]
	It is enough to consider a leaf $a$ satisfying
	\begin{equation}
		N_C(a)=\{a_1,a_2\},\quad a_1a_2\in E(C).
		\label{eq:35-adjacent-pair-a}
	\end{equation}
	Indeed, a clique in a hole has size at most two, so this is the only case
	in which a leaf has at least two rim neighbors but no nonadjacent pair.
	If both of the other attachment sets are contained in
	$\{a_1,a_2\}$, condition~\eqref{eq:35-distinct-attachments} permits a
	relabeling such that $a_1b,a_2c\in E(G)$.  Then $a$ is linked to the hole \(a_1-a_2-cP_csP_bb-a_1\) by the three paths given by $aa_1$, $aa_2$, and $aP_as$,
	contradicting $\mathrm{ISK}_4$-freeness.
	
	Otherwise, after interchanging $b$ and $c$ if necessary, $b$ has a
	neighbor on $C\setminus\{a_1,a_2\}$.  The intersection bound in
	\eqref{eq:35-distinct-attachments} allows us to name $a_1,a_2$ so that
	$a_2\notin N_C(b)$.  Apply Lemma~\ref{lem:external-connector-obstruction}
	with rim $R=C$, distinguished rim vertices $a_1,a_2$, and outside path
	$Q=aP_asP_bb$.  Take the first rim arc to be the edge $a_1a_2$ and the
	second to be its complementary arc.  The path $Q$ is induced, its
	interior is anticomplete to $C$, and the attachment assumptions follow
	from $N_C(a)=\{a_1,a_2\}$, $a_2\notin N_C(b)$, and the neighbor of
	$b$ outside this pair.  The subdivision and prism outcomes are forbidden.
	The remaining outcome is an exact-four wheel with two opposite short
	sectors, which \textbf{(N4)} turns into a proper exact-four wheel.
	This proves the local assertion.
	\end{proof}
	
	Assume now that outcome~\ref{lsw-35-2} does not hold.  By the local assertion, every
	leaf having more than one attachment has two nonadjacent neighbors on $C$.
	For each such leaf $t\in\{a,b,c\}$, let $\ell_t,r_t$ be its first and last
	attachments along the block $I_t$, and let \(J_t=\ell_tCr_t\subseteq I_t\) be the subpath whose interior contains
	$N_C(t)\setminus\{\ell_t,r_t\}$.  Each $J_t$ has at least two edges:
	otherwise $t$ would have exactly two adjacent attachments, contrary to
	the local assertion.  Replace $J_t$ by the two-edge path $\ell_t-t-r_t$.  By the noncrossing-block condition, the paths $J_t$ have pairwise
	disjoint interiors and can meet only at common ends, so these replacements
	can be performed simultaneously.
	
	We first verify that the result, denoted by $C^*$, is a simple cycle.
	Traverse $C$ once in cyclic order and replace each encountered subpath $J_t$
	by $\ell_t-t-r_t$.  Every retained vertex of $C$ is encountered once and the
	inserted leaves are distinct and lie outside $C$.  If two replacement paths
	meet at a block endpoint $q$, then $q$ is an end of both deleted subpaths and
	the two new edges incident with $q$ replace the two corresponding rim edges.
	If only one deleted subpath has end $q$, its new edge and the one retained
	rim edge give degree two at $q$.  Thus the resulting closed walk has no
	repeated vertex and every one of its vertices has degree two.  It is
	therefore a simple cycle; since $G$ is simple, it has at least three vertices.
	
	The cycle $C^*$ is induced.  An edge of $C$ whose two ends survive every
	replacement remains an edge of $C^*$, since each deleted subpath $J_t$
	has at least two edges.  Thus two retained vertices of $C$ cannot form a
	chord of $C^*$.  For every inserted leaf $t$, all vertices of
	$N_C(t)\setminus\{\ell_t,r_t\}$ lie in the deleted interior of $J_t$, and
	hence \(N_{C^*}(t)=\{\ell_t,r_t\}\).
	Finally, distinct inserted leaves are nonadjacent because $H$ is an induced
	subdivision of $K_{1,3}$.  These observations exclude every possible chord
	of $C^*$.
	
	We now construct three paths from $s$ to $C^*$.  If a leaf $t$ has
	multiple attachments, use $sP_tt$, whose end $t$ lies on $C^*$.  If $t$
	has the unique attachment $t'$, use $sP_tt-t'$.  The vertex $t'$ is
	retained on $C^*$: if $t'\in V(I_r)$ for a leaf $r$ with multiple
	attachments, then $t'$ is a common end of $I_t,I_r$ and hence lies
	outside $J_r^*$.  These three paths are
	internally vertex-disjoint, and there are no extra edges between them or
	from their interiors to $C^*$.  Their ends on $C^*$ are distinct:
	distinctness of the three attachment sets handles singleton arms, while
	inserted leaves are mutually distinct and lie outside $C$.  More explicitly,
	if the endpoint $q$ of a singleton-attachment arm is also an extreme
	attachment of an inserted leaf $r$, then $qr$ has both ends on $C^*$ and is
	a rim edge of $C^*$, not an extra cross-edge between path interiors.  Thus the three displayed paths link $s$ to the induced cycle $C^*$,
	a contradiction.  This also
	covers the case of no replacements, when $C^*=C$.  Hence one of the two
	stated outcomes must hold.
\end{proof}

\section{Wheel repairs and endpoint control}\label{sec:wheel-repairs}

We now work in the class \(\mathcal{C}\) of \(\{\mathrm{ISK}_4,\text{diamond}, \text{prism},K_{3,3}\}\)-free graphs, so the $K_{3,3}$-free hypothesis is imposed
in addition to the three exclusions used in the preceding section.
A sector $S$ of $W=(C,x)$ is a \textit{carrier} of a set
$A\subseteq V(C)$ if $A\subseteq V(S)$.
We use the notation and the linked-cycle principle from
Section~\ref{sec:preliminaries}, together with \textbf{(N4)} and the
definition of global spoke-minimality from Section~\ref{sec:local-wheels}.
The repair lemmas below state their own wheel hypotheses; global
spoke-minimality is required only where it is explicitly assumed.
Since a three-spoke wheel is an induced subdivision of $K_4$, every wheel
$W$ in an \(\{\mathrm{ISK}_4,\text{diamond}, \text{prism},K_{3,3}\}\)-free graph satisfies
\begin{equation}\label{eq:global-wheel-at-least-four-full}
	\sigma(W)\ge 4.
\end{equation}
Also, $K_4$ itself is an induced subdivision of $K_4$, so every graph in
$\mathcal C$ is $K_4$-free.  We first prove the auxiliary repair and endpoint
lemmas, then establish the complete four-spoke attachment case.

\subsection{Auxiliary attachment lemmas}

\begin{lemma}\label{lem:sector-carrier-facts}
	Let $W=(C,x)$ have at least four spokes.
	\begin{enumerate}[label=(\roman*)]
		\item If a set $Z$ is contained in a sector and is covered by two
		intersecting sectors, one member of that pair contains all of $Z$.\label{carrier-facts-1}
		\item Let $S_1,S_2$ be disjoint sectors with ends
		$a_1,b_1,a_2,b_2$ in cyclic order, and let $Q_1$ be the
		$b_2$--$a_1$ rim path avoiding $b_1$.  Suppose $Q_1^*$ contains
		a spoke.  Let nonempty sets $A_h\subseteq V(S_h)$ satisfy
		$A_1\ne\{b_1\}$ and $A_2\ne\{a_2\}$.  If two intersecting
		sectors cover $A_1\cup A_2$, then \(A_1=\{a_1\}\), \(A_2=\{b_2\}\), and $Q_1$ consists of exactly two sectors.  Those two sectors are
		the covering pair.\label{carrier-facts-2}
	\end{enumerate}
\end{lemma}

\begin{proof}
	For \ref{carrier-facts-1}, let $S$ carry $Z$.  If neither covering sector contains
	$Z$, neither is $S$.  Distinct sectors meet only at a common end;
	therefore the two covering sectors must be the sectors immediately
	before and after $S$.  These do not intersect when there are at least
	four spokes, a contradiction.
	
	For \ref{carrier-facts-2}, use \ref{carrier-facts-1} for each of $A_1,A_2$.  If $A_1\ne\{a_1\}$,
	then $A_1\ne\{b_1\}$ forces one covering sector to be $S_1$.
	The other covering sector, which meets $S_1$, can carry $A_2$ only by carrying the
	singleton $\{a_2\}$ on the $b_1$ side, or the singleton
	$\{b_2\}$ on the $a_1$ side.  The former is excluded by hypothesis;
	the latter would make $Q_1$ a single sector, contrary to the
	internal spoke.  Hence $A_1=\{a_1\}$, and symmetrically
	$A_2=\{b_2\}$.  Neither $S_1$ nor $S_2$ can then belong to an
	intersecting covering pair.  The pair must consist of the two end
	sectors of $Q_1$; they intersect exactly when $Q_1$ has two sectors.
\end{proof}

\begin{lemma}\label{lem:adjacent-sector-multiplicity}
	Let $G$ be \(\{\mathrm{ISK}_4,\text{diamond}, \text{prism},K_{3,3}\}\)-free, and let $W=(C,x)$ be a wheel in $G$.  Let $S_1,S_2$ be
	consecutive sectors of $W$, with $V(S_1)\cap V(S_2)=\{y\}$, and let
	$v\in N(x)\setminus V(C)$.  Suppose that \(N_C(v)\subseteq V(S_1)\cup V(S_2)\) and that $v$ has a neighbor in each of
	$V(S_1)\setminus V(S_2)$ and $V(S_2)\setminus V(S_1)$.  Then
	\[
	|N(v)\cap(V(S_i)\setminus V(S_{3-i}))|\ge2
	\quad(i=1,2).
	\]
	Moreover, if equality holds for either $i=1$ or $i=2$, then
	$vy\in E(G)$.
\end{lemma}

\begin{proof}
	For $i=1,2$, let $a_i$ be the end of $S_i$ different from $y$.
	Suppose, by symmetry, that $z$ is the unique neighbor of $v$ in
	$V(S_1)\setminus V(S_2)$.  First $vy\notin E(G)$, since otherwise $vx,vy,vz$ would link $v$ to
	the induced cycle $xS_1x$.
	
	Let $w$ be the neighbor of $v$ in $S_2$ closest to $y$.  If $w\ne a_2$, the paths
	\[
	vz,\quad vx,\quad vwS_2y
	\]
	link $v$ to the induced cycle $xS_1x$, a contradiction.  Hence $w=a_2$.  If $z=a_1$, then $G[\{x,v,a_1,a_2\}]$ is a diamond when
	$a_1a_2\notin E(G)$ and a $K_4$ when $a_1a_2\in E(G)$.  Thus $z\ne a_1$.
	Now the paths
	\[
	va_2,\quad vx,\quad vzS_1y
	\]
	link $v$ to the induced cycle $xS_2x$, a contradiction.  This proves
	the two lower bounds.
	
	Finally, suppose that one lower bound, say the one for $i$, is an equality
	and $vy\notin E(G)$.  The sector $S_i$ is necessarily long, and $v$ has
	exactly three neighbors on the hole $xS_ix$: the vertex $x$ and its two
	neighbors in $V(S_i)\setminus V(S_{3-i})$.  This is a forbidden
	three-spoke wheel.  Therefore $vy\in E(G)$.
\end{proof}

We next recall the notion of a skip.  Given a wheel $W=(C,x)$, a vertex
$v\in V(G)\setminus V(W)$ is a \textit{skip} for $W$ if there are two
consecutive sectors $S_1,S_2$ of $W$ such that
\begin{enumerate}[label=(S\arabic*)]
	\item $vx\in E(G)$;\label{skip-c-1}
	\item $v$ has a neighbor in each of
	$V(S_1)\setminus V(S_2)$ and $V(S_2)\setminus V(S_1)$;\label{skip-c-2}
	\item $N_C(v)\subseteq V(S_1)\cup V(S_2)$; and\label{skip-c-3}
	\item every $u\in V(G)\setminus V(W)$ adjacent to $v$ satisfies
	$N_C(u)\subseteq V(S_1)\cup V(S_2)$.\label{skip-c-4}
\end{enumerate}
If $V(S_1)\cap V(S_2)=\{a\}$, we also call $v$ an \textit{$a$-skip}.

For a positive integer $r$, the wheel $W$ is \textit{$r$-almost proper} if
there are pairwise nonconsecutive spokes $x_1,\ldots,x_r$ and a set
$X\subseteq V(G)\setminus V(W)$ such that $W$ is proper in $G-X$ and every
vertex of $X$ is an $x_i$-skip for some $i$.

\begin{lemma}\label{lem:one-skip-repair}
	Let $G$ be \(\{\mathrm{ISK}_4,\text{diamond}, \text{prism},K_{3,3}\}\)-free, and let $W=(C,x)$ be a one-almost proper wheel, with
	distinguished spoke $x_1$ and exceptional set $X$.  If $S_1,S_2$ are the
	two sectors incident with $x_1$, then $G$ contains a proper wheel
	$W'=(C',x)$ with $\sigma(W')=\sigma(W)$.  Moreover, either $W'=W$, or
	\[
	V(W')\setminus V(W)=\{v^*\},\quad
	V(W)\setminus V(W')\subseteq V(S_1^*)\cup V(S_2^*)\cup\{x_1\},
	\]
	where $v^*$ is an $x_1$-skip for $W$.
\end{lemma}

\begin{proof}
	If $W$ is already proper, take $W'=W$; hence assume it is not proper.
	Then $X\ne\emptyset$.  For $v\in X$, let $P(v)$ be the subpath of
	$C[V(S_1)\cup V(S_2)]$ between the two extreme rim neighbors of $v$.
	Choose $v^*\in X$ so that $|V(P(v^*))|$ is maximum, and let $Y$ be the
	interior of $P(v^*)$.  Replacing $P(v^*)$ by the two-edge path through
	$v^*$ gives the induced cycle \(C'=G[(V(C)\cup\{v^*\})\setminus Y]\).
	Every rim neighbor of $v^*$ lies on $P(v^*)$, and after its interior is
	deleted the only such neighbors retained are the two ends of $P(v^*)$.
	Thus no chord incident with $v^*$ remains, while the rest of the cycle is
	an induced subpath of $C$.  Every old spoke other than $x_1$ is retained,
	and $v^*$ is a new spoke.  Thus $C'$ has at least $\sigma(W)\ge4$ vertices
	and is a hole.  Also
	\[
	N(x)\cap V(C')=
	\bigl((N(x)\cap V(C))\setminus\{x_1\}\bigr)\cup\{v^*\},
	\]
	so $W'=(C',x)$ is a wheel with the same number of spokes as $W$.
	
	Suppose $W'$ is not proper, and choose $y\notin V(W')$ nonproper for it.
	Let $S'_1,S'_2$ be the two sectors of $W'$ containing the surviving parts
	of $S_1,S_2$.
	
	First suppose $y\in V(W)$.  Then $y\in Y$.  By
	Lemma~\ref{lem:adjacent-sector-multiplicity}, $v^*$ has at least two
	neighbors on each side of $x_1$, so $|V(P(v^*))|\ge4$.  Every neighbor of
	$y$ on $C'$ lies in one of $S'_1,S'_2$.  Apart from $v^*$, the only
	possible such neighbors are the two ends of $P(v^*)$, since $C$ is a
	hole.  The inequality $|V(P(v^*))|\ge4$ prevents $y$ from being adjacent
	on $C$ to both ends.  Hence $|N_{C'}(y)|\le2$; moreover, a possible old
	end neighbor and $v^*$ lie in the same new sector incident with $v^*$.
	Thus $y$ is proper for $W'$, a contradiction.
	
	Next suppose $y\in X$.  Then $yx\in E(G)$ and
	$N_{C'}(y)\subseteq V(S'_1)\cup V(S'_2)$.  Nonproperness gives a neighbor
	in each exclusive side.  Lemma~\ref{lem:adjacent-sector-multiplicity}
	gives at least two on each side.  On each side, one of these lies strictly
	beyond the corresponding extreme end of $P(v^*)$; consequently the path
	$P(y)$ strictly contains $P(v^*)$, contrary to maximality.
	
	Thus $y\notin V(W)\cup X$.  If $yv^*\notin E(G)$, every $C'$-neighbor of
	$y$ is an old rim neighbor, and properness for $W$ implies properness for
	$W'$.  Hence $yv^*\in E(G)$.  Skip condition~\ref{skip-c-4}, properness for $W$, and
	Lemma~\ref{lem:sector-carrier-facts}\ref{carrier-facts-1} allow us, after relabelling,
	to take $N_C(y)\subseteq V(S_1)$.  Hence
	$N_{C'}(y)\subseteq V(S'_1)$.  Nonproperness now forces
	$yx\notin E(G)$ and $|N_{C'}(y)|\ge3$.  But properness for $W$
	gives $|N_C(y)|\le2$, and the only possible new neighbor is
	$v^*$, so $|N_{C'}(y)|=3$, a forbidden three-spoke wheel.
\end{proof}

\begin{lemma}\label{lem:two-skip-repair}
	Let $G$ be \(\{\mathrm{ISK}_4,\text{diamond}, \text{prism},K_{3,3}\}\)-free, and let $W=(C,x)$ be a two-almost proper wheel.
	Then $G$ contains a proper wheel with center $x$ and with at most
	$\sigma(W)$ spokes.
\end{lemma}

\begin{proof}
	Let $x_1,x_2$ and $X^0$ certify that $W$ is two-almost proper.  Replace
	$X^0$ by the set $X$ of all vertices outside $W$ that are not proper for
	$W$.  Since $W$ is proper in $G-X^0$, the new set is contained in $X^0$
	and is still a valid exceptional set.  If it is empty, there is nothing to
	prove.  Let
	\[
	X_1=\{z\in X:z\text{ is an }x_1\text{-skip}\},\quad X_2=X\setminus X_1.
	\]
	Every member of $X_2$ is an $x_2$-skip.  If either part is empty,
	Lemma~\ref{lem:one-skip-repair} applies, so assume both are nonempty.
	Let $S_1,S_2$ be the sectors incident with $x_1$.
	
	The wheel $W$ is one-almost proper but not proper in $G-X_2$: every member
	of $X_1$ remains nonproper for the fixed wheel $W$ after $X_2$ is deleted.
	Apply Lemma~\ref{lem:one-skip-repair} in $G-X_2$.  Its output
	$W'=(C',x)$ therefore differs from $W$ and inserts a vertex
	$v^*\in X_1$.  In particular $v^*$ is still an $x_1$-skip in the full
	graph $G$.  The wheel $W'$ is proper in $G-X_2$ and has the same number of
	spokes as $W$.  If every nonproper vertex of $W'$ in $X_2$ is an
	$x_2$-skip, then $W'$ is one-almost proper and a second application of the
	one-skip repair completes the proof.  Thus the set
	\[
	\mathcal B=\{z\in X_2:z\text{ is neither proper nor an }x_2
	\text{-skip for }W'\}
	\]
	is nonempty.
	
	Let $S_3,S_4$ be the sectors of $W$ incident with $x_2$.  Since $x_1$ and
	$x_2$ are not consecutive, neither $S_3$ nor $S_4$ is $S_1$ or $S_2$;
	the one-skip repair leaves them unchanged, and they are sectors of $W'$.
	
	Moreover, $X_1$ is anticomplete to $X_2$.  Indeed, an edge
	$z_1z_2$ with $z_i\in X_i$ would, by skip condition~\ref{skip-c-4}, give
	\[
	N_C(z_1)\subseteq
	(V(S_1)\cup V(S_2))\cap(V(S_3)\cup V(S_4)).
	\]
	This intersection consists only of outer boundary spokes, at most one
	in each of $V(S_1)\setminus V(S_2)$ and
	$V(S_2)\setminus V(S_1)$, contrary to
	Lemma~\ref{lem:adjacent-sector-multiplicity}.
	
	Choose $v\in\mathcal B$, and for $i\in\{3,4\}$ let $s_i$ be its
	neighbor in $S_i$ farthest from $x_2$.  The vertex $v$ has a neighbor
	in each exclusive side of $S_3,S_4$.
	
	In particular, $vv^*\notin E(G)$.
	Consequently $v$ still satisfies conditions~\ref{skip-c-1}--\ref{skip-c-3} in the definition
	of an $x_2$-skip for $W'$: it is adjacent to $x$, it has a neighbor in
	each exclusive side of $S_3,S_4$, and
	$N_{C'}(v)=N_C(v)\subseteq V(S_3)\cup V(S_4)$.  Since $v$ is not an
	$x_2$-skip for $W'$, it is condition~\ref{skip-c-4} that fails.  Hence there is a
	vertex $u\in N(v)\setminus V(W')$ with a neighbor on
	$C'-(V(S_3)\cup V(S_4))$.
	
	We have $u\notin V(W)$.  Otherwise
	$u\in V(W)\setminus V(W')$, and the one-skip repair places $u$ in the
	deleted interior on the two sides of $x_1$.  But $uv\in E(G)$ would then
	put this old rim vertex in
	$N_C(v)\subseteq V(S_3)\cup V(S_4)$, impossible because $x_1,x_2$ are
	nonconsecutive.  Now skip condition~\ref{skip-c-4} for $v$ on the old wheel gives
	$N_C(u)\subseteq V(S_3)\cup V(S_4)$.  Therefore the neighbor of $u$ on
	$C'-(V(S_3)\cup V(S_4))$ is the only new rim vertex $v^*$, and
	$uv^*\in E(G)$.
	
	The two skip conditions now give
	\[
	N_C(u)\subseteq
	(V(S_1)\cup V(S_2))\cap(V(S_3)\cup V(S_4))\subseteq N_C(x).
	\]
	Moreover $ux\notin E(G)$: otherwise
	$G[\{x,u,v,v^*\}]$ is a diamond, with missing edge $vv^*$.  Hence
	$u\notin X$, and $u$ is proper for $W$.  The displayed intersection
	consists only of boundary spokes of the four sectors; since $x_1,x_2$ are
	nonconsecutive, no sector contains two of these common boundary vertices.
	It follows from properness that $|N_C(u)|\le1$.
	
	Let $C_v$ be the cycle obtained from $C'$ by replacing the
	$s_3$--$s_4$ subpath through $x_2$ by $s_3-v-s_4$.
	The extremal choices of $s_3,s_4$ exclude every chord from $v$,
	while $C'$ is a hole.  The retained vertex $v^*$ is nonadjacent
	to $v$, so $C_v$ has length at least four and is a hole.
	If $N_C(u)=\{q\}$, then $q$ is a common boundary spoke of the two
	pairs of sectors and hence is retained on $C_v$.  Thus $u$ has
	exactly the three neighbors $v,v^*,q$ on $C_v$, impossible.
	Therefore $N_C(u)=\emptyset$.
	
	Let $P'_1$ be the $s_3$--$v^*$ subpath of $C'$ avoiding $x_2$, and put
	$P_1=v-s_3P'_1v^*$.  Define $P'_2$ and $P_2=v-s_4P'_2v^*$ analogously.
	Thus $C_v=P_1\cup P_2$.  The cycle \(D=u-v-s_3P'_1v^*-u\) is a hole.  Indeed, the extreme choice of $s_3$ leaves no further
	$v$-neighbor on $P'_1-s_3$, the vertex $u$ is anticomplete to the old rim,
	$vv^*\notin E(G)$, and $P'_1$ is an induced subpath of the hole $C'$.
	The four displayed vertices $u,v,s_3,v^*$ are distinct, so this induced
	cycle has length at least four.  Let $q$ be the end of $S_3$ different
	from $x_2$.  Since $x_1,x_2$ are nonconsecutive, $q\ne x_1$, so $q$
	survives the first repair and lies on the retained $s_3$--$v^*$ arc.
	Thus $x$ has the three distinct neighbors $v,v^*,q$ on $D$, and
	$W''=(D,x)$ is a wheel.  In passing from $W'$ to $W''$, the old spoke
	$x_2$ is deleted and only the new spoke $v$ is
	inserted; any additional deleted spokes only decrease the count.  Hence
	$\sigma(W'')\le\sigma(W')=\sigma(W)$.  Let $S'_3$ be its sector
	containing $v$ but not $u$.
	
	If $W''$ is proper we are done.  Otherwise choose
	$y\notin V(W'')$ nonproper for $W''$.  Every vertex of
	$V(W)\setminus V(W'')$ has at most two neighbors on $D$, and those
	neighbors lie in a single sector of $W''$.  To see this directly, a rim
	vertex deleted in the first repair sees only $v^*$ and possibly one
	retained end of $P(v^*)$; a vertex deleted on the $x_2$-side sees only
	$v$ and possibly $s_3$; and a vertex on the discarded $P'_2$-side sees
	at most $v$ and $v^*$.  The first two possible pairs lie on a common rim
	subpath of one sector, and the last pair lies in the sector $v-u-v^*$.
	Thus all these old vertices are proper for $W''$, and $y\notin V(W)$.
	
	We first show $yu\in E(G)$.  Suppose not.  If $y\in X_2$, then
	$yx\in E(G)$ and
	$N_D(y)\subseteq V(S'_3)\cup\{v^*\}$.  Also $yv^*\notin E(G)$ by the anticompleteness of $X_1$ and $X_2$.  Therefore
	$N_D(y)\subseteq V(S'_3)$, making $y$ proper for $W''$, a contradiction.
	Thus $y\notin X_2$, and $y$ is proper for $W'$.  If $yv\notin E(G)$, then
	$N_D(y)\subseteq N_{C'}(y)$.  The retained arc
	$D\cap C'=s_3P'_1v^*$ has a complementary arc containing the old spoke
	$x_2$ in its interior.  Thus no sector of $W'$ can contain that
	complementary arc, and each sector has a connected trace on the
	retained arc.  No new spoke lies inside the retained arc, so each such
	trace lies in one sector of $W''$.  No new neighbor of $y$ is inserted,
	and therefore $y$ is proper for $W''$.  Hence
	$yv\in E(G)$.  Condition~\ref{skip-c-4} for the $x_2$-skip $v$ gives
	$N_C(y)\subseteq V(S_3)\cup V(S_4)$.
	
	Thus $N_D(y)\subseteq V(S'_3)\cup\{v^*\}$ and $yv\in E(G)$.
	If $yx\in E(G)$, then $yv^*\notin E(G)$, since otherwise
	$\{x,y,v,v^*\}$ induces a diamond; consequently $y$ is proper for
	$W''$.  If $yx\notin E(G)$, properness for $W'$ gives
	$|N_{C'}(y)|\le2$.  Because $yu\notin E(G)$ and the only other new
	rim vertex is $v$, we have $|N_D(y)|\le3$.  Equality is forbidden,
	so $|N_D(y)|\le2$.  One neighbor is $v$, and any other lies with
	$v$ in $S'_3$ or in the sector $v-u-v^*$; hence $y$ is again proper.
	Both cases contradict the choice of $y$, proving $yu\in E(G)$.
	
	Diamond-freeness now gives
	\begin{equation}\label{eq:y-three-nonedges-full}
		|N(y)\cap\{v^*,v,x\}|\le1.
	\end{equation}
	Since $y$ is nonproper for $W''$, it has a neighbor in $P_1^*$.
	We first show that $y\notin X_2$.  Otherwise $yx\in E(G)$, so
	\eqref{eq:y-three-nonedges-full} gives $yv,yv^*\notin E(G)$.
	The old skip condition gives \(N_D(y)\subseteq V(S'_3)\cup\{u\}\), since the part of $S_3\cup S_4$ retained on $D$ lies in $S'_3$.
	The consecutive sectors $S'_3$ and $v-u-v^*$ then contradict
	Lemma~\ref{lem:adjacent-sector-multiplicity}: $y$ sees
	$S'_3-\{v\}$ but has the unique neighbor $u$ on the exclusive
	side of $v-u-v^*$.  Hence $y\notin X_2$.  Since $y\notin V(W)$
	and the only new vertex of $W'$ is $v^*\in V(D)$, the vertex $y$
	is outside $W'$ and is therefore proper for $W'$.
	
	Suppose that $y$ sees $\{v,v^*\}$.  By
	\eqref{eq:y-three-nonedges-full}, it sees exactly one of these vertices
	and $yx\notin E(G)$.  Properness for $W'$ and
	$V(C_v)\subseteq V(C')\cup\{v\}$ give
	$|N_{C_v}(y)|\le3$.  Since $y$ also sees $P_1^*$, this number
	is at least two and therefore equals two, as three is forbidden.
	Both neighbors lie on $P_1\subseteq D$; together with $u$ they are
	exactly the three neighbors of $y$ on $D$, again impossible.
	Thus $yv,yv^*\notin E(G)$.
	
	Since $y$ is proper for $W'$ and sees $P_1^*$, it cannot also
	see $P_2^*$: these interiors are separated in $C'$ by the
	spokes $v^*$ and $x_2$, so no sector of $W'$ contains a vertex
	of each.  Together with $yv,yv^*\notin E(G)$, this makes $y$
	anticomplete to $P_2$.
	
	Apply Lemma~\ref{lem:external-connector-obstruction} with $R=C_v$,
	$Q=u-y$, $a=v$, $b=v^*$, $A=P_2$, and $B=P_1$.
	Here $N_{C_v}(u)=\{v,v^*\}$, while
	$\emptyset\ne N_{C_v}(y)\subseteq V(P_1^*)$, so all attachment
	conditions hold.  The exceptional outcome requires $vv^*\in E(G)$,
	which is false.  The resulting induced subdivision of $K_4$ or prism
	is forbidden, completing the proof.
\end{proof}

\subsection{The endpoint lemma}

The next elementary dichotomy is useful both for the globally minimal case
and for the four-spoke base case.

\begin{lemma}\label{lem:two-end-dichotomy-full}
	Let $G$ be \(\{\mathrm{ISK}_4,\text{diamond}, \text{prism},K_{3,3}\}\)-free, let $W=(C,x)$ be a proper wheel in $G$, and let
	$P=p_1\cdots p_k\subseteq G-V(W)$ be an induced path.  Suppose that its
	ends have neighbors on $C$, its internal vertices are anticomplete to $C$,
	and no sector of $W$ contains \((N_C(p_1)\cup N_C(p_k))\).
	If $xp_1\notin E(G)$, then either $p_1$ has a unique neighbor on $C$, or,
	after reversing the orientation of $C$ if necessary, there are three
	consecutive rim vertices $r,q,s$ such that
	\[
	N_C(p_1)=\{r,q\},\quad N_C(p_k)=\{q,s\},
	\quad xp_1,xp_k\notin E(G),
	\]
	where $q$ is a spoke and $rq,qs$ lie in the two sectors incident with
	$q$.  In the latter case, $G$ also contains a proper exact-four wheel
	centred at $q$.  The same assertion holds with $p_1,p_k$ interchanged.
\end{lemma}

\begin{proof}
	Properness puts $N_C(p_1)$ and $N_C(p_k)$ in sectors $S_1$ and $S_2$,
	respectively.  The two sectors are distinct, and in particular $k\ge2$.
	Since $xp_1\notin E(G)$,
	properness also gives $|N_C(p_1)|\le2$.  Assume that $p_1$ does not have a
	unique rim neighbor; thus $N_C(p_1)=\{a,b\}$ with $a\ne b$.
	
	One $a$--$b$ arc of $C$ lies inside $S_1$; denote the other by $B$.
	Since distinct sectors meet in at most one common end, $S_2\subseteq B$
	and $|N_C(p_k)\cap\{a,b\}|\le1$.  Interchange $a,b$ if necessary so
	that $b\notin N_C(p_k)$.  Moreover, $p_k$ has a neighbor in $B^*$:
	otherwise its rim neighborhood is contained in $\{a,b\}$ and the union
	of the two endpoint neighborhoods is contained in $S_1$, a contradiction.
	
	Apply Lemma~\ref{lem:external-connector-obstruction} to $R=C$ and
	$Q=P$, with outside endpoints $p_1,p_k$ and the rim arc $B$ just chosen.
	All its attachment conditions have been verified, and $P^*$ is
	anticomplete to $C$ by hypothesis.  Since induced subdivisions of $K_4$ and induced prisms
	are forbidden, its exceptional conclusion gives a vertex $t\in B^*$
	with
	\[
	ab,at\in E(C),\quad N_C(p_k)=\{a,t\}.
	\]
	Set $r=b$, $q=a$, and $s=t$.  These are three consecutive rim vertices,
	and the two endpoint neighborhoods have the required form.  If $q$ were
	not a spoke, both edges $rq,qs$ would lie in the sector containing $q$,
	contrary to the hypothesis on their union.  Thus $q$ is a spoke, and the
	two edges lie in its two incident sectors.  Finally, if $xp_k\in E(G)$,
	then $\{x,p_k,q,s\}$ induces a diamond when $xs\notin E(G)$ and a
	$K_4$ when $xs\in E(G)$; both are forbidden.  Hence $xp_k\notin E(G)$,
	as required.  The exceptional conclusion of
	Lemma~\ref{lem:external-connector-obstruction} also supplies an exact-four
	wheel centred at $a=q$, with opposite short sectors.  Applying \textbf{(N4)}
	gives the asserted proper exact-four wheel.  Interchanging the ends
	of $P$ proves the symmetric assertion.
\end{proof}

\begin{lemma}\label{lem:modified-endpoint}
	Under the hypotheses of Lemma~\ref{lem:two-end-dichotomy-full}, suppose
	that $W$ is globally spoke-minimal and $\sigma(W)\ge5$.
	For $i\in\{1,k\}$, if $xp_i\notin E(G)$, then $p_i$ has exactly one
	neighbor on $C$.
\end{lemma}

\begin{proof}
	Otherwise Lemma~\ref{lem:two-end-dichotomy-full}, applied from an end
	having more than one rim neighbor and not adjacent to $x$, gives a
	proper exact-four wheel.  This contradicts global spoke-minimality,
	since $\sigma(W)\ge5$.  The centre of the new wheel need not be $x$;
	this is why minimality is taken over all centres.
\end{proof}

\subsection{The four-spoke base case}

\begin{lemma}\label{lem:four-spoke-attachment}
	Let $G$ be \(\{\mathrm{ISK}_4,\text{diamond}, \text{prism},K_{3,3}\}\)-free, let $W=(C,x)$ be a proper wheel with
	$\sigma(W)=4$, and let $P\subseteq G-V(W)$ be an induced path with
	$|P\cap N(x)|\le1$.  If $P\cap N(x)=\emptyset$, all $C$-ends of
	$P$--$C$ edges lie in one sector.  If $|P\cap N(x)|=1$, all such ends lie
	in the union of two intersecting sectors, which may be equal.
\end{lemma}

\begin{proof}
	Choose a counterexample $P=p_1\cdots p_k$ with $k$ minimum.  Both ends
	have a neighbor on $C$.  Indeed, an end without a rim neighbor may be
	deleted without changing the rim-attachment set; if that end was the
	unique neighbor of $x$, failure of the two-sector conclusion for $P$
	implies failure of the one-sector conclusion for the remaining path.
	Either way a shorter counterexample results.  Properness gives $k>1$.
	Write the four spokes of $W$ as
	$a_1,a_2,a_3,a_4$ in cyclic order and let $T_i=a_iCa_{i+1}$, with
	subscripts modulo four.
	
	First suppose that $P^*$ is anticomplete to $C$.  By
	Lemma~\ref{lem:two-end-dichotomy-full}, every end nonadjacent to $x$ has a
	unique rim neighbor, unless the two attachment pairs have the form
	\[
	N_C(p_1)=\{r,q\},\quad N_C(p_k)=\{q,s\},
	\]
	where $q$ is the common spoke of two consecutive sectors.  In the latter
	case, if $P\cap N(x)=\emptyset$, diamond-freeness gives
	$xr,xs\notin E(G)$: otherwise $\{x,p_1,q,r\}$ or
	$\{x,p_k,q,s\}$ induces a diamond.  Replacing $rCqCs$ by
	$r-p_1Pp_k-s$ then gives a three-spoke wheel centered at $x$.  If instead
	$|P\cap N(x)|=1$, the displayed two sectors already cover all attachments
	of $P$.  Both alternatives contradict our choice of $P$.  Thus every end
	nonadjacent to $x$ has a unique rim neighbor.
	
	At least one end of $P$ misses $x$; reverse $P$ so that it is $p_k$.
	After cyclic relabeling, write
	$N_C(p_k)=\{z\}\subseteq V(T_1)$.  Suppose first that $p_1x\notin E(G)$,
	and write $N_C(p_1)=\{s\}$.  If $P$ has no neighbor of $x$, then no sector
	contains both $s$ and $z$.  Up to the dihedral symmetries of the four
	sectors, there are two possibilities.  In the case where $s,z$ can be
	assigned to two consecutive containing sectors, take
	$z\in V(T_1)$ and $s\in V(T_4)$; if one attachment is a spoke, choose the
	labels so that it is $z$, while the case of two nonspoke attachments in
	adjacent sectors has the same form.  Neither attachment is $a_1$, since
	then $T_1$ or $T_4$ would contain both.  Consequently \(H=p_1Pp_kzT_1a_2T_2a_3T_3a_4T_4sp_1\) is a hole and $N_H(x)=\{a_2,a_3,a_4\}$, a forbidden three-spoke wheel.
	In the remaining case, where both attachments are interior to opposite
	sectors, reflection permits us to take
	$z\in T_1^*$ and $s\in T_3^*$.  The vertex $x$ is linked to the
	hole \(H'=p_1Pp_kzT_1a_2T_2a_3T_3sp_1\) by $xa_2,xa_3$, and $xa_1T_1z$.
	
	The cycles just displayed are induced: $P^*$ is anticomplete to
	$C$, the two path ends have the unique rim neighbors $s,z$, and $P,C$
	are induced.  The chosen open sector arcs also exclude every extra
	$x$-edge in the stated three-spoke or linked-hole construction.
	
	If $P$ has the unique $x$-neighbor $p_j$, then $p_j\in P^*$.  A
	counterexample cannot have $s,z$ in intersecting sectors, and hence the
	same relabeling gives $s\in T_3^*$ and $z\in T_1^*$.  On the hole
	$H'$ we then have $N_{H'}(x)=\{p_j,a_2,a_3\}$, again impossible.  Thus
	$p_1x\in E(G)$ and it is the unique edge from $x$ to $P$.
	
	Let $S$ be the sector containing $N_C(p_1)$, whose existence follows
	from properness.  The unique neighbor $z$ of $p_k$ is not a spoke.  For
	if it were, one of the two sectors incident with $z$ would intersect
	$S$, and those two intersecting sectors would cover all attachments of
	$P$.  Relabel so that $z\in T_1^*$.  Failure of the conclusion then
	forces $S=T_3$.  Let $u$ be the neighbor of $p_1$ in $T_3$ closest to
	$a_3$.  We have $u\ne a_4$: otherwise extremality gives
	$N_C(p_1)=\{a_4\}$ and the intersecting sectors $T_4,T_1$ cover every
	attachment.  The paths
	\[
	p_1x,\quad p_1Pp_kzT_1a_2,\quad p_1uT_3a_3
	\]
	link $p_1$ to the induced cycle $xT_2x$.  Here
	$z\in T_1^*$ and $u\ne a_4$ exclude extra $x$-edges on the two long
	routes; the extreme choices of $z,u$, the fact that $P^*$ is
	anticomplete to $C$, and inducedness of $P,C$ exclude all other
	cross-edges.  Hence
	$P^*$ has a neighbor on $C$.
	
	Assume next that $P\cap N(x)=\emptyset$.  Apply minimality to the proper
	prefix $P^-=p_1\cdots p_{k-1}$ and suffix $P^+=p_2\cdots p_k$.
	There are sectors $U,V$ containing, respectively, every rim end of an edge
	from $P^-$ and from $P^+$.  Since $P$ is a counterexample,
	$U\ne V$, $p_1$ has an attachment in $V(U)\setminus V(V)$, and $p_k$ one
	in $V(V)\setminus V(U)$.  Every attachment of $P^*$ lies in
	$V(U)\cap V(V)$.  This intersection is nonempty, so $U,V$ are consecutive.
	Relabel them as $T_1,T_2$, with common spoke $a_2$.  Let $s$ be the
	neighbor of $p_1$ in $T_1$ closest to $a_1$, and let $z$ be the neighbor
	of $p_k$ in $T_2$ closest to $a_3$.  Then \(H=p_1Pp_kzT_2a_3T_3a_4T_4a_1T_1sp_1\) is a hole on which $x$ has exactly the three neighbors $a_1,a_3,a_4$.
	Indeed, every $C$-neighbor of $P^*$ is the omitted common spoke
	$a_2$; the extreme choices of $s,z$ exclude endpoint chords, and $P,C$
	are induced.  This forbidden three-spoke wheel is a contradiction.
	
	It remains that $P$ has a unique $x$-neighbor, say $p_j$.  Choose an
	arbitrary $p_i\in P^*$ with a neighbor on $C$.  We claim that $i=j$.
	If not,
	reverse $P$ so that $j<i$.  Minimality applied to
	$p_1\cdots p_{k-1}$ and to $p_{j+1}\cdots p_k$ gives two intersecting
	sectors $U,V$ covering the former attachments and one sector $R$ covering
	the latter.  The attachment of $p_i$ lies in
	$(V(U)\cup V(V))\cap V(R)$.  Since $P$ is a counterexample, $R$ is not
	one of $U,V$, and $U\ne V$; otherwise two intersecting sectors already
	cover all attachments of $P$.  Thus, after relabeling,
	$U=T_1,V=T_2,R=T_3$.  Every attachment of
	$p_{j+1},\ldots,p_{k-1}$ is consequently at the common spoke $a_3$.
	
	If every attachment of $p_1,\ldots,p_j$ lies in $T_2$, then $T_2,T_3$
	cover $P$.  Thus let $r\le j$ be largest such that $p_r$ has a neighbor
	in $V(T_1)\setminus V(T_2)$.  The vertex $p_k$ has an attachment in
	$V(T_3)\setminus\{a_3\}$; otherwise $T_1,T_2$ would cover all
	attachments.  Let $t$ be the neighbor of $p_r$ in $T_1$ closest to
	$a_1$, and let $z$ be the neighbor of $p_k$ in $T_3$ closest to $a_4$.
	Then \(H=p_rPp_kzT_3a_4T_4a_1T_1tp_r\) is a hole and has exactly the three $x$-neighbors $p_j,a_4,a_1$.
	Indeed, maximality of $r$ excludes every displayed $T_1$-attachment
	after $p_r$; every $p_h$ with $j<h<k$ can attach only at the omitted
	spoke $a_3$; the extreme choices of $t,z$ exclude the two endpoint
	chords; and $P,C$ are induced.  This is a forbidden three-spoke wheel.
	Thus $i=j$.  Since the internal rim-attacher $p_i$ was arbitrary,
	$P-\{p_1,p_i,p_k\}$ is anticomplete to $C$.
	
	Properness puts $N_C(p_i)$ in one sector, say $T_1$.  If $p_i$ has a
	neighbor in $T_1^*$, apply minimality to $p_1Pp_i$ and $p_iPp_k$.
	Both paths have order less than $k$, and $p_i$ is the unique neighbor of
	$x$ on each.  Any covering pair of intersecting sectors must contain
	$T_1$; its other member is one of $T_4,T_1,T_2$.  If one end has no
	attachment outside $T_1$, or if both ends use the same neighboring
	sector, then $T_1$ together with that neighbor covers every attachment
	of $P$.  Since $P$ is a counterexample, reflection therefore gives
	\[
	N_C(p_1)\subseteq V(T_4),\quad
	N_C(p_k)\subseteq V(T_2),
	\]
	and each end has an attachment outside $T_1$.  Choose
	$s\in N_C(p_1)$ closest to $a_4$ in $T_4$ and
	$z\in N_C(p_k)$ closest to $a_3$ in $T_2$.  Then \(H=p_1Pp_kzT_2a_3T_3a_4T_4sp_1\) is a hole.  The vertices of $P-\{p_1,p_i,p_k\}$ are anticomplete to $C$,
	all $C$-neighbors of $p_i$ lie on the omitted sector $T_1$, and the
	extreme choices of $s,z$ exclude endpoint chords; hence
	$N_H(x)=\{p_i,a_3,a_4\}$.  This forbidden three-spoke wheel is a
	contradiction.
	Therefore \(\emptyset\ne N_C(p_i)\subseteq\{a_1,a_2\}\).
	Adjacency to both ends gives a diamond when $T_1$ is long and a $K_4$
	when it is short.  Reflecting if needed, take $N_C(p_i)=\{a_1\}$.
	
	If the attachments of both ends lay in $V(T_1)\cup V(T_4)$, those two
	intersecting sectors would cover every attachment of $P$.  Failure of the
	conclusion therefore permits us, after reflecting and reversing $P$ if
	necessary, to take $N_C(p_1)\subseteq V(T_3)$.  Explicitly, the singleton
	$N_C(p_i)=\{a_1\}$ is carried by both $T_1$ and $T_4$; reflection fixing
	$a_1$ interchanges $T_2$ with $T_3$, and reversal of $P$ selects the end
	whose attachments escape $T_1\cup T_4$.  Let $s$ be its neighbor
	closest to $a_4$.  Suppose that $s\ne a_3$.  Then the paths
	\[
	p_ix,\quad p_ia_1,\quad p_iPp_1sT_3a_4
	\]
	link $p_i$ to the induced cycle $xT_4x$.  Indeed, the extreme choice of
	$s$ excludes every further $p_1$-edge to $sT_3a_4$,
	$P-\{p_1,p_i,p_k\}$ is anticomplete to $C$, and $s\ne a_3$ excludes an
	$xa_3$ chord on the third path.  This is impossible.
	Hence $s=a_3$, and extremality of $s$
	gives $N_C(p_1)=\{a_3\}$.  Finally properness puts $N_C(p_k)$ in some
	$T_h$.  If $h\in\{1,2\}$, then $T_1,T_2$ cover all attachments; if
	$h\in\{3,4\}$, then $T_3,T_4$ do.  This is the required conclusion.
\end{proof}

\section{Attachments of induced paths}\label{sec:path-attachments}

Throughout this section $G\in\mathcal C$.  We retain the carrier terminology
of Section~\ref{sec:wheel-repairs} and the definition of global
spoke-minimality from Section~\ref{sec:local-wheels}.
The next two lemmas concern paths with no neighbor, or exactly one
neighbor, of the wheel centre.  Their smaller-path hypotheses are stated
explicitly so that Section~\ref{sec:sector-separation} can combine them in
a well-founded simultaneous induction.  The four-spoke case has already
been treated in Lemma~\ref{lem:four-spoke-attachment}; here the wheel has
at least five spokes.

\subsection{The no-center-neighbor path lemma}

\begin{lemma}\label{lem:modified-lsw-313}
	Let $G$ be \(\{\mathrm{ISK}_4,\text{diamond}, \text{prism},K_{3,3}\}\)-free, let $W=(C,x)$ be globally spoke-minimal with
	$\sigma(W)\ge5$, and let $P=p_1\cdots p_k\subseteq G-V(W)$ be an induced
	path with $P\cap N(x)=\emptyset$.  Suppose that every induced path
	$P'\subseteq G-V(W)$ of order less than $k$ satisfies both statements
	below:
	\begin{enumerate}
		\item if $P'\cap N(x)=\emptyset$, all $C$-ends of $P'$--$C$ edges lie
		in one sector;
		\item if $|P'\cap N(x)|=1$, all such ends lie in the union of two
		intersecting sectors, which may be equal.
	\end{enumerate}
	Then all $C$-ends of $P$--$C$ edges lie in one sector of $W$.
\end{lemma}

\begin{proof}
	\setcounter{claim}{0}
	Assume that $P$ violates the conclusion.  Properness gives $k>1$.  The
	induction hypothesis applied to the two proper end-deleted subpaths gives
	distinct sectors $S_1,S_2$ such that
	\[
	\bigcup_{h<k}N_C(p_h)\subseteq V(S_1),\quad
	\bigcup_{h>1}N_C(p_h)\subseteq V(S_2).
	\]
	In particular $p_1$ has an attachment in $V(S_1)\setminus V(S_2)$,
	$p_k$ has one in $V(S_2)\setminus V(S_1)$, and every attachment of
	$P^*$ lies in $V(S_1)\cap V(S_2)$.
	
	For $i=1,2$, write the ends of $S_i$ as $a_i,b_i$, so that
	$a_1,b_1,a_2,b_2$ occur in this cyclic order and $a_1\ne b_2$.  Let
	$Q_1$ be the $b_2$--$a_1$ subpath of $C$ avoiding $b_1$, and $Q_2$ the
	$b_1$--$a_2$ subpath avoiding $a_1$.  Choose $S_1,S_2$ so that
	$|V(Q_2)|$ is minimum.  Let $s,t$ be the neighbors of $p_1$ in $S_1$
	closest to $a_1,b_1$, respectively, and let $y,z$ be the neighbors of
	$p_k$ in $S_2$ closest to $a_2,b_2$, respectively.  Then
	$s\ne b_1$ and $z\ne a_2$.  We give the minimality argument, since it
	also fixes the endpoint convention used below.  Suppose first that
	$s=b_1$.  By the choice of $s$, this says $N_C(p_1)=\{b_1\}$.  If $Q_2$
	is the singleton $b_1=a_2$, then $b_1\in V(S_2)$, contradicting the
	exclusive $S_1\setminus S_2$ attachment of $p_1$.  Thus $Q_2$ is
	nontrivial.  In that case $S_1$ and $S_2$ are disjoint (two sectors can
	meet only in a common end, which here would make $b_1=a_2$), and hence
	$P^*$ is anticomplete to $C$.  Consequently every attachment of
	$p_1\cdots p_{k-1}$ is the vertex $b_1$.  Let $R$ be the sector contained
	in $Q_2$ and incident with $b_1$.  Replacing $S_1$ by $R$ still covers
	all attachments of $p_1\cdots p_{k-1}$, while the new path between the
	end of $R$ different from $b_1$ and $a_2$ is a proper subpath of $Q_2$.
	This contradicts the minimal choice of $|V(Q_2)|$.  Thus $s\ne b_1$.
	The proof that $z\ne a_2$ is symmetric: if $z=a_2$, then
	$N_C(p_k)=\{a_2\}$, and, unless $Q_2$ is the forbidden singleton,
	replacing $S_2$ by the sector of $Q_2$ incident with $a_2$ strictly
	shortens $Q_2$ while preserving the suffix cover.  Consequently
	$\{s,z\}\cap V(Q_2)=\emptyset$.  If the
	two sectors meet, they meet at $b_1=a_2$ and $Q_2$ is that one vertex.
	In all cases $P^*$ is anticomplete to $Q_1$.  Hence
	\begin{equation}\label{eq:313-D1}
		D_1=s-p_1Pp_k-zS_2b_2Q_1a_1S_1s
	\end{equation}
	is a hole.

	\begin{claim}\label{modified-lsw-313-claim1}
		$W_1=(D_1,x)$ is a wheel with fewer spokes than $W$.
	\end{claim}
	\begin{proof}[Proof of Claim~\ref{modified-lsw-313-claim1}]
	The path $Q_2$ is disjoint from $D_1$ and contains a spoke, so $x$ has
	fewer neighbors on $D_1$ than on $C$.  It remains to show that $Q_1$
	contains at least three spokes, counting its ends.  Otherwise $Q_1$ is a
	sector of $W$.  At least one of $t\ne a_1$ and $y\ne b_2$ holds.  Indeed,
	if $t=a_1$ and $y=b_2$, then
	$N_C(p_1)=\{a_1\}$ and $N_C(p_k)=\{b_2\}$.  Since $Q_1$ is here a sector
	and $\sigma(W)\ge5$, $S_1,S_2$ are disjoint (if they met,
	$S_1,S_2,Q_1$ would be all three sectors of $W$).  Hence $P^*$ is
	anticomplete to $C$, and all attachments of $P$ lie in $Q_1$, a
	contradiction.  Assume by symmetry that $t\ne a_1$.  As $W$ has at least
	five spokes, $S_1,S_2$
	are disjoint.  Thus $P^*$ is anticomplete to all of $C$, and Lemma~\ref{lem:modified-endpoint} gives
	$s=t$.  The vertex $s$ is then linked to the induced cycle
	$a_1Q_1b_2-x-a_1$ by
	\[
	sS_1a_1,\quad s-p_1Pp_k-zS_2b_2,\quad sS_1b_1-x,
	\]
	a contradiction.  Here $s\in S_1^*$, $P^*$ is anticomplete
	to $C$, and the extreme choices of $s,z$ exclude further end-attachments.
	Since $\sigma(W)\ge5$ and $Q_1$ is a sector, $Q_2^*$ contains a
	spoke, excluding an edge from $b_1$ to the retained $S_2$-segment.
	This proves Claim~\ref{modified-lsw-313-claim1}.
	\end{proof}
	
	The wheel $W_1$ cannot have exactly three spokes; hence $x$ has at least
	two neighbors in $Q_1^*$.  Since $W_1$ has fewer spokes than the
	globally minimal proper wheel $W$, it is not proper.  Let \(S_0=a_1S_1s-p_1Pp_k-zS_2b_2\) be its sector complementary to the sectors contained in $Q_1$.

	\begin{claim}\label{modified-lsw-313-claim2}
		There is a vertex
		$v\notin V(W_1)\cup V(W)$ satisfying one of the following:
		\begin{enumerate}[label=(\alph*)]
			\item $vx\notin E(G)$,
			$|N(v)\cap V(S_0)|\ge3$, and $N_{D_1}(v)\subseteq V(S_0)$;\label{modified-lsw-313-claim2-1}
			\item for a sector $S_3$ of $W$ contained in $Q_1$, the vertex $v$
			has a neighbor in $V(S_3)\setminus V(S_0)$ and one in
			$V(S_0)\setminus V(S_3)$, and $N_C(v)\subseteq V(S_3)$.\label{modified-lsw-313-claim2-2}
		\end{enumerate}
	\end{claim}
	\begin{proof}[Proof of Claim~\ref{modified-lsw-313-claim2}]
	Choose a nonproper vertex $v\notin V(W_1)$.  We first check explicitly
	that $v\notin V(C)$.  The vertices deleted from $C$ in forming $D_1$
	lie on the $s$--$b_1$ tail of $S_1$, on $Q_2$, or on the
	$a_2$--$z$ tail of $S_2$.  An interior vertex of the first tail has all
	its new-rim neighbors in the adjacent pair $\{s,p_1\}$, and an interior
	vertex of the last tail has all of them in $\{p_k,z\}$; the same is true
	of $b_1$ and $a_2$, respectively, when these ends are distinct.
	Vertices of $Q_2^*$ have no neighbor on the new rim.  These facts
	follow from chordlessness of $C$, the attachment containments in
	$S_1,S_2$ displayed above, and the
	fact that every attachment of $P^*$ lies in
	$V(S_1)\cap V(S_2)$.  If $b_1=a_2$, that common spoke may see several
	vertices of $P^*$, but all its $D_1$-neighbors lie in $S_0$ and it
	is adjacent to $x$.  Thus every vertex of $C\setminus D_1$ is proper for
	$W_1$, proving $v\notin V(C)$.  Since $v$ is proper for $W$ but not for
	$W_1$, it has a neighbor on $P$.  If \ref{modified-lsw-313-claim2-1} fails, nonproperness supplies a sector
	$S_3\subseteq Q_1$ in which $v$ has an attachment outside $S_0$.
	Properness for $W$ gives $N_C(v)\subseteq V(S_3)$, while the already noted
	neighbor on $P\subseteq S_0$ proves \ref{modified-lsw-313-claim2-2}.
	\end{proof}
	
	Fix $v$ as in Claim~\ref{modified-lsw-313-claim2}.  In case \ref{modified-lsw-313-claim2-2}, let $T_a,T_b$ be the sectors of
	$Q_1$ incident with $a_1,b_2$, respectively.  Since $Q_1^*$ contains
	at least two spokes, $T_a,T_b$ are disjoint, so at most one carries the
	nonempty set $N_C(v)$.  Reflect the configuration if necessary so that
	$T_b$ does not carry $N_C(v)$, and choose $S_3$ to be the first carrier
	of $N_C(v)$ encountered from $b_2$ along $Q_1$.  This choice preserves
	\ref{modified-lsw-313-claim2-2}; in particular, $S_3$ is not incident with $b_2$.

	\begin{claim}\label{modified-lsw-313-claim3}
		The vertex $v$ has a unique neighbor on $P$.
	\end{claim}
	\begin{proof}[Proof of Claim~\ref{modified-lsw-313-claim3}]

	Suppose not.  Let $v_1,v_2$ be its neighbors on $P$ closest to $p_1,p_k$,
	respectively.
	
	Suppose first that $vx\notin E(G)$.  Properness gives
	$|N_C(v)|\le2$.  If $v$ has at least three neighbors on $P$, the
	shortcut assertion of Lemma~\ref{lem:sparse-neighborhood} makes
	\[
	P'=p_1Pv_1-v-v_2Pp_k
	\]
	an induced path of order less than $k$.  It has no neighbor of $x$ and
	retains the attachments $s,z$.  No sector contains both: $s\ne b_1$,
	$z\ne a_2$, and $Q_1$ contains at least four spokes.  This contradicts
	the induction hypothesis.  Hence $v$ has exactly two neighbors on $P$.
	
	In either alternative of Claim~\ref{modified-lsw-313-claim2}, $v$ has a neighbor on $D_1-P$.
	If it had exactly one there, $(D_1,v)$ would be a forbidden three-spoke
	wheel.  Thus it has two, and these are all its $C$-neighbors; write
	$N_C(v)=\{r,r'\}$.  In case~\ref{modified-lsw-313-claim2-1} of Claim~\ref{modified-lsw-313-claim2} they lie on one of the retained tails
	$a_1S_1s$ and $zS_2b_2$: no sector contains a vertex of each tail.
	Reflect within this case so that both lie on $a_1S_1s$.  Their unique
	carrier is then $S_1$.  In case~\ref{modified-lsw-313-claim2-2} of Claim~\ref{modified-lsw-313-claim2} retain the fixed orientation;
	their unique carrier is $S_3\subseteq Q_1$, not incident with $b_2$.
	
	If $v_2\ne p_2$, the shorter induced path $v-v_2Pp_k$ has no neighbor
	of $x$.  Its one-sector cover would have to be the unique carrier of
	$\{r,r'\}$.  In case~\ref{modified-lsw-313-claim2-1} of Claim~\ref{modified-lsw-313-claim2} that sector is $S_1$, which misses $z$; in
	case~\ref{modified-lsw-313-claim2-2} of Claim~\ref{modified-lsw-313-claim2} it is $S_3$, which is disjoint from $S_2$ and also misses $z$.
	This is impossible.  Hence $N_P(v)=\{p_1,p_2\}$.
	
	If $s\in\{r,r'\}$, then $\{v,p_1,p_2,s\}$ induces a diamond or a
	$K_4$.  Let $A$ be the $r$--$r'$ subpath of their carrier, and let $B$
	be the complementary arc of $C$.  We have
	\[
	\emptyset\ne N_C(p_1)\subseteq V(B^*).
	\]
	Indeed, in case~\ref{modified-lsw-313-claim2-1} of Claim~\ref{modified-lsw-313-claim2} both ends of $A$ strictly precede $s$, while all
	$p_1$-neighbors lie between $s$ and $t$.  In case~\ref{modified-lsw-313-claim2-2} of Claim~\ref{modified-lsw-313-claim2}, $S_3$ meets
	$S_1$ only possibly at $a_1$; if $p_1a_1\in E(G)$, then $s=a_1$ is
	not an end of $A$, and as an end of $S_3$ it cannot lie inside $A$.
	Apply Lemma~\ref{lem:external-connector-obstruction} with $R=C$,
	$Q=v-p_1$, and rim ends $r,r'$.  The path has no internal vertex,
	and the displayed attachment sets satisfy all hypotheses.  Its
	exceptional conclusion requires a common rim neighbor of $v,p_1$,
	which has just been excluded.  The other outcomes are forbidden.
	Thus $vx\in E(G)$, and Claim~\ref{modified-lsw-313-claim2}\ref{modified-lsw-313-claim2-2} holds with the fixed orientation.
	
	Suppose $v_2\ne p_2$.  Let $S_4\subseteq Q_1$ be the sector incident
	with $b_2$, and write its other end as $a_3$.  The shorter path
	$v-v_2Pp_k$ contains the unique $x$-neighbor $v$, so its attachments lie
	in two intersecting sectors.  By Lemma~\ref{lem:sector-carrier-facts}\ref{carrier-facts-1},
	one covering sector carries $N_C(v)$ and one carries $N_C(p_k)$.
	Since $z\ne a_2$, the latter is $S_2$, unless $N_C(p_k)=\{b_2\}$,
	when $S_4$ is also available.  A pair using $S_2$ would require its
	$Q_1$-neighbor $S_4$ to carry $N_C(v)$: that set contains a vertex of
	$Q_1^*$, so neither $S_2$ nor its other neighboring sector can
	carry it.  This contradicts our orientation.  Hence
	$N_C(p_k)=\{b_2\}$, carried by $S_4$.  The other covering sector must
	be the next sector along $Q_1$, which is $S_3$ by its choice.
	Thus $a_3\in V(S_3)$ and $y=b_2$.  Let $v'_1$ be the neighbor of
	$v$ in $S_3$ closest to $a_3$, and let $b_3$ be the other end of $S_3$.
	If $v'_1\ne b_3$, then $x$ has exactly three neighbors on \(v-v_2Pp_k-b_2S_4a_3S_3v'_1-v\).
	If $v'_1=b_3$, put \(D=v-v_2Pp_k-b_2S_4a_3S_3b_3-v\).
	Then $(D,x)$ is an exact four-spoke wheel, with spokes $v,b_2,a_3,b_3$ in rim
	order.  The adjacent pair $v,b_3$ bounds the short rim edge $vb_3$, and
	the designated sector $v-v_2Pp_k-b_2$ is the only part of $D$ not
	contained in $C$.  Lemma~\ref{lem:modified-lsw-32}, applied to that pair
	and that designated sector, yields a proper exact-four wheel.  Both
	outcomes are impossible, so $v_2=p_2$ and also $v_1=p_1$.
	
	Consider the hole \(H=p_2Pp_k-zS_2b_2-x-v-p_2\) and the induced path $Q=p_1-sS_1a_1$, disjoint from $H$.  The attachment
	distribution gives $s\in V(S_1)\setminus V(S_2)$ and
	$z\in V(S_2)\setminus V(S_1)$; also $S_3$ is not incident with $b_2$
	and $N_P(v)=\{p_1,p_2\}$.  These facts, together with the extreme
	choices of $s,z$, exclude chords of $H$ and edges from $Q^*$ to $H$.
	More explicitly, $Q^*$ lies in $S_1^*$, whereas the
	$C$-neighbors of $P-p_1$ lie in $S_2$ and those of $v$ lie in $S_3$;
	the only possible $x$-neighbor on $Q$ is its end $a_1$.  Thus
	\[
	N_H(p_1)=\{v,p_2\},\quad
	\{x\}\subseteq N_H(a_1)\subseteq\{x,v\}.
	\]
	Here $a_1b_2\notin E(G)$ since $Q_1^*$ contains at least two spokes.
	Apply Lemma~\ref{lem:external-connector-obstruction} with $R=H$, the
	path $Q$, rim ends $a=v,b=p_2$, and the arc $B=H-vp_2$.  The vertex
	$x$ lies in $B^*$, so every required attachment condition holds.
	Since $G$ contains neither an induced subdivision of $K_4$ nor a prism,
	the lemma yields an exact-four wheel with opposite short sectors.
	By \textbf{(N4)} this gives a proper exact-four wheel, contrary to global
	spoke-minimality.  Claim~\ref{modified-lsw-313-claim3} follows.
	\end{proof}
	
	Write $N_P(v)=\{p_j\}$.  Alternative \ref{modified-lsw-313-claim2-1} of Claim~\ref{modified-lsw-313-claim2} is now impossible, because it would make $v$ have exactly three
	neighbors on $D_1$.  Hence \ref{modified-lsw-313-claim2-2} of Claim~\ref{modified-lsw-313-claim2} holds and $v$ has a neighbor in $Q_1^*$.

	\begin{claim}\label{modified-lsw-313-claim4}
		Some internal vertex of $P$ has a neighbor on $C$.
	\end{claim}
	\begin{proof}[Proof of Claim~\ref{modified-lsw-313-claim4}]
	Suppose not.  Lemma~\ref{lem:modified-endpoint} gives
	$N_C(p_1)=\{s\}$ and $N_C(p_k)=\{z\}$.  If $N_C(v)=\{v'\}$,
	then $v'\in Q_1^*$ by Claim~\ref{modified-lsw-313-claim2}\ref{modified-lsw-313-claim2-2}, so $v'\notin\{s,z\}$, and
	$p_j$ is linked to $C$ by
	\[
	p_jPp_1-s,\quad p_jPp_k-z,\quad p_j-v-v',
	\]
	a contradiction.  If $N_C(v)$ has two nonadjacent vertices, choose its
	extreme neighbors $r,r'$ in $S_3$.  Replacing $rS_3r'$ by $r-v-r'$
	gives a hole retaining $s,z$.  Indeed, $S_3$ is disjoint from $S_2$ and
	meets $S_1$ only possibly at $a_1$, which can only be an extreme retained
	endpoint if it belongs to the replaced arc.  The vertex $p_j$ is linked
	to the new hole by the two $P$-subpaths ending through $s,z$ and the edge
	$p_jv$.  In both constructions, the singleton end-attachments,
	anticompleteness of $P^*$ to $C$, the extreme choices, and
	$N_P(v)=\{p_j\}$ exclude all extra edges, including when $p_j$ is an
	end of $P$.
	
	The only remaining possibility is $N_C(v)=\{r,r'\}$ with $rr'\in E(C)$.
	But then $N_{D_1}(v)=\{r,r',p_j\}$, making $(D_1,v)$ a forbidden
	three-spoke wheel.  This proves Claim~\ref{modified-lsw-313-claim4}.
	\end{proof}
	
	The attachment distribution now gives $b_1=a_2$, $k\ge3$, and an edge
	from $b_1$ to $P^*$.  Let $p_i$ be the neighbor of $b_1$ in
	$P^*$ closest to $p_1$.

	\begin{claim}\label{modified-lsw-313-claim5}
		$v$ is anticomplete to $\{p_1,p_k\}$.
	\end{claim}
	\begin{proof}[Proof of Claim~\ref{modified-lsw-313-claim5}]
	Suppose by symmetry that $vp_1\in E(G)$.  First let $vx\notin E(G)$.
	The order-two path $v-p_1$ has no neighbor of $x$.  Let $T$ be the
	sector supplied by the one-sector conclusion for this shorter path.
	Since $N_C(v)$ meets $Q_1^*$ whereas
	$N_C(p_1)\subseteq V(S_1)$, the sector $T$ must be the sector of $Q_1$
	incident with $a_1$, and $N_C(p_1)=\{a_1\}$.  Hence $s=t=a_1$.
	Since $vx\notin E(G)$, properness gives $|N_C(v)|\le2$.
	All these neighbors lie on $Q_1\subseteq D_1$, and Claim~\ref{modified-lsw-313-claim3} gives
	$N_P(v)=\{p_1\}$.  Two $C$-neighbors would therefore make $(D_1,v)$
	a forbidden three-spoke wheel.  Consequently
	\[
	N_C(v)=\{q\},\quad q\in Q_1^*.
	\]
	Now the three paths
	\[
	p_1a_1,\quad p_1vq,\quad p_1Pp_i b_1
	\]
	link $p_1$ to $C$.  Their rim ends $a_1,q,b_1$ are distinct.
	The choice of $p_i$ makes $p_2,\ldots,p_{i-1}$ anticomplete to $C$;
	$p_i$ has only the rim neighbor $b_1$, and $v$ has the unique
	$P$-neighbor $p_1$.  Together with $N_C(p_1)=\{a_1\}$ and
	$N_C(v)=\{q\}$, these facts exclude every additional attachment
	and cross-edge.  This contradiction excludes $vx\notin E(G)$.
	
	Thus $vx\in E(G)$.  If $va_1\in E(G)$, properness of $v$ for $W$ and
	its neighbor in $Q_1^*$ make it anticomplete to $S_1^*$, and
	$v$ is linked to $xS_0x$ by $vx,va_1,vp_1$.  Symmetrically
	$vb_2\notin E(G)$.  If $s\ne t$, the paths $p_1s,p_1t,p_1vx$ link
	$p_1$ to the induced cycle $xS_1x$.  If $s=t$, use
	$p_1s,p_1Pp_ib_1,p_1vx$ instead.  Here $v$ is anticomplete to $S_1$
	and has the unique $P$-neighbor $p_1$; properness gives
	$|N_C(p_1)|\le2$, and $p_i$ is the first internal vertex of $P$
	adjacent to $b_1$.  Both alternatives contradict the linked-cycle
	principle, proving Claim~\ref{modified-lsw-313-claim5}.
	\end{proof}
	  
	\begin{claim}\label{modified-lsw-313-claim6}
		$vx\in E(G)$.
	\end{claim}
	\begin{proof}[Proof of Claim~\ref{modified-lsw-313-claim6}]
	Put \(L=p_iPp_j-v\).
	Claims~\ref{modified-lsw-313-claim3} and~\ref{modified-lsw-313-claim5} give $N_P(v)=\{p_j\}$ and $p_j\in P^*$.
	Thus $L$ is induced and, since $p_i,p_j$ are internal vertices of $P$,
	has order less than $k$.  Suppose that $vx\notin E(G)$.  Then $L$
	has no neighbor of $x$.
	Its $p_i$-end has the rim attachment $b_1$, whereas its $v$-end has
	an attachment in $V(S_3)\setminus V(S_0)$.  No sector of $W$ contains
	both: the sectors incident with $b_1=a_2$ are $S_1,S_2$, while $S_3$ is
	a sector of $Q_1$ with an attachment outside $\{a_1,b_2\}$.  This
	contradicts the no-$x$ induction hypothesis and proves Claim~\ref{modified-lsw-313-claim6}.
\end{proof}
	  
	\begin{claim}\label{modified-lsw-313-claim7}
		$v$ is anticomplete to $\{a_1,b_2\}$.
	\end{claim}
	\begin{proof}[Proof of Claim~\ref{modified-lsw-313-claim7}]
	If, say, $va_1\in E(G)$, then $v$ has exactly the three neighbors
	$x,a_1,p_j$ on the hole $xS_0x$, a forbidden three-spoke wheel.
	Here $S_0$ contains all of $P$, whose order is at least three.
	The other end is symmetric.
\end{proof}
	
	The same path $L=p_iPp_j-v$ now contains exactly one neighbor of $x$,
	namely $v$.  By induction its attachments lie in two intersecting sectors.  One of those sectors
	contains the attachment $b_1$ and is therefore $S_1$ or $S_2$; the other
	must meet it and carry the attachment of $v$ in $Q_1$.  After reflecting
	the whole configuration we may consequently choose $S_3$ with
	$a_1\in V(S_3)$.
	
	Let $v^*$ be the neighbor of $v$ in $S_3$ closest to $a_1$.  Claim~\ref{modified-lsw-313-claim7}
	gives $v^*\ne a_1$, and \(R=v-v^*S_3a_1S_1s-p_1Pp_j-v\) is a hole.  Indeed, the choice of $v^*$ excludes further edges from
	$v$ to $v^*S_3a_1$; Claim~\ref{modified-lsw-313-claim3} excludes further edges from $v$ to $P$;
	the choice of $s$ excludes further edges from $p_1$ to $a_1S_1s$; and,
	since $S_1\cap S_2=\{b_1\}$, every attachment of $P^*$ to $C$ is
	$b_1$, which is absent from $R$.  All remaining pieces are induced
	subpaths of the induced cycle $C$ and the induced path $P$.
	
	If $v^*$ is the end of $S_3$ different from $a_1$, then
	$N_R(x)=\{v,v^*,a_1\}$.  Thus the three edges $xv,xv^*,xa_1$ link $x$ to $R$.
	
	Suppose that $v^*$ is not the other end of $S_3$.  Then
	$N_R(x)=\{v,a_1\}$, and a third $x$--$R$ path is \(B=x-b_2S_2z-p_kPp_j\).
	The path $B$ is induced and its interior is anticomplete to $R-p_j$:
	$z$ was chosen closest to $b_2$ among the $p_k$-neighbors in $S_2$;
	$P$ is induced and $x$ is anticomplete to $P$; the attachment
	distribution gives $N_C(P^*)\subseteq\{b_1\}$, while $b_1$ lies on
	neither displayed $C$-segment; and $N_C(v)\subseteq V(S_3)$ together
	with $N_P(v)=\{p_j\}$ excludes the remaining possible cross-edges.
	The three paths $xv$, $xa_1$, and $B$ therefore link $x$ to the hole
	$R$, again a contradiction.  This proves the lemma.
\end{proof}

\subsection{The one-center-neighbor path lemma}

\begin{lemma}\label{lem:modified-lsw-314}
	Let $G$ be \(\{\mathrm{ISK}_4,\text{diamond}, \text{prism},K_{3,3}\}\)-free, let $W=(C,x)$ be globally spoke-minimal with
	$\sigma(W)\ge5$, and let $P=p_1\cdots p_k\subseteq G-V(W)$ be an induced
	path containing exactly one neighbor of $x$.  Suppose that every induced
	path $P'\subseteq G-V(W)$ of order less than $k$ satisfies the one-sector
	conclusion when $P'\cap N(x)=\emptyset$ and the two-intersecting-sector
	conclusion when $|P'\cap N(x)|=1$.  Suppose also that the one-sector
	conclusion holds for every induced path of order $k$ having no neighbor of
	$x$.  Then all $C$-ends of $P$--$C$ edges lie in the union of two
	intersecting sectors, which may be equal.
\end{lemma}

\begin{proof}
	\setcounter{claim}{0}
	Suppose otherwise, and initially let $P$ be any counterexample of the
	prescribed order.  Let $p_i$ be the unique neighbor of $x$ on $P$.
	Properness gives $k>1$.
	Both ends of $P$ have a
	neighbor on $C$.  Otherwise delete an unattached
	end.  If that end is $p_i$, the resulting path has no neighbor of $x$ and
	the already available one-sector conclusion covers the unchanged
	attachment set; otherwise the shorter path still has one $x$-neighbor and
	the two-sector induction conclusion applies.  Either outcome contradicts
	the choice of $P$.  By properness, choose sectors $S_1,S_2$ carrying
	$N_C(p_1),N_C(p_k)$, respectively, and write their ends as $a_h,b_h$
	for $h=1,2$.  If $P^*$ is anticomplete to $C$, these sectors are
	disjoint, since otherwise they would cover every attachment of $P$.
	When internal attachments exist, we refine this choice in the following claim.

	\begin{claim}\label{modified-lsw-314-claim1}
		One of the following holds.
		\begin{enumerate}[label=(\alph*)]
			\item $P^*$ is anticomplete to $C$.\label{modified-lsw-314-claim1-a}
			\item The sectors may be chosen so that $a_1,b_1,a_2,b_2$ occur in
			this cyclic order and there is a sector $S_3$ with ends $b_1,a_2$
			such that every edge between $P^*$ and $C$ is one of the
			following types:
			$$b_1p_h\ (2\le h<i),\quad
			p_iq\ (q\in V(S_3)),\quad
			a_2p_h\ (i<h\le k-1).\label{modified-lsw-314-claim1-b}$$
		\end{enumerate}
	\end{claim}
	\begin{proof}[Proof of Claim~\ref{modified-lsw-314-claim1}]
	Assume that \ref{modified-lsw-314-claim1-a} fails.  First suppose that some $p_j\in P^*$ has a
	neighbor in the interior of a sector $R$.  Then $j=i$.  Indeed, if for
	example $i<j$, the induction hypothesis puts all attachments of
	$p_{i+1},\ldots,p_k$ in $R$, while it puts the attachments of
	$p_1,\ldots,p_j$ in two intersecting sectors, one of which must be $R$.
	Those same two sectors then cover all of $P$, a contradiction.  Thus
	$P^*-\{p_i\}$ is anticomplete to $C-N(x)$.
	
	Let $R^-,R^+$ be the two sectors immediately before and after $R$.  Apply
	the two-sector hypothesis to $p_1\cdots p_{k-1}$ and
	$p_2\cdots p_k$.  Since $P$ is a counterexample and $p_i$ attaches in
	$R^*$, the former pair is $R^-,R$ and the latter is $R,R^+$, after
	reflection if needed.  Hence $p_1$ has an attachment in
	$R^-\setminus R$, $p_k$ one in $R^+\setminus R$, and all other internal
	attachments are confined to the two boundary spokes of $R$, on the side
	dictated by their position relative to $p_i$.  Taking
	$S_1=R^-$, $S_2=R^+$, and $S_3=R$ proves \ref{modified-lsw-314-claim1-b}.
	
	It remains to consider the case in which all internal attachments are
	spokes.  If some $p_j\in P^*-\{p_i\}$ has a spoke neighbor, reverse $P$
	if necessary so that $j<i$, and call that neighbor $b_1$.
	Let $S_1,S_3$ be the sectors meeting at $b_1$, with
	the attachments of $p_1,\ldots,p_{i-1}$ in $S_1$.  Let $a_2$ be the other
	end of $S_3$ and $S_2$ the next sector at $a_2$.  The two-sector conclusion
	for $p_2\cdots p_k$, together with the fact that $P$ itself is a
	counterexample, forces that pair to be $S_3,S_2$.  It follows that every
	internal attachment before $p_i$ is at $b_1$, every one after $p_i$ is at
	$a_2$, and, when $1<i<k$, $N_C(p_i)\subseteq V(S_3)$.
	This proves \ref{modified-lsw-314-claim1-b} in this case.
	
	It remains that $p_i$ is the only internal vertex having a rim attachment.
	Thus $p_i\in P^*$ and has an attachment.  Every spoke
	neighbor of $p_i$ is a common neighbor of the adjacent vertices $x,p_i$;
	Lemma~\ref{lem:sparse-neighborhood} therefore says that \(p_i\) has a unique
	spoke neighbor, say $b_1$.  Let the two sectors at $b_1$ be
	$R_1,R_2$.  Write $c_h$ for the end of $R_h$ different from $b_1$, and
	let $R'_h$ be the other sector incident with $c_h$.  We claim that both
	$N_C(p_1)$ and $N_C(p_k)$ are contained in
	$V(R_1)\cup V(R_2)$.
	The following argument
	applies to either end
	$p_e\in\{p_1,p_k\}$.  If its attachments are not in this inner pair,
	the two-sector conclusion for the subpath between $p_e$ and $p_i$, and
	properness of $p_e$, give an index $h\in\{1,2\}$ such that \(N_C(p_e)\subseteq V(R'_h)\).
	Let $t_e$ be the $p_e$-neighbor in $R'_h$ closest to $c_h$, and let $d_h$
	be the other end of $R'_h$.  If $t_e\ne d_h$, the three paths
	\[
	p_ib_1,\quad p_ix,\quad p_iPp_e-t_eR'_hc_h
	\]
	link $p_i$ to the induced cycle $xR_hx$, a contradiction.
	
	It remains that $t_e=d_h$.  Then \(D_e=p_iPp_e-t_eR'_hc_hR_hb_1-p_i\) is a hole.  Indeed, $P$ is induced, the last-subcase assumption makes the
	interior of the displayed $P$-segment anticomplete to $C$, the extreme
	choice of $t_e$ excludes every chord from $p_e$ to the retained
	$C$-segment, and $C$ itself is chordless.  Moreover $(D_e,x)$ is an exact four-spoke wheel, with spokes
	$p_i,t_e,c_h,b_1$ in this order.  The adjacent center-neighbors
	$p_i,b_1$ form the short rim edge $p_ib_1$.  Its sector
	$p_iPp_e-t_e$, containing $p_i$ but not $b_1$, is the only part of the
	rim not contained in $C$.  Lemma~\ref{lem:modified-lsw-32}, applied with
	$u=p_i$, $v=b_1$, and this designated sector, gives a proper exact-four
	wheel.  This is impossible by global minimality.  Thus both end
	attachment sets lie in $R_1\cup R_2$.  Since every other internal vertex
	is anticomplete to $C$ and $N_C(p_i)=\{b_1\}$, the intersecting sectors
	$R_1,R_2$ cover every attachment of $P$, contradicting that $P$ is a
	counterexample.  Claim~\ref{modified-lsw-314-claim1} follows.
\end{proof}
	
	Let us call a choice of the two end sectors supplied by Claim~\ref{modified-lsw-314-claim1} a
	\textit{representation}.  We now normalize such a representation before
	making the only extremal choice used in the rest of the proof.  Orient
	and label its two end sectors so that
	$a_1,b_1,a_2,b_2$ occur in this order.  Let $Q_1$ be the
	$b_2$--$a_1$ path avoiding $b_1$ and $Q_2$ the $b_1$--$a_2$ path
	avoiding $a_1$; in case \ref{modified-lsw-314-claim1-b} of Claim~\ref{modified-lsw-314-claim1}, $Q_2=S_3$.  Properness gives
	$N_C(p_1)\subseteq V(S_1)$ and $N_C(p_k)\subseteq V(S_2)$.  Let $s,t$
	be the respective extreme neighbors of $p_1$ toward $a_1,b_1$, and let
	$y,z$ be those of $p_k$ toward $a_2,b_2$.  Let us call the representation
	\textit{normalized} when
	\[
	s\ne b_1\quad\hbox{and}\quad z\ne a_2.
	\]
	
	Every counterexample has a normalized representation.  In case \ref{modified-lsw-314-claim1-b} of Claim~\ref{modified-lsw-314-claim1}, if
	$s=b_1$, then $N_C(p_1)=\{b_1\}$ and Claim~\ref{modified-lsw-314-claim1}\ref{modified-lsw-314-claim1-b} shows that the
	intersecting sectors $S_3,S_2$ cover every attachment of $P$, a
	contradiction.  Similarly, $z=a_2$ would make $S_1,S_3$ a covering pair.
	In case \ref{modified-lsw-314-claim1-a} of Claim~\ref{modified-lsw-314-claim1}, $P^*$ is anticomplete to $C$.  If $s=b_1$, then
	$N_C(p_1)=\{b_1\}$.  Let $R$ be the sector contained in $Q_2$ and
	incident with $b_1$.  (The path $Q_2$ is nontrivial, since intersecting
	$S_1,S_2$ would already cover all attachments.)  If $R$ meets $S_2$,
	then $R,S_2$ cover the two nonempty end-attachment sets.  Otherwise
	replace $S_1$ by $R$ and relabel the old vertex $b_1$ as the new $a_1$;
	the unique $p_1$-attachment is now the $a_1$-end, so the first
	normalization inequality holds.  The symmetric inward move at $a_2$
	enforces the second inequality.  If the two inward replacement sectors
	meet, they themselves cover all attachments; otherwise both moves can be
	made and neither destroys the normalization at the other end.  Claim~\ref{modified-lsw-314-claim1}\ref{modified-lsw-314-claim1-a}
	is unchanged throughout.
	
	Among all counterexamples of order $k$ and all their normalized
	representations supplied by Claim~\ref{modified-lsw-314-claim1}, choose one for which $|V(Q_1)|$ is
	minimum, and retain the notation just fixed.  Thus
	$s\ne b_1$ and $z\ne a_2$.  Notice also that $S_1,S_2$ do not intersect:
	in case \ref{modified-lsw-314-claim1-a} of Claim~\ref{modified-lsw-314-claim1} they would cover all attachments because $P^*$ is
	anticomplete to $C$, while in case \ref{modified-lsw-314-claim1-b} of Claim~\ref{modified-lsw-314-claim1} the intervening sector $S_3$ and
	$\sigma(W)\ge5$ separate them.  The cycle
	\begin{equation}\label{eq:314-D1}
		D_1=a_1S_1s-p_1Pp_k-zS_2b_2Q_1a_1
	\end{equation}
	is a hole.  Replacing the part through $Q_2$ by $P$ deletes at least the
	two spokes $b_1,a_2$ and inserts only $p_i$, so $W_1=(D_1,x)$ has fewer
	spokes than $W$.  It has at least four spokes: otherwise its three spokes
	$a_1,p_i,b_2$ give an induced subdivision of $K_4$.  Thus
	$Q_1^*$ contains a spoke.  Let
	\[
	S_0=a_1S_1s-p_1Pp_i,\quad
	T_0=p_iPp_k-zS_2b_2
	\]
	be the two sectors of $W_1$ incident with $p_i$.
	
	We record where the old rim vertices deleted in forming $D_1$ can attach
	to the new rim.  For every $c\in V(C)\setminus V(D_1)$, all neighbors of
	$c$ on $D_1$ are contained in $S_0$, or all are contained in $T_0$, or
	the only such neighbor is the common spoke $p_i$.  Indeed, on the deleted
	$S_1$-side the retained cyclic boundary neighbor and every possible
	$P$-neighbor lie in $S_0$; the symmetric statement holds on the deleted
	$S_2$-side.  In case \ref{modified-lsw-314-claim1-a} of Claim~\ref{modified-lsw-314-claim1}, $P^*$ is anticomplete to $C$.  Here
	$S_1,S_2$ cannot meet, since otherwise those two sectors already cover
	all attachments of $P$; in particular $b_1\ne a_2$.  A vertex of
	$Q_2^*$ has at most a retained cyclic boundary neighbor.  The
	endpoint $b_1$ has all its $D_1$-neighbors in
	$\{s,p_1\}\subseteq V(S_0)$, and symmetrically all $D_1$-neighbors of
	$a_2$ lie in $\{p_k,z\}\subseteq V(T_0)$.  In case \ref{modified-lsw-314-claim1-b} Claim~\ref{modified-lsw-314-claim1}, a
	vertex of $S_3^*$ can see on $P$ only $p_i$, while all $D_1$-neighbors
	of $b_1$ lie in $S_0$ and all those of $a_2$ lie in $T_0$.  Every deleted
	vertex other than $b_1,a_2$ consequently has at most two $D_1$-neighbors;
	$b_1,a_2$ are old spokes and are adjacent to $x$.  Thus every vertex of
	$C\setminus D_1$ is proper for $W_1$, and none has neighbors in both
	$S_0\setminus T_0$ and $T_0\setminus S_0$.  In particular, every
	nonproper vertex or skip for $W_1$ lies outside $V(W)$.

	\begin{claim}\label{modified-lsw-314-claim2}
		No vertex $v\notin V(W_1)$ has both a neighbor in $V(S_0)\setminus V(T_0)$ and a neighbor in $V(T_0)\setminus V(S_0)$.
	\end{claim}
	\begin{proof}[Proof of Claim~\ref{modified-lsw-314-claim2}]
	Suppose that such a $v$ exists.  The preceding observation puts
	$v\notin V(W)$, so it is proper for $W$.  First assume
	$vx\notin E(G)$.  If $v$ sees $a_1S_1s$, properness makes it anticomplete
	to $zS_2b_2$, so it has a neighbor $p_j$ with $j>i$; choose $j$ maximum.
	The no-$x$-neighbor lemma applied to $v-p_jPp_k$ (whose order is at most
	$k$) would put its attachments in one sector, contrary to the exclusive
	$S_1,S_2$ attachments.  Hence $v$ has no rim neighbor on either truncated
	side and all its $S_0\cup T_0$ neighbors lie on $P$.  Let $p_j,p_l$ be
	its first and last $P$-neighbors.  Then $j<i<l$, and the path obtained
	from $P$ by replacing $p_jPp_l$ with $p_j-v-p_l$ contains no neighbor of
	$x$, has order at most $k$, and has the same end attachments.  The
	one-sector conclusion again contradicts the choice of $P$.  Therefore
	$vx\in E(G)$.
	
	Suppose next that $N_{D_1}(v)\subseteq V(S_0)\cup V(T_0)$.  Applying
	Lemma~\ref{lem:adjacent-sector-multiplicity} to the two sectors incident
	with $p_i$ shows that $v$ has at least three neighbors in each of
	$S_0,T_0$: each exclusive side has at least two, and if either bound is an
	equality the common spoke $p_i$ is also adjacent to $v$.  If $v$ sees
	$a_1S_1s$, it has at least three neighbors on $p_iPp_k$.  Let $p_r$ be
	its last neighbor there.  Then $r\ge i+2$, and $L=v-p_rPp_k$ is
	induced, has order less than $k$, and has $v$ as its unique neighbor of
	$x$.  Choose $w\in N_C(v)\cap V(a_1S_1s)$.  Thus
	$w\in V(S_1)\setminus\{b_1\}$, while
	$z\in N_C(p_k)\cap(V(S_2)\setminus\{a_2\})$.  Apply induction to $L$ and Lemma~\ref{lem:sector-carrier-facts}\ref{carrier-facts-2}
	to $\{w\},\{z\}$.  They give $w=a_1$, $z=b_2$, and a covering pair
	consisting of the two sectors of $Q_1$, with unique internal spoke $q$.
	Since that pair contains every attachment of $L$, the standing inclusion
	$N_{D_1}(v)\subseteq V(S_0)\cup V(T_0)$ and
	$N_C(p_k)\subseteq V(S_2)$ further give
	\[
	N_C(v)=\{a_1\},\quad N_C(p_k)=\{b_2\}.
	\]
	In this exceptional case
	\[
	H=v-p_rPp_k-b_2Q_1a_1-v
	\]
	is a hole and $(H,x)$ is an exact four-spoke wheel, with spokes
	$v,b_2,q,a_1$.  Indeed, the choice of $p_r$ excludes further
	$v$--$P$ edges on the displayed segment; the two singleton attachment
	sets exclude endpoint chords; in Claim~\ref{modified-lsw-314-claim1}\ref{modified-lsw-314-claim1-a}, $P^*$ is anticomplete
	to $C$, while in Claim~\ref{modified-lsw-314-claim1}\ref{modified-lsw-314-claim1-b} every $C$-neighbor of $p_h$ with
	$i<h<k$ is $a_2\notin V(Q_1)$; and $C,P$ are induced.  The adjacent
	spokes $v,a_1$ form the short sector $va_1$, and the sector
	$v-p_rPp_k-b_2$ containing $v$ but not $a_1$ is the only part of $H$
	outside $C$.  Lemma~\ref{lem:modified-lsw-32} contradicts global
	spoke-minimality.  Thus induction cannot have the asserted covering
	pair.  The symmetric argument excludes a neighbor of $v$ on
	$zS_2b_2$.
	
	Consequently all neighbors of $v$ on $D_1$ lie on $P$.  Necessarily
	$1<i<k$.  Let $p_j,p_l$ be the first and last $P$-neighbors of $v$.
	Multiplicity gives \(j\le i-2<i+2\le l\).
	The shortcut \(L=p_1Pp_j-v-p_lPp_k\) is induced, has order less than $k$, and has $v$ as its unique neighbor
	of $x$.  It retains the attachments
	$s\in V(S_1)\setminus\{b_1\}$ and
	$z\in V(S_2)\setminus\{a_2\}$.  If its attachments are not covered by
	two intersecting sectors, induction gives the desired contradiction.
	By Lemma~\ref{lem:sector-carrier-facts}\ref{carrier-facts-2}, applied to
	$N_C(p_1),N_C(p_k)$, such a cover is possible only if
	\begin{equation}\label{eq:314-claim2-boundary}
		N_C(p_1)=\{a_1\},\quad N_C(p_k)=\{b_2\},
	\end{equation}
	and $Q_1$ consists of two sectors having a unique common spoke
	$q\in V(Q_1^*)$.  Moreover $N_C(v)=\emptyset$: all $D_1$-neighbors
	of $v$ lie on $P$, whereas the covering pair is contained in
	$Q_1\subseteq D_1$.  Claim~\ref{modified-lsw-314-claim1} also shows that every retained internal
	vertex of $L$ is anticomplete to $C$.  Hence every edge from $P^*$
	to $C$ has its $P$-end strictly between $p_j$ and $p_l$.
	
	We show that even this boundary case is impossible.  The two sectors of
	$Q_1$ already cover the attachment sets in
	\eqref{eq:314-claim2-boundary}; since $P$ is a counterexample,
	$P^*$ has a neighbor on $C$.  Thus Claim~\ref{modified-lsw-314-claim1}\ref{modified-lsw-314-claim1-b} holds.  Suppose first
	that some $p_h$ with $1<h<i$ has a rim neighbor, and
	choose $h$ least.  Claim~\ref{modified-lsw-314-claim1}\ref{modified-lsw-314-claim1-b} gives $N_C(p_h)=\{b_1\}$ and
	$j<h<i$; in particular $p_h\in P^*$.  Thus \(R=a_1-p_1Pp_h-b_1S_1a_1\) is a hole.  Apply
	Lemma~\ref{lem:external-connector-obstruction} with
	$a=a_1$, $b=b_1$, $A=S_1$, $B=a_1-p_1Pp_h-b_1$, $w=v$,
	and $Q=xv$.  The vertex $x$ sees exactly $a_1,b_1$ on $R$ because
	$h<i$; moreover $N_C(v)=\emptyset$, and $vp_j\in E(G)$ supplies
	an attachment in $B^*$.  These facts verify every hypothesis.
	The first two outcomes are forbidden; the third, by \textbf{(N4)}, contradicts
	global spoke-minimality.  The symmetric argument, choosing
	the greatest attaching index after $i$ and using $S_2$, excludes every
	attachment after $p_i$.
	
	It remains that $N_C(p_i)\ne\emptyset$ and
	$N_C(P^*-\{p_i\})=\emptyset$.  Let
	$c\in N_C(p_i)$ be closest to $b_1$ along $S_3$.  If $c\ne a_2$, then \(p_i-cS_3b_1S_1a_1-p_1Pp_i\) is a hole on which $x$ has exactly the three neighbors
	$p_i,b_1,a_1$.  If $c=a_2$, extremality gives
	$N_C(p_i)=\{a_2\}$, and \(p_iPp_k-b_2S_2a_2-p_i\) is a hole on which $x$ has exactly the three neighbors
	$p_i,b_2,a_2$.  The extreme choice of $c$, the absence of every other
	internal attachment, and \eqref{eq:314-claim2-boundary} verify
	inducedness.  Both are forbidden three-spoke wheels.  Thus the boundary
	case cannot occur, and the induction hypothesis is contradicted.
	
	Consequently $v$ has a neighbor in a sector $S_4\subseteq Q_1$ outside
	$S_0\cup T_0$.  By properness, choose $S_4$ to be a carrier of
	$N_C(v)$ containing such a neighbor.
	
	Since $v$ is proper for $W$, it sees at most one of $a_1,b_2$; assume
	$vb_2\notin E(G)$.  Suppose $va_1\in E(G)$.  Diamond- and
	$K_4$-freeness give $vp_i\notin E(G)$.  Let $p_j$ be its last neighbor
	on $P$; then $j>i$.  If $j>2$, induction applied to $v-p_jPp_k$ forces
	$y=b_2$ and forces $Q_1^*$ to contain exactly one spoke $a_3$.
	Indeed, this shorter path contains exactly the $x$-neighbor $v$.  A pair
	of intersecting sectors covering its attachments must contain $S_4$.
	Here $a_1\in V(S_4)$ because $va_1\in E(G)$ and
	$N_C(v)\subseteq V(S_4)$, while the choice of $S_4$ in the preceding
	paragraph supplies a second $v$-attachment in
	$S_4\setminus\{a_1\}$; hence $S_4$ is the unique sector carrying
	$N_C(v)$.  The other covering sector must
	both meet $S_4$ and cover the exclusive $S_2$-attachment of $p_k$.  In
	the cyclic order along $Q_1$ this is possible only when $S_4$ meets the
	sector $S_5$ incident with $b_2$ at a single spoke $a_3$ and
	$N_C(p_k)\cap V(S_2)=\{b_2\}$.  Thus $S_4,S_5$ are the only sectors of
	$Q_1$, and $a_3$ is its unique internal spoke.  If
	$v'$ is the neighbor of $v$ in $S_4$ closest to $a_3$ and $S_5$ is the
	sector from $a_3$ to $b_2$, then $x$ is linked to \(v-v'S_4a_3S_5b_2-p_kPp_j-v\) by $xv,xa_3,xb_2$.  Hence $j=2$ and $i=1$.  The edge $p_1v$ is impossible:
	$\{x,v,a_1,p_1\}$ induces a diamond if $a_1p_1\notin E(G)$ and a $K_4$ if
	$a_1p_1\in E(G)$.  Now $x$ is linked to
	$v-p_2-p_1-sS_1a_1-v$ by $xv,xa_1,xp_1$.  This contradiction proves
	$va_1\notin E(G)$, and symmetrically $vb_2\notin E(G)$.
	
	Thus all neighbors of $v$ in $S_0\cup T_0$ lie on $P$.  If $p_j,p_l$
	are its first and last $P$-neighbors, then $j<i<l$.  Suppose first that
	$l-j\ge3$.  The shortcut \(L=p_1Pp_j-v-p_lPp_k\) is a shorter induced path with $v$ as its unique neighbor of $x$.
	Apply the induction hypothesis.  The normalized end attachments,
	together with the attachment of $v$ in
	$S_4\setminus(V(S_0)\cup V(T_0))$, force the only possible covering
	pair to be the two sectors of $Q_1$.  More explicitly,
	\[
	N_C(p_1)=\{a_1\},\quad N_C(p_k)=\{b_2\},
	\]
	$Q_1$ consists of exactly two sectors with common internal spoke $q$,
	and $S_4$ is one of them.  Recall that
	$va_1,vb_2\notin E(G)$.  If $S_4$ is the $a_1$--$q$ sector, choose
	$r\in N_C(v)\cap V(S_4)$ closest to $q$.  Then \(v-p_lPp_k-b_2Q_1qS_4r-v\) is a hole on which $x$ has exactly the three neighbors $v,b_2,q$.
	If $S_4$ is the $q$--$b_2$ sector, the symmetric hole \(v-p_jPp_1-a_1Q_1qS_4r-v\), where again $r$ is chosen closest to $q$, has precisely the three
	$x$-neighbors $v,a_1,q$.  Both are forbidden three-spoke wheels.  For
	inducedness, the extreme choices of $p_j,p_l$ exclude further
	$v$--$P$ edges; Claim~\ref{modified-lsw-314-claim1} and the two singleton end attachments exclude
	all chords from the retained $P$-segment to $Q_1$; the choice of $r$
	excludes further $v$--$C$ edges on the retained sector segment; and the
	retained $P$-segment avoids $p_i$.  Thus $l-j=2$, and necessarily
	$j=i-1$ and $l=i+1$.  We have $vp_i\notin E(G)$, for otherwise
	$G[\{v,p_{i-1},p_i,p_{i+1}\}]$ is a diamond.  Hence
	$v-p_{i-1}-p_i-p_{i+1}-v$ is a hole, and $x$ is linked to it by
	$xv,xp_i,xa_1S_1s-p_1Pp_{i-1}$.  This proves Claim~\ref{modified-lsw-314-claim2}.
	\end{proof}
  
	\begin{claim}\label{modified-lsw-314-claim3}
		Every skip for $W_1$ is an $a_1$-skip or a
		$b_2$-skip.
	\end{claim}
	\begin{proof}[Proof of Claim~\ref{modified-lsw-314-claim3}]
	Indeed, the preceding old-rim check puts the skip outside $V(W)$, and
	properness for $W$ puts its old rim attachments into one
	sector of $W$.  Hence the only possible junctions on $W_1$ are
	$a_1,b_2,p_i$; Claim~\ref{modified-lsw-314-claim2} excludes a $p_i$-skip.
	\end{proof}
	
	Let $X$ be the set of all skips for $W_1$.  The spokes $a_1,b_2$ are
	nonconsecutive in $W_1$: $Q_1^*$ contains a spoke, while the other
	arc contains $p_i$.  If $W_1$ were proper in $G-X$, it would be
	two-almost proper.  Lemma~\ref{lem:two-skip-repair} would then give a
	proper wheel with at most $\sigma(W_1)<\sigma(W)$ spokes, contrary to
	global minimality.  Thus $W_1$ is not proper in $G-X$.

	\begin{claim}\label{modified-lsw-314-claim4}
		There is
		$v\notin V(W_1)\cup V(W)\cup X$ satisfying one of:
		\begin{enumerate}[label=(\alph*)]
			\item $vx\notin E(G)$, $|N(v)\cap V(S_0)|\ge3$, and
			$N_{D_1}(v)\subseteq V(S_0)$;\label{modified-lsw-314-claim4-a}
			\item $vx\notin E(G)$, $|N(v)\cap V(T_0)|\ge3$, and
			$N_{D_1}(v)\subseteq V(T_0)$;\label{modified-lsw-314-claim4-b}
			\item for a sector $S_4\subseteq Q_1$, the vertex $v$ has a neighbor
			in $V(S_4)\setminus(V(S_0)\cup V(T_0))$, a neighbor in
			$V(S_0)\setminus V(S_4)$, no neighbor in
			$V(T_0)\setminus(V(S_0)\cup V(S_4))$, and
			$N_C(v)\subseteq V(S_4)$; the roles of $S_0,T_0$ may be exchanged.\label{modified-lsw-314-claim4-c}
		\end{enumerate}
	\end{claim}
	\begin{proof}[Proof of Claim~\ref{modified-lsw-314-claim4}]
		Choose a nonproper vertex of $W_1$ outside $X$.  The preceding old-rim
		check puts it outside $V(W)$.  If neither \ref{modified-lsw-314-claim4-a} nor \ref{modified-lsw-314-claim4-b} holds, properness
		for $W$, nonproperness for $W_1$, and Claim~\ref{modified-lsw-314-claim2} give a
		sector $S_4\subseteq Q_1$ and, after symmetry, exactly the distribution
		in \ref{modified-lsw-314-claim4-c}. This proves the claim.
	\end{proof}

	Fix such a vertex $v$ as in Claim~\ref{modified-lsw-314-claim4}.
	A sector that satisfies Claim~\ref{modified-lsw-314-claim4}\ref{modified-lsw-314-claim4-c} will be called an \textit{admissible carrier} for $N_C(v)$.
	
	\noindent\textbf{Return-path observation.}
	Suppose that Claim~\ref{modified-lsw-314-claim4}\ref{modified-lsw-314-claim4-c} holds, and let $p_j$ be the first neighbor of $v$
	on $P$.  Let $T\subseteq Q_1$ be the sector incident with $a_1$, with
	other end $b_3$.  No admissible carrier $B$ with
	$a_1\notin V(B)$ and $V(B)\cap V(T)=\{b_3\}$ can exist in either of
	the following situations:
	\begin{enumerate}[label=(\roman*)]
		\item $j<i$ and $xv\in E(G)$;\label{modified-lsw-314-claim4-+i}
		\item $j=1$, $xp_1\in E(G)$, $xv\notin E(G)$, and
		$N_C(p_1)=\{a_1\}$.\label{modified-lsw-314-claim4-+ii}
	\end{enumerate}
	Indeed, let $a_3$ be the other end of $B$ and choose the neighbor $r$ of
	$v$ in $B$ closest to $b_3$.  The cycle
	\[
	D=p_1Pp_j-v-rBb_3Ta_1S_1s-p_1
	\]
	is a hole.  The first-neighbor choice of $p_j$ excludes additional
	$v$--$P$ edges, the extreme choices of $s,r$ exclude endpoint chords,
	and Claim~\ref{modified-lsw-314-claim1} puts every possible attachment of the internal prefix
	vertices at the omitted spoke $b_1$.  In \ref{modified-lsw-314-claim4-+ii} the prefix consists only
	of $p_1$.  Moreover $p_1,v,b_3,a_1$ are four distinct vertices of $D$.
	If $r\ne a_3$, the only neighbors of $x$ on $D$ are $a_1,b_3$ and
	$v$ in \ref{modified-lsw-314-claim4-+i}, or $a_1,b_3$ and $p_1$ in \ref{modified-lsw-314-claim4-+ii}, giving a forbidden
	three-spoke wheel.  If $r=a_3$, then $(D,x)$ is an exact four-spoke wheel.
	In \ref{modified-lsw-314-claim4-+i}, apply Lemma~\ref{lem:modified-lsw-32} to the adjacent spokes
	$v,a_3$ and the designated sector $v-p_jPp_1-sS_1a_1$.
	In \ref{modified-lsw-314-claim4-+ii}, $s=a_1$; use the adjacent spokes $p_1,a_1$ and the designated
	sector $p_1-v-a_3$.  In each case every rim vertex outside the designated
	sector belongs to $C$, so the lemma contradicts global spoke-minimality.

	\begin{claim}\label{modified-lsw-314-claim5}
		The vertex $v$ has a unique neighbor on $P$.
	\end{claim}
	\begin{proof}[Proof of Claim~\ref{modified-lsw-314-claim5}]
	Assume that $v$ has at least two neighbors on $P$.  We first record an
	endpoint-safe form of the shortcut argument.  Let
	$w\notin V(W)\cup V(P)$ be proper for $W$, let $wx\notin E(G)$, and
	suppose that every $P$-neighbor of $w$ lies on $p_1Pp_i$.  Then
	\begin{equation}\label{eq:314-one-sided-shortcut}
		|N(w)\cap V(P)|\le 2.
	\end{equation}
	Indeed, suppose that $w$ has at least three $P$-neighbors, and let
	$p_j,p_l$ be its first and last such neighbors.  The shortcut assertion of Lemma~\ref{lem:sparse-neighborhood}
	gives $l-j\ge3$.  Hence \(L=p_1Pp_j-w-p_lPp_k\) is an induced path of order less than $k$ and has $p_i$ as its unique
	neighbor of $x$.  The induction hypothesis gives two intersecting
	sectors covering all attachments of $L$ to $C$.
	
	Apply Lemma~\ref{lem:sector-carrier-facts}\ref{carrier-facts-2} to the retained
	end-attachment sets, using normalization and the spoke in $Q_1^*$.
	It gives
	\begin{equation}\label{eq:314-one-sided-boundary}
		N_C(p_1)=\{a_1\},\quad N_C(p_k)=\{b_2\},
	\end{equation}
	and $Q_1$ consists of exactly two sectors, which are the covering
	pair.  Write $q$ for their unique common spoke in $Q_1^*$.  The covering
	pair also contains $N_C(w)$.  Moreover every retained internal vertex of $P$ that
	occurs on $L$ is anticomplete to $C$: this is immediate in Claim~\ref{modified-lsw-314-claim1}\ref{modified-lsw-314-claim1-a}, while in
	Claim~\ref{modified-lsw-314-claim1}\ref{modified-lsw-314-claim1-b} every possible internal attachment lies at $b_1$, at $a_2$,
	or in $S_3$, all outside the covering pair.
	
	Since $P$ is a counterexample whereas its end attachments are covered
	by the pair in \eqref{eq:314-one-sided-boundary}, some vertex of
	$p_{j+1}Pp_{l-1}$ attaches to $C$.  Thus Claim~\ref{modified-lsw-314-claim1}\ref{modified-lsw-314-claim1-b} holds.  Choose
	$h$ least with $j<h<l$ and $N_C(p_h)\ne\emptyset$.
	Then $h<i$, $N_C(p_h)=\{b_1\}$, and \(R=a_1-p_1Pp_h-b_1S_1a_1\) is a hole.  Indeed, the retained internal vertices of the shortcut
	$L$ are anticomplete to $C$, and the choice of $h$ excludes every
	other attachment before $p_h$.  If $h=l-1$, then
	$xa_1,xb_1,x-p_iPp_l-p_h$ link $x$ to $R$, a contradiction.
	Hence $h\le l-2$.
	
	Apply Lemma~\ref{lem:external-connector-obstruction} with
	$a=a_1$, $b=b_1$, $A=S_1$, $B=a_1-p_1Pp_h-b_1$, and
	$Q=x-p_iPp_l-w$.  The choice of $p_l$ as the last $P$-neighbor
	of $w$ makes $Q$ induced.  Every internal vertex of $Q$ other than
	a possible $p_k$ is a retained internal vertex of $L$ and hence is
	anticomplete to $C$.  If $p_k$ occurs, then
	$N_C(p_k)=\{b_2\}$ and $b_2\notin V(R)$.  Thus $Q^*$ has
	no edge to $R\cap C$; the inequality $h\le l-2$ and inducedness
	of $P$ exclude its edges to $p_1Pp_h$.  Hence $Q^*$ is
	anticomplete to $R$.  Also $N_R(x)=\{a_1,b_1\}$ because $h<i$.
	The covering pair gives $N_C(w)\subseteq V(Q_1)$, so
	$N_R(w)\subseteq V(B)\setminus\{b_1\}$, and its neighbor $p_j$
	lies in $B^*$.  All hypotheses hold.  The first two outcomes
	are forbidden, and the third, by \textbf{(N4)}, contradicts global
	spoke-minimality.  This proves \eqref{eq:314-one-sided-shortcut}.  By reversing $P$ and interchanging
	$(S_1,a_1,b_1)$ with $(S_2,b_2,a_2)$, the symmetric assertion holds when
	all $P$-neighbors of $w$ lie on $p_iPp_k$.
	
	First suppose  Claim~\ref{modified-lsw-314-claim4}\ref{modified-lsw-314-claim4-a} holds; case  Claim~\ref{modified-lsw-314-claim4}\ref{modified-lsw-314-claim4-b} is symmetric.  The one-sided shortcut
	observation gives exactly two $P$-neighbors.  Since $v$ has at least
	three neighbors in $S_0$, it has a neighbor on $a_1S_1s$.  Properness
	and $vx\notin E(G)$ give $1\le |N_C(v)|\le2$.  If only one rim neighbor
	lay on $D_1$, then $(D_1,v)$ would be a forbidden three-spoke wheel.
	Thus $v$ has exactly two rim neighbors $v_1,v_2$, both on $a_1S_1s$.
	Let $p_j$ be its last $P$-neighbor.  If $j\ge3$, then $v-p_jPp_k$ is a
	shorter induced path having $p_i$ as its unique neighbor of $x$.  The
	induction hypothesis would require two intersecting sectors to cover its
	attachments.  The two distinct $S_1$-attachments of $v$ force one of
	them to be $S_1$.  A sector meeting $S_1$ can carry $N_C(p_k)$ only
	through $a_2$ along $Q_2$, which would force
	$N_C(p_k)=\{a_2\}$ contrary to $z\ne a_2$, or through $b_2$ along
	$Q_1$, which would force $Q_1$ to have no internal spoke.  Both are
	impossible.  Hence $j=2$ and $N_P(v)=\{p_1,p_2\}$.  Neither $v_1$ nor
	$v_2$ belongs to
	$\{s,t\}$: otherwise $\{v,p_1,p_2,s\}$ or the analogous four-vertex set induces a
	diamond or a $K_4$.  Relabel $v_1,v_2$ in their order from $a_1$
	toward $s$, let $A=v_1S_1v_2$, and let $B$ be its complementary
	arc on $C$.  Both $v_1,v_2$ strictly precede $s$, whereas every
	neighbor of $p_1$ on $C$ lies on $sS_1t$.  Hence \(\emptyset\ne N_C(p_1)\subseteq V(B^*)\).
	Apply Lemma~\ref{lem:external-connector-obstruction} with $R=C$,
	$Q=v-p_1$, and rim ends $a=v_1$, $b=v_2$.
	The path has no internal vertex, and the exceptional conclusion would
	require a common rim neighbor of $v$ and $p_1$, which the displayed
	attachment sets exclude.  Thus an induced subdivision of $K_4$ or a
	prism results, a contradiction.  Consequently  Claim~\ref{modified-lsw-314-claim4}\ref{modified-lsw-314-claim4-c} holds after symmetry
	and $N_P(v)\subseteq V(S_0)$.
	
	Let $p_j,p_l$ be the first and last $P$-neighbors of $v$; then
	$j<l\le i$.  Suppose $vx\notin E(G)$.  It has exactly two such neighbors.
	Apply the no-$x$-neighbor conclusion to
	$L=p_1Pp_j-v$, whose order is $j+1\le k$ and which avoids $p_i$.  The
	no-$x$-neighbor conclusion puts all rim attachments of $L$ in one sector
	$R$.  This sector contains an attachment of $p_1$ in $S_1$ and an
	attachment of $v$ in $Q_1\setminus(V(S_0)\cup V(T_0))$.
	Consequently $R$ meets $S_1$ at $a_1$.  Replacing the carrying sector
	$S_4$ in  Claim~\ref{modified-lsw-314-claim4}\ref{modified-lsw-314-claim4-c} by $R$ if necessary, we may assume
	$a_1\in V(S_4)$; the replacement preserves the outside attachment and
	all three incidence conditions in  Claim~\ref{modified-lsw-314-claim4}\ref{modified-lsw-314-claim4-c}.  Moreover
	\[
	N_C(p_1)=\{a_1\},\quad s=t=a_1.
	\]
	Let $v'$
	be the neighbor of $v$ in $S_4$ closest to $a_1$.  If $v'$ is not the
	other end of $S_4$, then $p_l$ is linked to the induced cycle \(p_1Pp_j-v-v'S_4a_1-p_1\) by $p_lPp_j,p_lv,p_lPp_i-x-a_1$.  This includes the triangular case
	$j=1$ and $v'=a_1$.  Hence $v'$ is the other end of $S_4$; its extreme choice gives
	$N_C(v)=\{v'\}$.  Put
	\[
	R=p_1Pp_i-x-a_1-p_1,\quad A=p_jPp_l,
	\]
	and let $B$ be the complementary $p_j$--$p_l$ arc of $R$.
	The cycle $R$ is a hole: $j<l\le i$ gives $i\ge2$, the only
	$x$-neighbor on $P$ is $p_i$, and  Claim~\ref{modified-lsw-314-claim1} excludes every
	$a_1$--$P$ edge other than $a_1p_1$.
	The spoke $v'$ lies in $Q_1^*$, so it is outside $R$, and
	 Claim~\ref{modified-lsw-314-claim1} also gives
	\[
	N_R(v)=\{p_j,p_l\},\quad
	\{x\}\subseteq N_R(v')\subseteq\{x,a_1\}
	\subseteq V(B^*).
	\]
	Apply Lemma~\ref{lem:external-connector-obstruction} with
	$Q=v-v'$, $a=p_j$, and $b=p_l$.  Its exceptional outcome would
	require a common rim neighbor of $v$ and $v'$, which the displayed
	sets exclude.  Thus an induced subdivision of $K_4$ or a prism
	results, a contradiction.  Therefore $vx\in E(G)$.
	
	Choose the carrier $S_4$ in  Claim~\ref{modified-lsw-314-claim4}\ref{modified-lsw-314-claim4-c} to contain $a_1$ whenever
	possible.  Suppose that no admissible carrying sector contains $a_1$.
	Let $S_5\subseteq Q_1$ be the sector incident with $a_1$, with other end
	$b_3$.  Because $v$ has at least two
	$P$-neighbors and both $v,p_i$ lie in $N(x)$, sparse neighborhoods exclude
	the simultaneous pair $vp_{i-1},vp_i$; hence $j\le i-2$.  The path
	$p_1Pp_j-v$ therefore has order less than $k$ and has the unique
	$x$-neighbor $v$.  By Lemma~\ref{lem:sector-carrier-facts}\ref{carrier-facts-1}, one member $B$ of
	its two-sector cover contains all of $N_C(v)$.  The sector $B$ may
	replace $S_4$ in  Claim~\ref{modified-lsw-314-claim4}\ref{modified-lsw-314-claim4-c}, so $a_1\notin V(B)$.  The other member
	cannot be $S_1$.  Normalization therefore forces
	$N_C(p_1)=\{a_1\}$ (so $t=a_1$), the other member is $S_5$, and
	$B\cap S_5=\{b_3\}$.  The return-path observation~(i), with
	$T=S_5$, gives a contradiction, since $j\le i-2$ and $xv\in E(G)$.
	Hence an admissible carrying sector containing
	$a_1$ exists.  Choose it as $S_4$, and let $b_3$ be its other end.
	The vertex $v$ has a neighbor in $S_4^*$, for otherwise
	$vb_3\in E(G)$ and the paths $vx,vb_3$ and either $va_1$ or
	$vp_jPp_1sS_1a_1$ link $v$ to the induced cycle $xS_4x$.  Let $S_5$ be the next sector after $S_4$
	at $b_3$, and let $a_3$ be the other end of $S_5$.
	
	At this point $S_0,S_4$ are consecutive sectors of $W_1$ meeting at
	$a_1$.   Claim~\ref{modified-lsw-314-claim4}\ref{modified-lsw-314-claim4-c},  Claim~\ref{modified-lsw-314-claim2}, $N_P(v)\subseteq V(S_0)$, and
	$N_C(v)\subseteq V(S_4)$ show that $v$ satisfies skip conditions
	\ref{skip-c-1}--\ref{skip-c-3} for this pair.  Because $v\notin X$, it is not a skip; hence
	condition~\ref{skip-c-4} fails and supplies
	$u\notin V(W_1)$ adjacent to $v$ and having a neighbor in
	\[
	V(D_1)\setminus(V(S_4)\cup V(S_0)).
	\]
	In fact $u\notin V(W)$.  It is not $x$, since $x\in V(W_1)$; and if it
	were an old rim vertex, the edge $uv$ and
	$N_C(v)\subseteq V(S_4)$ would put $u$ in $S_4\subseteq D_1$, again
	contrary to $u\notin V(W_1)$.
	We claim $ux\notin E(G)$.  Suppose otherwise.  If $u$ has a rim neighbor, Lemma~\ref{lem:modified-lsw-34} puts $N_C(u)\cup N_C(v)$ in one sector.
	Because $v$ sees $S_4^*$, that sector is $S_4$; the same containment
	$N_C(u)\subseteq V(S_4)$ is trivial if $u$ has no rim neighbor.  Hence the
	displayed witness is $p_h$ for some $h>i$.  Sparse neighborhoods applied to
	the edge $uv$ inside $N(x)$ give \(up_i,vp_i\notin E(G)\).
	 Claim~\ref{modified-lsw-314-claim2} excludes every edge $up_q$ with $q<i$, since $u$ already sees the
	$T_0\setminus S_0$ side.  If $p_l$ is the last $P$-neighbor of $v$ and
	$p_r$ the first of $u$, then $l<i<r$, and \(v-p_lPp_r-u-v\) is a hole on which $x$ has exactly the three neighbors $v,p_i,u$.  This
	is a forbidden three-spoke wheel.  Therefore $ux\notin E(G)$.
	
	We next locate the remaining attachments of $u$.  First $i\ge3$.  Indeed,
	if $i\le2$, the two distinct $P$-neighbors of $v$ force $i=2$ and
	$N_P(v)=\{p_1,p_2\}$; because $p_i$ is the unique $x$-neighbor on $P$,
	$G[\{x,v,p_1,p_2\}]$ is a diamond.  Hence the edge $uv$, viewed as a path
	of order two, is shorter than $P$ and contains exactly one neighbor of
	$x$, namely $v$.  The induction hypothesis for $uv$, properness of $u$,
	and the fact that $v$ has a neighbor in $S_4^*$ imply that every
	$C$-neighbor of $u$ lies in $S_4$ or in one of the two sectors adjacent
	to $S_4$.
	
	Suppose that $u$ has a neighbor in $Q_1\setminus S_4$.  Then all its
	$C$-neighbors lie in the sector $S_5$ adjacent to $S_4$ at $b_3$.  Since
	$ux\notin E(G)$, properness gives $1\le |N_C(u)|\le2$.  Let $v'$ be the
	neighbor of $v$ in $S_4$ closest to $b_3$.  If $u$ has one neighbor $r$
	in $S_5$, then the three paths $vx$, $vv'S_4b_3$, and $v-u-r$ link $v$
	to the induced cycle $xS_5x$.  Notice that $r\ne b_3$ because the chosen
	neighbor lies outside $S_4$.
	It remains that $N_C(u)=\{r_1,r_2\}$, labeled in their order
	from $b_3$ toward $a_3$ on $S_5$.  The sector $S_5$ is long, since otherwise
	$G[\{x,u,b_3,a_3\}]$ is a diamond, its only missing edge being $xu$.
	Set $R=xS_5x$, let $A=r_1S_5r_2$, and let $B$ be the complementary
	$r_1$--$r_2$ arc of $R$.  Apply
	Lemma~\ref{lem:external-connector-obstruction} with outside path
	$Q=u-v$ and rim ends $a=r_1$, $b=r_2$.  Indeed,
	$N_R(u)=\{r_1,r_2\}$ because $ux\notin E(G)$,
	whereas $\{x\}\subseteq N_R(v)\subseteq\{x,b_3\}$
	because $vx\in E(G)$ and $N_C(v)\subseteq V(S_4)$.
	Thus $N_R(v)\subseteq V(B)\setminus\{r_2\}$ and $x\in V(B^*)$;
	the path $Q$ has no internal vertex.  All hypotheses hold, so the lemma
	and \textbf{(N4)} contradict global spoke-minimality.  Hence \(N_C(u)\cap(V(Q_1)\setminus V(S_4))=\emptyset\).
	
	The witness furnished by the failure of the skip condition lies outside
	$S_4\cup S_0$.  In the preceding two-sector conclusion the sectors
	adjacent to $S_4$ are $S_1$ at $a_1$ and $S_5$ at $b_3$.  The $S_5$
	possibility has just been excluded, while every vertex of $D_1\cap S_1$
	lies in $S_0$.  Thus the witness cannot lie on $C$, and therefore it is
	$p_h$ for some $h>i$.  Claim~\ref{modified-lsw-314-claim2} then excludes
	every neighbor of $u$ in $S_0\setminus T_0$.  In particular,
	\[
	N_C(u)\cap V(D_1)\subseteq V(S_4),\quad
	N_P(u)\subseteq V(p_iPp_k),
	\]
	and $u$ has no $P$-neighbor before $p_i$.  Shortcutting $P$ through $u$
	is now covered by the symmetric endpoint-safe observation proved at the
	start of Claim~\ref{modified-lsw-314-claim5}: the vertex $u$ is outside $W\cup P$, is proper for
	$W$, satisfies $ux\notin E(G)$, and has all its $P$-neighbors on
	$p_iPp_k$.  Since the witness $p_h$ is one such neighbor, the observation
	gives \(1\le |N(u)\cap V(p_iPp_k)|\le2\).
	Let $u_1,u_2$ be respectively the first and last such neighbors; they are
	equal exactly when this neighbor is unique.  Since the witness $p_h$ is
	one of these neighbors and $h>i$, the last neighbor $u_2$ occurs strictly
	after $p_i$ on $P$; in particular, if $u_1=u_2$, their common value occurs
	strictly after $p_i$.
	
	Let $p_l$ be the last $P$-neighbor of $v$, so $l\le i$.
	If $u_1=u_2$, their common value occurs strictly after $p_i$, and \(H=v-u-u_1Pp_l-v\) is a hole.  The three paths
	\[
	xv,\quad xp_i,\quad xb_2S_2z-p_kPu_1
	\]
	link $x$ to $H$, with distinct ends $v,p_i,u_1$.
	For inducedness, $u$ has the unique $P$-neighbor $u_1$, $p_l$ is
	the last $P$-neighbor of $v$, and neither $u$ nor $v$ sees
	$zS_2b_2$.  Claim~\ref{modified-lsw-314-claim1} excludes every unintended edge from the retained
	$P$-segments to this sector segment: in case \ref{modified-lsw-314-claim1-b} the internal
	attachments after $p_i$ lie at the omitted spoke $a_2$, while in
	case \ref{modified-lsw-314-claim1-a} $P^*$ is anticomplete to $C$.
	The extreme choice of $z$ excludes every further $p_k$-attachment,
	and $x$ has no other $P$-neighbor and is anticomplete to $S_2^*$.
	This also covers $l=i$ and $u_1=p_k$.
	
	If $u_1\ne u_2$, set $R=xT_0x$, let $A=u_1Pu_2$, and let $B$
	be the complementary $u_1$--$u_2$ arc of $R$.  Since $u_2$ lies
	strictly after $p_i$, the sector $T_0$ is long and $R$ is a hole.
	The sectors $S_4$ and $T_0$ are disjoint, so the attachment
	containments established above give
	\[
	N_R(u)=\{u_1,u_2\},\quad
	\{x\}\subseteq N_R(v)\subseteq\{x,p_i\}
	\subseteq V(B)\setminus\{u_2\},\quad x\in V(B^*).
	\]
	Apply Lemma~\ref{lem:external-connector-obstruction} with
	$Q=u-v$, $a=u_1$, and $b=u_2$.  The path is disjoint from $R$
	and has no internal vertex; the displayed conditions include the
	possibility $u_1=p_i$.  The lemma and (N4) contradict global
	spoke-minimality.  Every case is impossible, proving Claim~\ref{modified-lsw-314-claim5}.
	\end{proof}
	Let $p_j$ be the unique neighbor of $v$ on $P$.

	\begin{claim}\label{modified-lsw-314-claim6}
		Alternative \ref{modified-lsw-314-claim4-c} of Claim~\ref{modified-lsw-314-claim4} holds.
	\end{claim}
	\begin{proof}[Proof of Claim~\ref{modified-lsw-314-claim6}]
	If \ref{modified-lsw-314-claim4-a} of Claim~\ref{modified-lsw-314-claim4} held, the linked-cycle principle would give at least four neighbors
	of $v$ on $xS_0x$.  By Claim~\ref{modified-lsw-314-claim5} only one is on $P$, so $v$ would have at
	least three neighbors on the sector subpath $a_1S_1s$, contrary to
	properness of $W$.  Alternative \ref{modified-lsw-314-claim4-b} of Claim~\ref{modified-lsw-314-claim4} is symmetric.
	\end{proof}
	  
	\begin{claim}\label{modified-lsw-314-claim7}
		At least one of the following holds:
		\begin{enumerate}[label=(\alph*)]
			\item $P^*$ has an edge to $C$;\label{modified-lsw-314-claim7-a}
			\item $j=1$ and $S_4$ can be chosen with $a_1\in V(S_4)$;\label{modified-lsw-314-claim7-b}
			\item $j=k$ and $S_4$ can be chosen with $b_2\in V(S_4)$.\label{modified-lsw-314-claim7-c}
		\end{enumerate}
	\end{claim}
	\begin{proof}[Proof of Claim~\ref{modified-lsw-314-claim7}]
	Suppose not.  If $1<j<k$, the paths $p_jPp_1$, $p_jPp_k$, and the edge
	$p_jv$ form an induced subdivision of $K_{1,3}$ with center $p_j$ and
	leaves $p_1,p_k,v$.  Claim~\ref{modified-lsw-314-claim5} gives $N_P(v)=\{p_j\}$, and the assumed
	failure of \ref{modified-lsw-314-claim7-a} makes every nonleaf vertex anticomplete to $C$.  Take the
	three cyclic blocks to be
	\[
	I_1=S_1,\quad I_2=S_2,\quad I_3=S_4.
	\]
	They occur noncrossingly in cyclic order because $S_4\subseteq Q_1$;
	two can meet only at a common spoke.  The exclusive $S_1$-attachment of
	$p_1$, the exclusive $S_2$-attachment of $p_k$, and the attachment of
	$v$ in $S_4\setminus(V(S_0)\cup V(T_0))$ show that all three attachment
	sets are nonempty and pairwise distinct.  Since each set is contained in
	its block, pairwise intersections have size at most one.
	Lemma~\ref{lem:modified-lsw-35} would make one leaf anticomplete to $C$ or
	produce a proper four-spoke wheel.  Neither is possible, so
	$j\in\{1,k\}$; assume $j=1$.
	
	Because Claim~\ref{modified-lsw-314-claim7} is assumed false and $j=1$, no choice of the carrying
	sector $S_4$ satisfying Claim~\ref{modified-lsw-314-claim4}\ref{modified-lsw-314-claim4-c} contains
	$a_1$.  Let $S_5$ be the sector of $Q_1$ incident with $a_1$, and
	denote its other end by $b_3$.  Fix a witness
	\[
	q\in N_C(v)\cap
	\bigl(V(S_4)\setminus(V(S_0)\cup V(T_0))\bigr)
	\]
	supplied by Claim~\ref{modified-lsw-314-claim4}\ref{modified-lsw-314-claim4-c}.  Since
	$V(Q_1)\cap V(S_0)=\{a_1\}$ and
	$V(Q_1)\cap V(T_0)=\{b_2\}$, we have $q\in V(Q_1^*)$.
	
	No single sector of $W$ contains $N_C(p_1)\cup N_C(v)$.  Indeed, such a
	sector $R$ would contain $q$ and an attachment of $p_1$ in
	$S_1\setminus S_2$.  Hence $R\subseteq Q_1$ is incident with $a_1$,
	and all of $N_C(v)$ lies in $R$.  Thus $R$ could replace $S_4$ in
	Claim~\ref{modified-lsw-314-claim4}\ref{modified-lsw-314-claim4-c}, while the neighbor $p_1\in V(S_0)\setminus V(R)$ preserves
	the $S_0$-side condition.  This gives an admissible carrier containing
	$a_1$, a contradiction.
	
	It follows that exactly one of $p_1,v$ is adjacent to $x$.  If both were
	adjacent, Lemma~\ref{lem:modified-lsw-34} would put their rim attachments in
	one sector; if neither were adjacent, the no-$x$-neighbor lemma, available
	at order two, would do the same.  Both conclusions contradict the
	preceding paragraph.
	
	We next claim that $p_1-v$ is a counterexample to the
	two-intersecting-sector conclusion.  Suppose that two intersecting
	sectors cover its rim attachments.  By Lemma~\ref{lem:sector-carrier-facts}\ref{carrier-facts-1}, one member, say $B$,
	contains all of $N_C(v)$.
	The sector $B$ may replace the carrier $S_4$ in Claim~\ref{modified-lsw-314-claim4}\ref{modified-lsw-314-claim4-c}.  The assumed
	failure of Claim~\ref{modified-lsw-314-claim7} therefore gives $a_1\notin V(B)$.  The other covering
	sector cannot be $S_1$, since a $Q_1$-sector intersecting $S_1$ would then
	meet it at $a_1$ and force $a_1\in V(B)$.  Hence normalization
	$s\ne b_1$ forces $N_C(p_1)=\{a_1\}$; the other covering sector is
	$S_5$, and $B\cap S_5=\{b_3\}$.
	If $xv\in E(G)$, then $xp_1\notin E(G)$ and $i>1$, so the return-path observation~\ref{modified-lsw-314-claim4-+i} applies with $j=1$ and $T=S_5$.
	Otherwise $xp_1\in E(G)$, and the return-path observation~\ref{modified-lsw-314-claim4-+ii} applies.
	Both alternatives are impossible, proving the claim.
	
	We record the \textit{endpoint observation} for later use: under the
	standing setup, if $N_P(v)=\{p_1\}$, exactly one of $p_1,v$ is adjacent
	to $x$, and no admissible carrier contains $a_1$, then $p_1v$ violates
	the two-intersecting-sector conclusion.  The preceding cover argument
	uses only these conditions and normalization; it does not require
	$P^*$ to be anticomplete to $C$.
	
	If $k\ge3$, this order-two counterexample contradicts the induction
	hypothesis.  Hence $k=2$.  We now give it a normalized representation
	with a smaller complementary path.  Retain $S_1$ as the carrier of
	$N_C(p_1)$, and let $U\subseteq Q_1$ be a sector carrying $N_C(v)$.
	No such $U$ contains $a_1$, since otherwise it could replace $S_4$ in
	Claim~\ref{modified-lsw-314-claim4}\ref{modified-lsw-314-claim4-c}.  Orient the new pair $S_1,U$ so that the end $r$ of $U$
	nearer $a_1$ is its new $b_2$-end.  The $p_1$-side remains normalized.
	If the $v$-side is not normalized, then $N_C(v)$ consists solely of the
	other end $d$ of $U$.  The fixed witness $q$ then equals $d$, so
	$d\in V(Q_1^*)$ and in particular $d$ is not the old $b_2$.
	Replace $U$ by the sector incident with $d$ on the side toward the old
	$b_2$; this sector lies in $Q_1$, still carries $N_C(v)=\{d\}$, and has
	$d$ as its new $b_2$-end.  Thus the new representation is normalized.
	The new complementary path is
	respectively $rQ_1a_1$ or $dQ_1a_1$, and in either case it is a proper
	subpath of the old $Q_1$.  This contradicts the normalized minimum choice of
	$|V(Q_1)|$ and proves Claim~\ref{modified-lsw-314-claim7}.
	\end{proof}
	
	\begin{claim}\label{modified-lsw-314-claim8}
		$i=j$.
	\end{claim}
	\begin{proof}[Proof of Claim~\ref{modified-lsw-314-claim8}]
	Assume by symmetry that $j<i$.  First let $vx\notin E(G)$.  The
	no-$x$-neighbor lemma applied to $p_1Pp_j-v$ puts all of its rim
	attachments in one sector $R$.  That sector contains an attachment of
	$p_1$ in $S_1$ and an attachment of $v$ in $Q_1$; hence it meets $S_1$
	at $a_1$.  Replacing the carrying sector in Claim~\ref{modified-lsw-314-claim4}\ref{modified-lsw-314-claim4-c} by $R$, we may
	take $a_1\in V(S_4)$, and normalization gives
	$N_C(p_1)=\{a_1\}$ and $t=a_1$.  Properness gives at most two
	$v$-neighbors there.
	If there are two and either $va_1\notin E(G)$ or $j\ne1$, use the two attachment edges from \(v\) to \(S_4\) together with $v-p_jPp_i-x$ when $j>1$, and with $v-p_1-a_1$
	when $j=1$ and $va_1\notin E(G)$.  These paths link $v$ to the
	induced cycle $xS_4x$.
	If there are two neighbors, the preceding argument therefore forces
	$va_1\in E(G)$ and $j=1$.  Writing $v'$ for the other neighbor, we have
	$a_1v'\notin E(G)$ by diamond-freeness, and $p_1$ is linked to the hole
	obtained by rerouting $S_4$ through $v$ by
	$p_1a_1,p_1v,p_1Pp_i-x$.  Hence $v$ has a unique neighbor $v'$ in
	$S_4-a_1$, and the paths $p_jvv',p_jPp_1a_1,p_jPp_i-x$ link $p_j$ to
	the induced cycle $xS_4x$.  Therefore $vx\in E(G)$.
	
	If now $a_1\in V(S_4)$, take the $v$-neighbor $v'$ closest to $a_1$.
	Put
	\[
	D=p_1Pp_j-v-v'S_4a_1S_1s-p_1.
	\]
	The extreme choice of $v'$, $N_P(v)=\{p_j\}$, and Claim~\ref{modified-lsw-314-claim1} exclude
	chords of $D$.  The only possible triangular degeneration has $j=1$
	and $v'=s=a_1$; then $G[\{x,v,p_1,a_1\}]$ is a diamond, whose only
	missing edge is $xp_1$ because $j<i$.  Thus $D$ is a hole.
	If $v'$ is the end of $S_4$ different from $a_1$, then the three linking edges $xv,xa_1,xv'$ link $x$ to $D$.  Otherwise, if $v'\ne a_1$, then
	$v'\in S_4^*$ and $xv'\notin E(G)$; if $v'=a_1$, its adjacency to
	$x$ is already the listed edge $xa_1$.  In either case
	$N_D(x)=\{v,a_1\}$, and the three paths
	\[
	xv,\quad xa_1,\quad x-p_iPp_j
	\]
	link $x$ to $D$.  The extreme choice of $v'$, $N_P(v)=\{p_j\}$, and
	Claim~\ref{modified-lsw-314-claim1} exclude all cross-edges in these linked-hole constructions.
	Hence $a_1\notin V(S_4)$.  Let $S_5$ be the sector
	preceding $S_1$ at $a_1$, with other end $b_3$.
	The preceding argument excludes every admissible carrier containing
	$a_1$.  Any admissible carrier meeting $S_5$ must therefore meet it
	at $b_3$.  Since $j<i$ and $xv\in E(G)$, the return-path observation~\ref{modified-lsw-314-claim4-+i} excludes every such carrier.
	Suppose nevertheless that two intersecting sectors
	cover the attachments of $p_1Pp_j-v$.  By Lemma~\ref{lem:sector-carrier-facts}\ref{carrier-facts-1}, one member $B$ of
	the covering pair contains all of $N_C(v)$.  Relabel
	$B$ as $S_4$; it is an admissible carrier in Claim~\ref{modified-lsw-314-claim4}\ref{modified-lsw-314-claim4-c}.  The other
	covering sector cannot be $S_1$, since then $B\subseteq Q_1$ would meet
	$S_1$ at $a_1$.  Consequently normalization forces
	$N_C(p_1)=\{a_1\}$, the other covering sector is $S_5$, and
	$B\cap S_5\ne\emptyset$, contrary to the return-path observation.  Thus no two
	intersecting sectors cover the attachments of $p_1Pp_j-v$.
	
	The induction hypothesis forces $j=k-1$ and $i=k$.  Claim~\ref{modified-lsw-314-claim1} makes
	$a_2$ anticomplete to $P-p_k$.  Claim~\ref{modified-lsw-314-claim7} now gives an internal
	$P$--$C$ edge, so Claim~\ref{modified-lsw-314-claim1} supplies a sector $S_3$ with ends $b_1,a_2$
	and an internal neighbor of $b_1$.  Choose such a neighbor $p_{j'}$
	closest to $p_{k-1}$.  Fix the Claim~\ref{modified-lsw-314-claim4}\ref{modified-lsw-314-claim4-c} witness
	$q\in N_C(v)\cap V(Q_1^*)$.  Since every admissible carrier of
	$N_C(v)$ is disjoint from $S_5$, the attachments $b_1$ and $q$ cannot
	be covered by two intersecting sectors.  The induced path
	$p_{j'}Pp_j-v$ has order less than $k$, has $v$ as its unique neighbor
	of $x$, and retains these two attachments.  It therefore contradicts
	the induction hypothesis.  This proves Claim~\ref{modified-lsw-314-claim8}.
	\end{proof}
	
	\begin{claim}\label{modified-lsw-314-claim8+}
		If $1<i<k$, then $N_C(p_i)\ne\emptyset$.
	\end{claim}
	\begin{proof}[Proof of Claim~\ref{modified-lsw-314-claim8+}]
		Suppose otherwise, and choose $j_1<i$ maximum and $j_2>i$ minimum
		such that $p_{j_1},p_{j_2}$ attach to $C$.  These indices exist
		because both ends of $P$ attach to $C$.  By Claims~\ref{modified-lsw-314-claim5} and~\ref{modified-lsw-314-claim8},
		the paths $p_iPp_{j_1}$, $p_iPp_{j_2}$, and $p_i-v$ form an
		induced subdivision of \(K_{1,3}\).  Its nonleaf vertices are anticomplete
		to $C$ by the choices of $j_1,j_2$ and the present assumption.
		Use the three blocks
		\[
		I_L=\begin{cases}S_1,&j_1=1,\\ \{b_1\},&j_1>1,\end{cases}
		\quad
		I_R=\begin{cases}S_2,&j_2=k,\\ \{a_2\},&j_2<k,\end{cases}
		\quad I_v=S_4.
		\]
		Claim~\ref{modified-lsw-314-claim1} gives containment of the three leaf-attachment sets in these
		blocks; in Claim~\ref{modified-lsw-314-claim1}\ref{modified-lsw-314-claim1-a}, necessarily $j_1=1,j_2=k$.  The blocks occur
		noncrossingly on $C$ and meet, if at all, at a common end spoke.
		The two end witnesses or the distinct spokes $b_1,a_2$, and the
		Claim~\ref{modified-lsw-314-claim4}\ref{modified-lsw-314-claim4-c} witness in $Q_1^*$, make the three attachment sets
		nonempty and pairwise distinct.  Any two meet in at most one vertex.
		Lemma~\ref{lem:modified-lsw-35} therefore gives an anticomplete leaf
		or a proper exact-four wheel.  Both are impossible, proving the claim.
	\end{proof}

	\begin{claim}\label{modified-lsw-314-claim9}
	$vx\notin E(G)$.
	\end{claim}
	\begin{proof}[Proof of Claim~\ref{modified-lsw-314-claim9}]
	Suppose $vx\in E(G)$.  Then $p_i,v\in N(x)$ and $p_iv\in E(G)$ by
	Claim~\ref{modified-lsw-314-claim8}.  Applying Lemma~\ref{lem:sparse-neighborhood} to the edge
	$p_iv$ inside $G[N(x)]$ gives \(N_C(p_i)\cap N_C(x)=N_C(v)\cap N_C(x)=\emptyset\), because the edge $p_iv$ already uses the unique possible neighbor of
	each of $p_i,v$ in $G[N(x)]$.  In particular, the Claim~\ref{modified-lsw-314-claim4}\ref{modified-lsw-314-claim4-c} witness in
	$S_4\setminus(V(S_0)\cup V(T_0))$ lies in $S_4^*$.  We claim that
	$p_i$ is anticomplete to $C$.  If it attached to $C$, properness and the
	preceding equality would put all its attachments in the interior
	of a unique sector.  Claim~\ref{modified-lsw-314-claim1} places that sector on the
	$S_1$--$Q_2$--$S_2$ side, whereas $v$ has an attachment in
	$S_4^*\subseteq Q_1$.  No sector contains both attachment sets,
	contradicting Lemma~\ref{lem:modified-lsw-34}.  Hence $N_C(p_i)=\emptyset$, so
	$p_i\in P^*$.
	Since $p_i\in P^*$, this contradicts Claim~\ref{modified-lsw-314-claim8+}.  Claim~\ref{modified-lsw-314-claim9} follows.
	\end{proof}
	  
	\begin{claim}\label{modified-lsw-314-claim10}
		No admissible carrier contains $a_1$ or $b_2$.
	\end{claim}
	\begin{proof}[Proof of Claim~\ref{modified-lsw-314-claim10}]
	Suppose that an admissible carrier $S_4$ can be chosen with
	$a_1\in V(S_4)$, and let $a_3$ be its other end.  Since
	$Q_1^*$ contains a spoke, $b_2\notin V(S_4)$ and $S_4$ is disjoint
	from $S_2$.  Claim~\ref{modified-lsw-314-claim1} therefore gives
	$N_C(p_i)\cap V(S_4)\subseteq\{a_1\}$, including when $i=k$;
	moreover, $p_i a_1\in E(G)$ forces $i=1$.
	If $v$ has a unique
	neighbor $v'$ on $S_4$, the fixed witness in $Q_1^*$ ensures
	$v'\ne a_1$.  Then the paths
	\[
	p_i-v-v',\quad p_i-x,\quad p_iPp_1-sS_1a_1
	\]
	link $p_i$ to the induced cycle $xS_4x$.
	If $p_i a_1\in E(G)$, then Claim~\ref{modified-lsw-314-claim1} gives $i=1$ and the extremal
	definition gives $s=a_1$, so the third displayed path is the edge
	$p_1a_1$; otherwise Claim~\ref{modified-lsw-314-claim1} and the choice of $s$ exclude every chord
	from $p_iPp_1-sS_1a_1$ to $xS_4x$.
	Thus $v$ has at least two neighbors on $S_4$; by properness and Claim~\ref{modified-lsw-314-claim9}
	it has exactly two.  Denote them by $r_1,r_2$ in their order from $a_1$
	toward $a_3$.
	
	The sector $S_4$ is long, since otherwise
	$G[\{x,v,a_1,a_3\}]$ is a diamond, its only missing edge being $xv$.
	Set $R=xS_4x$, let $A=r_1S_4r_2$, and let $B$ be the complementary
	$r_1$--$r_2$ arc of $R$.  By Claim~\ref{modified-lsw-314-claim9} and the attachment restriction
	at the start of this claim,
	\[
	N_R(v)=\{r_1,r_2\},\quad
	\{x\}\subseteq N_R(p_i)\subseteq\{x,a_1\}
	\subseteq V(B)\setminus\{r_2\},\quad x\in V(B^*).
	\]
	Here $r_2\ne a_1$ because $r_1,r_2$ occur in that order from $a_1$
	toward $a_3$.  Thus Lemma~\ref{lem:external-connector-obstruction}
	applies with $Q=v-p_i$, $a=r_1$, and $b=r_2$: the path is disjoint
	from the hole $R$ and has no internal vertex.  The lemma and \textbf{(N4)}
	contradict global spoke-minimality.  Therefore no admissible choice
	of $S_4$ contains $a_1$.  By symmetry no admissible carrier contains $b_2$ either:
	Claims~\ref{modified-lsw-314-claim5} and~\ref{modified-lsw-314-claim8} give $N_P(v)=\{p_i\}$, and $p_i\in V(S_0)\cap V(T_0)$,
	so reflecting the two sides preserves the conditions in Claim~\ref{modified-lsw-314-claim4}\ref{modified-lsw-314-claim4-c}.
	This proves Claim~\ref{modified-lsw-314-claim10}.
	\end{proof}
	
	Claim~\ref{modified-lsw-314-claim7} now forces an edge between $P^*$ and $C$.  Consequently
	$k\ge3$ and Claim~\ref{modified-lsw-314-claim1}\ref{modified-lsw-314-claim1-b} holds.  Suppose first that $1<i<k$.
	Claim~\ref{modified-lsw-314-claim8+} gives \(\emptyset\ne N_C(p_i)\subseteq V(S_3)\).
	Apply the induction hypothesis to the edge $p_i-v$, whose unique
	$x$-neighbor is $p_i$ by Claim~\ref{modified-lsw-314-claim9}.  Its order is $2<k$, so two
	intersecting sectors cover its rim attachments.  By
	Lemma~\ref{lem:sector-carrier-facts}\ref{carrier-facts-1}, one member $B$ of this pair
	contains all of $N_C(v)$.  The fixed witness in $Q_1^*$ puts $B$
	in $Q_1$.  Since $N_P(v)=\{p_i\}$, replacing $S_4$ by $B$ preserves
	Claim~\ref{modified-lsw-314-claim4}\ref{modified-lsw-314-claim4-c}; thus $B$ is admissible and contains neither $a_1$ nor $b_2$.
	The other member must meet the nonempty set $N_C(p_i)\subseteq V(S_3)$,
	and hence is one of $S_1,S_3,S_2$.  The sector $S_3$ is disjoint from
	$Q_1$, while $S_1$ and $S_2$ can meet a sector of $Q_1$ only at $a_1$
	and $b_2$, respectively.  None can meet $B$, a contradiction.
	Therefore $i\in\{1,k\}$.
	
	By symmetry let $i=1$.  Claims~\ref{modified-lsw-314-claim5}, \ref{modified-lsw-314-claim8}, \ref{modified-lsw-314-claim9} and~\ref{modified-lsw-314-claim10} verify the endpoint
	observation in Claim~\ref{modified-lsw-314-claim7}, so $p_1v$ is an order-two counterexample to the
	two-intersecting-sector conclusion.  Since $k\ge3$, this contradicts
	the induction hypothesis and proves the lemma.
\end{proof}

\section{Sector separation}\label{sec:sector-separation}

We now combine the preceding attachment lemmas.

\begin{theorem}\label{thm:wheel-attachment}
	Let $G$ be \(\{\mathrm{ISK}_4,\text{diamond}, \text{prism},K_{3,3}\}\)-free, let $W=(C,x)$ be a proper wheel with globally minimum
	number of spokes, and let $P\subseteq G-V(W)$ be an induced path containing
	at most one neighbor of $x$.
	\begin{enumerate}[label=(\roman*)]
		\item If $P\cap N(x)=\emptyset$, all ends on $C$ of edges between $P$ and
		$C$ lie in a single sector.\label{wheel-attachment-1}
		\item If $|P\cap N(x)|=1$, those ends lie in the union of two intersecting
		sectors; the two sectors are allowed to be equal.\label{wheel-attachment-2}
	\end{enumerate}
\end{theorem}

\begin{proof}
	If $\sigma(W)=4$, apply Lemma~\ref{lem:four-spoke-attachment}; its proof
	includes both the case $P^*$ is anticomplete to $C$ and the case in
	which an internal vertex of $P$ attaches to the rim.
	
	Let $\sigma(W)\ge5$ and use strong induction on $k=|V(P)|$.
	For $k=1$, both conclusions follow from properness: a single outside
	vertex has all its rim neighbors in one sector, and in assertion~\ref{wheel-attachment-2}
	the two sectors may be taken equal.  For $k\ge2$, first prove
	assertion~\ref{wheel-attachment-1} by Lemma~\ref{lem:modified-lsw-313}; its hypotheses use
	both path conclusions only at orders below $k$.
	Then apply Lemma~\ref{lem:modified-lsw-314} to obtain assertion~\ref{wheel-attachment-2}.
	That lemma uses the one-sector conclusion at order at most $k$, including
	the case just proved, and the two-sector conclusion only below $k$.
	Thus the simultaneous induction is well founded and has no circular
	invocation; it is used only for $\sigma(W)\ge5$.
\end{proof}

\begin{theorem}\label{thm:sector-separation}
	Let $G$ be \(\{\mathrm{ISK}_4,\text{diamond}, \text{prism},K_{3,3}\}\)-free, and let $W=(C,x)$ be a proper wheel whose number of spokes is minimum over all proper wheels and all centres in $G$.  Then the
	following hold.
	\begin{enumerate}[label=(\roman*)]
		\item Every component of $G-N[x]$ meets the nonempty interior of at most one
		sector of $W$.\label{sector-separation-1}
		\item For $u\in N(x)$, let $D_u$ be the component containing $u$ in
		$G-(N[x]\setminus\{u\})$.  Then $D_u$ meets the nonempty interiors of
		at most two sectors; if it meets two, those sectors intersect.\label{sector-separation-2}
	\end{enumerate}
\end{theorem}

\begin{proof}
	For \ref{sector-separation-1}, suppose that a component $K$ of $G-N[x]$ meets the interiors of two
	distinct sectors $S,T$.  Among paths in $K$ joining two distinct sector
	interiors, choose
	\[
	v_0v_1\cdots v_\ell,
	\quad v_0\in S^*,\quad v_\ell\in T^*,
	\]
	with minimum length.  No internal vertex lies on $C$: it cannot be a spoke
	because $K\subseteq G-N[x]$, while splitting at an internal nonspoke rim
	vertex leaves a shorter path between two distinct sector interiors.  Since
	$C$ is induced and $S\ne T$, we have $\ell\ge2$.  Hence
	$v_1\cdots v_{\ell-1}$ is an induced path outside $W$ with no neighbor of
	$x$.  It attaches in $S^*$ and $T^*$, which cannot be covered by one
	sector, contrary to Theorem~\ref{thm:wheel-attachment}\ref{wheel-attachment-1}.
	
	For \ref{sector-separation-2}, every component of $D_u-u$ is a component of $G-N[x]$, so \ref{sector-separation-1}
	applies to it.  First suppose $u\notin V(C)$.  If $D_u$ meets interiors of
	two disjoint sectors $S,T$, they lie in different components $K_S,K_T$ of
	$D_u-u$.  In each of $G[K_S\cup\{u\}]$ and $G[K_T\cup\{u\}]$, take a shortest
	path from $u$ to the prescribed sector interior, stopping at its first rim
	vertex.  Apart from that endpoint, each path lies outside $C$, because \ref{sector-separation-1}
	places all rim vertices of its component in the same sector interior.  The
	two paths otherwise lie in different components of $D_u-u$; shortestness
	therefore makes their concatenation at $u$ induced.  Delete the two rim
	endpoints.  The resulting outside path contains exactly one neighbor of
	$x$, namely $u$, and has attachments in $S^*$ and $T^*$.  No two
	intersecting sectors cover those two
	attachments, contrary to Theorem~\ref{thm:wheel-attachment}\ref{wheel-attachment-2}.  The sectors
	represented in $D_u$ are therefore pairwise intersecting.  Since the sector
	intersection graph is a cycle of length at least four, there are at most two.
	
	It remains to suppose $u\in V(C)$.  If a component $K$ of $D_u-u$ meets
	the interior of a sector $S$ not incident with $u$, then \ref{sector-separation-1} says that all rim
	vertices in $K$ lie in $S^*$.  As $u$ has a neighbor in $K$, take in
	$G[K\cup\{u\}]$ a shortest path from $u$ to the set $V(C)\cap V(K)$,
	and write it as
	\[
	u r_1\cdots r_m a,
	\quad a\in V(C)\cap V(K)\subseteq S^*.
	\]
	Its internal vertices are outside $C$.  Moreover $m\ge1$, because $S$ is not
	incident with $u$ and $C$ is induced.  Thus $r_1\cdots r_m$ is an induced
	outside path with no $x$-neighbor, attached both at the spoke $u$ and in the
	interior of a nonincident sector.  This contradicts
	Theorem~\ref{thm:wheel-attachment}\ref{wheel-attachment-1}.  Hence every represented sector is
	incident with $u$, and there are at most the two sectors meeting at $u$.
\end{proof}

\section{Series--parallel lobes and rooted transversals}
\label{sec:core}\label{sec:rooted-marker}

Let us call $(c,D)$ a \textit{full-star pair} if $G-N[c]$ has at least two nonempty
components and $D$ is one of them.  Put
\begin{equation}\label{eq:full-star-lobe}
	A=N_G(D)=N(D)\cap N(c),
	\quad F=G[D\cup A\cup\{c\}].
\end{equation}
We call $F$ the \textit{lobe associated with $(c,D)$}.
The first equality uses the external-neighborhood convention.  Since $D$ is
a component of $G-N[c]$, all its external neighbors lie in $A$, and $D$ is
anticomplete to $c$.

\begin{lemma}\label{lem:min-lobe-sp}
	Let $G$ be \(\{\mathrm{ISK}_4,\text{diamond}, \text{prism},K_{3,3}\}\)-free and connected.  Suppose that $(c,D)$ is a full-star
	pair for which no full-star pair $(s,Q)$ has $Q\subsetneq D$.  Then the
	graph $F$ in \eqref{eq:full-star-lobe} is series--parallel.
\end{lemma}

\begin{proof}
	We first record a boundary-lifting observation.  If $s\in V(F)$ and
	$Q\subseteq D$ is a component of $F-N_F[s]$, then $Q$ is also a component of
	$G-N_G[s]$.  Indeed, if $q\in Q$, $qv\in E(G)$ and $v\notin N_G[s]$, then
	$N_G(D)=A$ gives $v\in D\cup A\subseteq V(F)$.  Thus
	$v\in V(F)\setminus N_F[s]$, and it must belong to $Q$.
	
	Suppose that $F$ is not series--parallel.  Since $F\in\mathcal C$, \ref{L1} and Lemma~\ref{lem:proper-wheel-exists} give a proper
	wheel in $F$.  In $F$, choose such a wheel $W=(C,s)$ with globally minimum
	number of spokes and apply Theorem~\ref{thm:sector-separation}.  It has two vertex-disjoint long sectors by the observation following
	Lemma~\ref{lem:sparse-neighborhood}.
	
	If $s=c$, then $F-N_F[c]=D$ is connected but contains at least two long-sector
	interiors, contrary to Theorem~\ref{thm:sector-separation}\ref{sector-separation-1}.
	
	Now suppose $s\ne c$.  Put \(J=F-(N_F[s]\setminus\{c\})\), and let $R$ be its component containing $c$.  If $sc\notin E(F)$,
	then $J=F-N_F[s]$, so separation~\ref{sector-separation-1} of Theorem~\ref{thm:sector-separation} puts interiors from at most one
	sector in $R$.  If $sc\in E(F)$, separation~\ref{sector-separation-2} of Theorem~\ref{thm:sector-separation}, with retained
	neighbor $c$, puts interiors from at most two intersecting sectors
	in $R$.  Since $W$ has two vertex-disjoint long sectors, some long-sector
	interior therefore lies in a component $Q\ne R$ of $J$.  Every
	surviving vertex of $A$ is adjacent to $c$, so $Q\subseteq D$.
	
	Passing from $J$ to $F-N_F[s]$ deletes at most $c$ and does not affect
	$Q$; hence $Q$ remains a component and the boundary-lifting observation
	applies.  All long-sector interiors survive this deletion because they
	contain no neighbor of $s$.  By separation~\ref{sector-separation-1} of Theorem~\ref{thm:sector-separation}, an interior from
	another long sector lies outside $Q$, so $G-N_G[s]$ has another
	nonempty component.  Moreover, $N_F[s]\cap D\ne\emptyset$: it
	contains $s$ if $s\in D$, and a neighbor of $s$ in $D$ if
	$s\in A=N_G(D)$.  Thus $Q\subsetneq D$, contradicting the choice of
	$(c,D)$.
	These contradictions show that $F$ is series--parallel.
\end{proof}

\begin{lemma}\label{lem:lobe-gluing}
	Let $(c,D)$ be a full-star pair and use the notation in \eqref{eq:full-star-lobe}.
	Suppose that $F$ is series--parallel.  If $G-D$ has a stable triangle
	transversal $S_0$, then there is a stable triangle transversal \(S\) of \(G\) with \(S \cap V (G-D) = S_0\).
\end{lemma}

\begin{proof}
	Set $X=S_0\cap A$.  In $F$, contract the connected set $\{c\}\cup X$ to one
	vertex $r$ and suppress loops and parallel edges.  The resulting graph $F_X$
	is a minor of $F$, hence is series--parallel.  Let $\varphi$ be a proper
	$3$-coloring of $F_X$, and set
	\begin{equation}\label{eq:lobe-transversal-gluing}
		Z=\{d\in D:\varphi(d)=\varphi(r)\},
		\quad S=S_0\cup Z.
	\end{equation}
	The set $Z$ is stable.  If $x\in X$ is adjacent to $d\in D$, then $rd$ is an
	edge of $F_X$, so $d\notin Z$.  Every external neighbor of $D$ lies in $A$,
	and $D$ is anticomplete to $c$; consequently $S$ is stable.
	
	Suppose that a triangle $T$ disjoint from $S$ meets $D$.  Since every
	external neighbor of $D$ lies in $A$, we have
	$T\subseteq D\cup(A\setminus X)$.  Thus no vertex of $T$ is contracted,
	and $T$ remains a triangle in $F_X$.
	Its vertices in $D$ avoid the color $\varphi(r)$ by the definition of $Z$,
	and its vertices in $A\setminus X$ avoid that color because they are
	adjacent to $r$.  Thus $T$ would use only two colors in a proper coloring,
	a contradiction.  Triangles contained in $G-D$ are already met by $S_0$.
\end{proof}

A degree-two marker allows us to prescribe whether its two neighbors have
one color or two.  This is needed later to glue colorings across a proper
$2$-cutset.

\begin{theorem}\label{thm:rooted-marker}
	Let $H$ be \(\{\mathrm{ISK}_4,\text{diamond}, \text{prism},K_{3,3}\}\)-free and connected, let $d_H(u)=2$, and suppose
	\[
	N_H(u)=\{a,b\},\quad ab\notin E(H).
	\]
	Put $B=H-u$.  There are stable triangle transversals
	$S_{=},S_{\ne}\subseteq V(B)$ such that
	\begin{equation}\label{eq:rooted-transversal-memberships}
		\{a,b\}\subseteq S_=,
		\quad a\in S_{\ne},\quad b\notin S_{\ne}.
	\end{equation}
	Consequently $B$ has two $4$-colorings $c_=$ and $c_{\ne}$ satisfying
	\begin{equation}\label{eq:rooted-color-signatures}
		c_=(a)=c_=(b),
		\quad c_{\ne}(a)\ne c_{\ne}(b).
	\end{equation}
\end{theorem}

\begin{proof}
	We prove the two assertions simultaneously by induction on $|V(H)|$.
	Suppose first that $H$ is series--parallel.  For each of the two connected
	sets $T\in\{\{a,u,b\},\{a,u\}\}$, contract $T$ to a vertex $r_T$ and
	suppress loops and parallel edges.  The resulting graph is a
	series--parallel minor of $H$, so fix a proper $3$-coloring of it.
	Let $S(T)$ consist of $T\cap V(B)$ and the vertices of $B-T$ having the
	color of $r_T$.  Since $T\cap V(B)$ is stable and every edge from $T$ to
	its complement survives as an edge incident with $r_T$, the set $S(T)$
	is stable.  Every triangle of $B$ meets $T$ in at most one vertex, because
	$ab\notin E(H)$, and therefore survives as a triangle in the minor.
	Its three colors show that it meets $S(T)$.  Taking
	$S_=S(\{a,u,b\})$ gives both roots in the transversal.  Taking
	$S_{\ne}=S(\{a,u\})$ includes $a$ and excludes $b$, since $b$ is adjacent
	to $r_T$ in this contraction.
	
	Assume now that $H$ is not series--parallel.  Since $H$ is \(\{\mathrm{ISK}_4,\text{diamond}, \text{prism},K_{3,3}\}\)-free, \ref{L1} gives a wheel in $H$, and
	Lemma~\ref{lem:proper-wheel-exists} gives a proper one.  Choose such a wheel
	globally spoke-minimal, and denote its centre by $c$.
	Theorem~\ref{thm:sector-separation} and the existence
	of two long sectors show that $H-N[c]$ has at least two components.  We must
	protect the entire rooted set
	\begin{equation}\label{eq:protected-roots}
		P_0=\{u,a,b\},
	\end{equation}
	not merely $\{a,b\}$.  All vertices of $P_0$ surviving the deletion of
	$N[c]$ lie in at most one component.  Indeed, if $c\notin P_0$, then
	$c\notin N(u)=\{a,b\}$, so $u$
	survives and every surviving one of $a,b$ is joined to it; if $c=a$ or
	$c=b$, at most the other root survives.  The case $c=u$ is impossible because
	a wheel centre has degree at least three.  Hence some component $D$ is
	disjoint from $P_0$.
	
	Among full-star pairs $(c',D')$ with $D'\cap P_0=\emptyset$, choose one
	with $D'$ inclusion-minimal.  Any proper subset of $D'$ also avoids $P_0$,
	so no full-star pair $(s,Q)$ has $Q\subsetneq D'$.  Lemma~\ref{lem:min-lobe-sp}
	therefore says that the associated lobe is series--parallel.
	
	The graph $H-D'$ is connected.  The vertices of $N[c']$ are joined through
	$c'$, and every other component of $H-N[c']$ has a neighbor in $N[c']$
	because $H$ is connected.  Since $D'$ avoids $\{u,a,b\}$, the same degree-two
	marker and both roots remain.  Apply the two induction statements to $H-D'$
	to obtain transversals in $(H-D')-u$.  The vertex $u$ lies in no triangle,
	because its two neighbors $a,b$ are nonadjacent.  Hence each of these sets
	is also a stable triangle transversal of $H-D'$.  Apply
	Lemma~\ref{lem:lobe-gluing} in $H$ to extend it through the lobe.  All newly
	added vertices lie in $D'$, so the
	root conditions in \eqref{eq:rooted-transversal-memberships} are unchanged.
	
	Finally, each of $B-S_=$ and $B-S_{\ne}$ is triangle-free and ISK$_4$-free,
	so \ref{L4} gives a $3$-coloring; give the corresponding
	transversal the fourth color.  The root memberships in \eqref{eq:rooted-transversal-memberships} give
	the two signatures in \eqref{eq:rooted-color-signatures}.
\end{proof}

\begin{corollary}\label{thm:core-transversal}
	Every \(\{\mathrm{ISK}_4,\text{diamond}, \text{prism},K_{3,3}\}\)-free graph $G$ has a stable triangle transversal and is $4$-colorable.  Moreover, the transversal can be required to contain any
	one prescribed vertex.
\end{corollary}

\begin{proof}
	It suffices to treat a nonempty connected graph $G$.  Choose $a\in V(G)$
	and add two new vertices $u,b$ and only the edges $au,ub$, obtaining $H$.
	The graph $H$ belongs to $\mathcal C$: every forbidden graph has minimum
	degree at least two, so none can contain $b$; after excluding \(b\), none can contain
	\(u\) either.
	Theorem~\ref{thm:rooted-marker} supplies a stable triangle transversal of
	$H-u=G\mathbin{\dot\cup}\{b\}$ containing $a,b$.  Its restriction to $G$
	is the required transversal $S$.  The graph $G-S$ is triangle-free and
	ISK$_4$-free, so it is $3$-colorable by \ref{L4}; use a
	fourth color on $S$.  Components can be treated independently.
\end{proof}

\section{The proper 2-cutset reduction}
\label{sec:proper-two-cut}

Throughout this section, let $G$ be a connected
$\{\mathrm{ISK}_4,\textup{diamond}\}$-free graph with no clique
cutset.
Suppose that \(G\) has a proper 2-cutset.
Among all oriented splits \eqref{eq:proper-cutset-split} of all proper 2-cutsets of \(G\), choose one with $|X|$ minimum, say \(X=\{a,b\}\).
Since a one-vertex set is a clique, $G$ is $2$-connected.
Consequently every component of $G-\{a,b\}$ has a neighbor in both $a$ and $b$.

Write $B_Y=G[Y\cup\{a,b\}]$.  This graph is connected, and every vertex
of $Y$ has degree at least two in it.  If $B_Y$ were a path, its ends would
therefore be $a,b$, contrary to properness.  Thus $B_Y$ is not a path,
and neither is any graph containing $B_Y$ as a subgraph.

Add a marker vertex and
form
\begin{equation}\label{eq:marker-graph}
	M=G[X\cup\{a,b\}]+u,
	\quad N_M(u)=\{a,b\}.
\end{equation}
The marker path \(a-u-b\) must not be contracted to the edge $ab$: that
contraction can create a new diamond.

\begin{lemma}\label{lem:marker-forbidden}
	The graph $M$ is ISK$_4$-free and diamond-free.
\end{lemma}

\begin{proof}
	By $2$-connectivity, $G[Y\cup\{a,b\}]$ contains an $a$--$b$ path; choose a
	shortest one $P_Y$.  It is induced, its interior is anticomplete to $X$, and
	its only incidences with $a,b$ are its first and last edges.
	
	Every induced subgraph of $M$ that avoids $u$ is an induced subgraph of
	$M-u=G[X\cup\{a,b\}]$, and hence of $G$.  Thus every forbidden
	configuration newly created by the marker contains $u$.
	
	If an induced subdivision of $K_4$ in $M$ uses $u$, then $u$, which has
	degree two, is internal to one subdivided branch and $a-u-b$ is a subpath of
	that branch.  Replacing $a-u-b$ by $P_Y$ gives an induced subdivision of
	$K_4$ in $G$, a contradiction.  A new diamond would have to use $u$, but the
	two neighbors of every degree-two vertex in a diamond are adjacent, whereas
	$ab\notin E(G)$; $u$ cannot be a degree-three vertex of a diamond either.	
\end{proof}

\begin{lemma}\label{lem:marker-uncrossing}
	The graph $M$ has neither a clique cutset nor a proper $2$-cutset.
\end{lemma}

\begin{proof}
	Suppose first that $K$ is a clique cutset of $M$.  If $u\in K$, then, because
	$N(u)=\{a,b\}$ and $ab\notin E(G)$,
	\[
	K\in\bigl\{\{u\},\{u,a\},\{u,b\}\bigr\}.
	\]
	Every component of $G[X]$ has a neighbor in each of $a,b$.
	Hence $G[X\cup\{a,b\}]$, $G[X\cup\{a\}]$, and
	$G[X\cup\{b\}]$ are connected, so none of the three candidates is a cutset.
	If $u\notin K$, all surviving vertices among $a,b$ lie with $u$ in one
	component of $M-K$.  Any other component lies in $X$.  Since $K$ cannot
	contain both nonadjacent roots, and every component of $G[Y]$ has a
	neighbor at each root, replacing $u$ by $Y$ keeps its component
	connected.  Thus the same $K$ is a clique cutset of $G$, a contradiction.  Thus $M$ has no clique cutset and, in particular, is
	$2$-connected.
	
	Suppose next that $T'=\{p,q\}$ is a proper $2$-cutset of $M$, with split
	$(U,V,T')$.  First let $u\notin T'$ and choose the notation so that $u\in U$.
	Every surviving one of $a,b$ is adjacent to $u$, and hence $V\subseteq X$.
	In $G$, the same cutset separates $V$ from $(U\setminus\{u\})\cup Y$.  The
	$V$-block is the same nonpath block as in $M$.  The other block contains a
	$p$--$q$ path: if a path in the old $U$-block uses $a-u-b$, replace that
	subpath by an $a$--$b$ path through $Y$.  The new complementary block contains $B_Y$, so it is not a path.
	Hence $T'$ is a proper
	$2$-cutset of $G$.  Moreover $V\subsetneq X$: equality would force
	$T'=\{a,b\}$ and $U=\{u\}$, whose marker block is the path $a-u-b$.
	This contradicts the minimum choice of $X$.
	
	It remains that $u\in T'$, say $T'=\{u,q\}$.  Then $q\in X$ because the two
	vertices of a proper $2$-cutset are nonadjacent.  Since $M$ is $2$-connected,
	$a,b$ lie on different sides; otherwise $q$ is a cut vertex.  Relabel the
	sides and $a,b$ so that $a\in U$ and $b\in V$.  We have
	$U\setminus\{a\}\ne\emptyset$: if $U=\{a\}$, $2$-connectivity forces
	$aq\in E(M)$ and the $U$-block is the path $u-a-q$.  If $aq\in E(M)$, then
	$\{a,q\}$ is a clique cutset of $M$, already excluded.  Thus $aq\notin E(M)$.
	
	Now $\{a,q\}$ is a proper $2$-cutset of $G$, with sides
	$U_0=U\setminus\{a\}$ and $V\cup Y$.  The first block cannot be an
	$a$--$q$ path, since adjoining the leaf edge $ua$ would make the old
	$U$-block a $u$--$q$ path.  The second block contains $B_Y$, so it is not a path.
	Finally, $U_0\subseteq X\setminus\{q\}\subsetneq X$, contradicting
	the choice of $X$.
	
	For completeness, both new blocks do contain an $a$--$q$ path.  Removing
	$u$ from a $u$--$q$ path in the old $U$-block leaves one from $a$ to $q$
	through $U_0$.  On the other side, remove the initial edge $ub$ from a
	$u$--$q$ path through $V$ and prepend an $a$--$b$ path in the original
	$Y$-block.
\end{proof}

Put
\[
B_X=M-u=G[X\cup\{a,b\}].
\]
We next construct on $B_X$ a $4$-coloring with $a,b$ equal and one with
$a,b$ different.  We shall use repeatedly the following observation: if
$B_X$ is $3$-colorable, then, since $ab\notin E(G)$, recoloring both
$a,b$ with a new color gives the equal signature, while recoloring only
$a$ gives the unequal signature.

The graph $M$ contains no induced $K_{3,3}$.  Otherwise, \ref{L2} and Lemma~\ref{lem:marker-uncrossing} would make
$M$ a thick complete bipartite or complete tripartite graph.  Every vertex
of such a graph has degree at least three, contradicting $d_M(u)=2$.
If $M$ contains no prism, it therefore belongs to $\mathcal C$, and
Theorem~\ref{thm:rooted-marker} gives the two signatures.

Suppose that $M$ contains a prism.  By \ref{L3} and Lemma~\ref{lem:marker-uncrossing}, $M$ is a rich square or a line graph.

If $M$ is a rich square, diamond-freeness implies that \(M\) has no short links.
Thus $M$ is a long rich square, hence wheel-free, and \ref{L4.5} gives a $3$-coloring of $M$.
Its restriction to $B_X$ and the preceding observation give both signatures.

It remains that
\begin{equation}\label{eq:marker-line-graph}
	M=L(Q),
	\quad Q\text{ is simple and }\Delta(Q)\le3.
\end{equation}
Let $u$ correspond to an edge $rs$ of $Q$.  Since $d_M(u)=2$ and its two
line-graph neighbors $a,b$ are nonadjacent, we must have
$d_Q(r)=d_Q(s)=2$; the alternatives $(1,3)$ and $(3,1)$ would make the two
neighbors adjacent.  Write $rp$ and $sq$ for the other incident edges,
corresponding to $a$ and $b$.  They are disjoint.  In $Q-rs$ they are pendant
edges; in particular $p\ne q$.  Set
\begin{equation}\label{eq:line-graph-auxiliary}
	Q_0=Q-\{rs,rp,sq\},
\end{equation}
where only the edges are deleted.

For the unequal signature, take a $4$-edge-coloring of $Q_0$.  Choose a
color $\alpha$ missing at $p$.  At least two colors are missing at $q$, so
choose a missing color $\beta\ne\alpha$ and give $rp,sq$ the colors
$\alpha,\beta$, respectively.  For the equal signature, first suppose
$pq\notin E(Q_0)$.  Add the edge $pq$ to $Q_0$; the graph remains simple and
subcubic because $p\ne q$ and each has degree at most two in $Q_0$.
After $4$-edge-coloring it, delete the artificial edge and give
its color to both $rp$ and $sq$.  If $pq\in E(Q_0)$, take a
$4$-edge-coloring of $Q_0$ directly.  The union of the incident edges at
$p$ and $q$ has size at most three, so some color is missing at both; give
that color to $rp$ and $sq$.  These edge-colorings are exactly the two
vertex-coloring signatures of $B_X=L(Q-rs)$.

We have proved the following lemma.

\begin{lemma}\label{lem:two-signatures}
	Under the assumptions of this section, $G[X\cup\{a,b\}]$ has a proper
	$4$-coloring in which $a,b$ have the same color and one in which they have
	different colors.
\end{lemma}

\begin{lemma}\label{lem:proper-two-gluing}
	Suppose that every smaller $\{\mathrm{ISK}_4,\textup{diamond}\}$-free graph is $4$-colorable.  Then a graph $G$ satisfying the assumptions of this section is $4$-colorable.
\end{lemma}

\begin{proof}
	The block $B_Y=G[Y\cup\{a,b\}]$ is a proper induced subgraph and hence has a
	$4$-coloring by induction.  If its boundary colors are equal, choose the
	equal signature from Lemma~\ref{lem:two-signatures}; otherwise choose the
	unequal signature.  A permutation of the four colors on the $X$-block makes
	the ordered pair of colors on $(a,b)$ agree exactly.  The two colorings then
	glue, since $X$ is anticomplete to $Y$.
\end{proof}

\section{Proof of the main theorem}
\label{sec:main-proof}


\begin{proof}[\bf Proof of Theorem~\ref{thm:main-intro}]
	Let \(G\) be an $\{\mathrm{ISK}_4,\textup{diamond}\}$-free graph.
	We induct on $|V(G)|$.  Components can be colored independently, so assume
	$G$ is connected.Suppose that \(G\) has a clique cutset \(K\). So \(V(G)\setminus K\) can be partitioned into two nonempty sets \(X, Y\) such that there is no edge between them. Since \(G\) contains no \(\mathrm{ISK}_4\), we have \(|K|\leqslant 3\). By the induction hypothesis, the two subgraphs of \(G\) induced by \(X\cup K\) and \(Y\cup K\) are 4-colorable. We can combine 4-colorings of these subgraphs so that they coincide on \(K\), and consequently we obtain a 4-coloring of \(G\).
	Hence we may assume that $G$ has no clique cutset.
	
	If $G$ is \(\{\mathrm{ISK}_4,\text{diamond}, \text{prism},K_{3,3}\}\)-free, then Corollary~\ref{thm:core-transversal} applies.  Otherwise
	$G$ contains an induced $K_{3,3}$ or a prism.
	
	If $G$ contains an induced $K_{3,3}$, then by \ref{L2}, and since \(G\) has no clique cutset,
	\(G\) is a thick complete bipartite or complete tripartite graph.  Such a graph is at most $3$-colorable.
	
	We may therefore assume that $G$ is $K_{3,3}$-free and contains a prism.  By \ref{L3}, $G$ is one of the following: a line graph
	$L(Q)$ of a simple graph of maximum degree at most three, a rich square, a graph with a clique
	cutset, or a graph with a proper $2$-cutset.  The clique-cutset case has
	already been excluded.  In the line-graph case, \ref{L5} gives \(\chi(G)=\chi'(Q)\le4\).
	In the rich-square case, diamond-freeness excludes short links, so
	$G$ is a wheel-free long rich square and is $3$-colorable by \ref{L4.5}.
	In the remaining case, choose a proper split with a minimum oriented side and apply Lemma~\ref{lem:proper-two-gluing}.
	
	All decomposition outcomes have now been covered, and the induction closes.
\end{proof}

\section{Sharpness of the four-color bound}\label{sec:sharpness}


Let \(A=\{a_1,a_2,a_3\}\), \(B=\{b_1,b_2,b_3\}\) be the two parts of $K_{3,3}$.  Delete the edge $a_1b_1$, introduce a new
vertex $w$, and add the edges $a_1w$ and $wb_1$.  Denote the resulting graph
by $H$, and put \(G=L(H)\) (see \autoref{fig:sharpness}).
Thus $H$ is obtained by subdividing $a_1b_1$ exactly once.  For later
reference, write
\[
x_1=a_1w,\quad x_2=wb_1,
\quad e_{ij}=a_ib_j\quad ((i,j)\ne(1,1)).
\]

\begin{figure}[htbp]
	\centering
	\begin{tikzpicture}[
		vertex/.style={circle,draw,fill=white,inner sep=1.4pt,minimum size=5.5mm},
		lvertex/.style={circle,draw,fill=white,inner sep=0.8pt,minimum size=6.4mm,
			font=\scriptsize},
		hedge/.style={line width=.65pt},
		aedge/.style={draw=blue,line width=.8pt},
		bedge/.style={draw=red,line width=.8pt}]
		\coordinate (a1c) at (0,3);
		\coordinate (a2c) at (0,1.5);
		\coordinate (a3c) at (0,0);
		\coordinate (b1c) at (4,3);
		\coordinate (b2c) at (4,1.5);
		\coordinate (b3c) at (4,0);
		\coordinate (wc)  at (2,3);
		\draw[hedge] (a1c)--(wc) node[midway,above,font=\scriptsize] {$x_1$}
		--(b1c) node[midway,above,font=\scriptsize] {$x_2$};
		\foreach \u/\v in {a1c/b2c,a1c/b3c,a2c/b1c,a2c/b2c,a2c/b3c,
			a3c/b1c,a3c/b2c,a3c/b3c}
		\draw[hedge] (\u)--(\v);
		\node[vertex,label=left:$a_1$] at (a1c) {};
		\node[vertex,label=left:$a_2$] at (a2c) {};
		\node[vertex,label=left:$a_3$] at (a3c) {};
		\node[vertex,label=right:$b_1$] at (b1c) {};
		\node[vertex,label=right:$b_2$] at (b2c) {};
		\node[vertex,label=right:$b_3$] at (b3c) {};
		\node[vertex,label=above:$w$] at (wc) {};
		\node[font=\small] at (2,-.72) {$(a)\ H$};
		
		\begin{scope}[xshift=7.2cm]
			\coordinate (x1c)  at (0,4.8);
			\coordinate (x2c)  at (-1.3,3.6);
			\coordinate (e12c) at (1.2,3.6);
			\coordinate (e13c) at (3.6,3.6);
			\coordinate (e21c) at (-1.3,1.8);
			\coordinate (e22c) at (1.2,1.8);
			\coordinate (e23c) at (3.6,1.8);
			\coordinate (e31c) at (-1.3,0);
			\coordinate (e32c) at (1.2,0);
			\coordinate (e33c) at (3.6,0);
			
			\draw[aedge] (x1c)--(e12c)--(e13c)--cycle;
			\draw[aedge] (e21c)--(e22c)--(e23c);
			\draw[aedge] (e21c) to[bend left=18] (e23c);
			\draw[aedge] (e31c)--(e32c)--(e33c);
			\draw[aedge] (e31c) to[bend right=18] (e33c);
			
			\draw[bedge] (x2c)--(e21c)--(e31c);
			\draw[bedge] (x2c) to[bend right=18] (e31c);
			\draw[bedge] (e12c)--(e22c)--(e32c);
			\draw[bedge] (e12c) to[bend left=15] (e32c);
			\draw[bedge] (e13c)--(e23c)--(e33c);
			\draw[bedge] (e13c) to[bend right=15] (e33c);
			
			\draw[line width=1pt] (x1c)--(x2c);
			
			\node[lvertex] at (x1c)  {$x_1$};
			\node[lvertex] at (x2c)  {$x_2$};
			\node[lvertex] at (e12c) {$e_{12}$};
			\node[lvertex] at (e13c) {$e_{13}$};
			\node[lvertex] at (e21c) {$e_{21}$};
			\node[lvertex,fill=yellow] at (e22c) {$e_{22}$};
			\node[lvertex] at (e23c) {$e_{23}$};
			\node[lvertex] at (e31c) {$e_{31}$};
			\node[lvertex] at (e32c) {$e_{32}$};
			\node[lvertex] at (e33c) {$e_{33}$};
			\node[font=\small] at (1.2,-.72) {$(b)\ G=L(H)$};
		\end{scope}
	\end{tikzpicture}
	\caption{The sharpness example.  In panel~$(b)$, the blue edges form the three
		cliques corresponding to the vertices \(a_1,a_2,a_3\) of \(H\), the red edges form the three cliques corresponding
		to \(b_1,b_2,b_3\), and the black edge \(x_1x_2\) corresponds to the new vertex
		\(w\).
		The blue triangle $\{e_{21},e_{22},e_{23}\}$ and the red triangle
		$\{e_{12},e_{22},e_{32}\}$ induce a bowtie with centre $e_{22}$.}
	\label{fig:sharpness}
\end{figure}
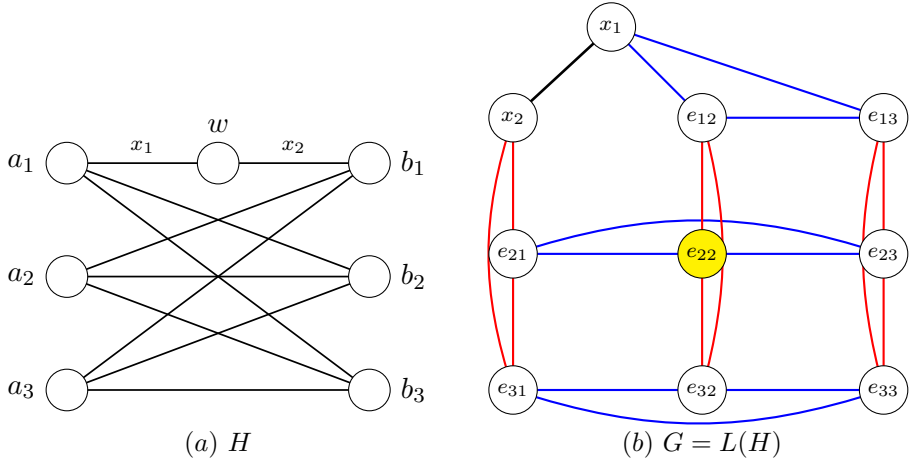

We first record a convenient observation about induced subdivisions in
claw-free graphs.

\begin{lemma}\label{lem:claw-diamond-isk4}
	Every $\{K_4,\textup{claw, diamond}\}$-free graph is $\mathrm{ISK}_4$-free.
\end{lemma}
\begin{proof}
	Suppose that an induced subgraph $Q$ is a subdivision of $K_4$ with branch
	vertices $a,b,c,d$.  Every triangle of $Q$ consists of branch vertices,
	since the two neighbors of an internal vertex are nonadjacent.
	Claw-freeness forces every branch vertex to lie in a triangle; relabel so
	that $abc$ is one.  Since $d$ also lies in a triangle, it is adjacent to at
	least two of $a,b,c$.  Hence $\{a,b,c,d\}$ induces a diamond or a $K_4$,
	a contradiction.
\end{proof}

\begin{proposition}\label{prop:sharpness}
	The graph $G=L(H)$ in~\autoref{fig:sharpness} is $\{\mathrm{ISK}_4,\textup{diamond}\}$-free, and \(\chi(G)=4\).
\end{proposition}

\begin{proof}
	We divide the proof into the structural assertion and the chromatic
	assertion.
	
	The graph $H$ is simple, triangle-free, and has maximum degree three.
	For any edge $rs\in E(H)$, the neighborhood of its vertex in $L(H)$
	is the disjoint union of the cliques formed by the other edges at $r$
	and at $s$.  There is no edge between these cliques, since such an edge
	would give a triangle in $H$; each clique has at most two vertices.
	Thus every neighborhood in $L(H)$ is a matching with isolated vertices.
	It contains neither a triangle nor an induced three-vertex path, excluding
	$K_4$ and diamonds in $L(H)$, respectively.
	As a union of at most two cliques, it also has stability number at most two, excluding claws.
	Lemma~\ref{lem:claw-diamond-isk4} now implies that $G$ is ISK$_4$-free.
	
	A vertex-coloring of $L(H)$ is precisely an edge-coloring of $H$, and so $\chi(G)=\chi'(H)$.  Since $H$ is a simple graph of maximum degree three, \ref{L5} gives $\chi'(H)\leq 4$.
	
	The graph $H$ has seven vertices and ten edges.  Each color class in
	a proper edge-coloring is a matching, so it contains at most
	$\lfloor 7/2\rfloor=3$ edges.  Three color classes would therefore
	cover at most nine edges, fewer than $|E(H)|=10$.  Thus
	$\chi'(H)\ge4$, and consequently $\chi(G)=\chi'(H)=4$.
\end{proof}

\begin{proof}[\bf Proof of Theorem~\ref{thm:sharp-intro}.]
	It follows from Proposition~\ref{prop:sharpness}.
\end{proof}


\section*{Acknowledgements}
\textbf{Qing Cui} was supported by the National Natural Science Foundation of China (No. 12171239).
\textbf{Kaiyang Lan} was partially supported by the Foundation for Cultivated Young Talents of Fujian Province, China (Grant No. 2026350294), by the Natural Science Foundation of Fujian Province, China (Grant No. 2026J001968), and by the Fujian Key Laboratory of Granular Computing and Applications (Minnan Normal University), the Institute of Meteorological Big Data-Digital Fujian, and the Fujian Key Laboratory of Data Science and Statistics.


%

\subsection*{Data availability}
Data sharing is not applicable to this article as no datasets were generated or analysed during the current study.

\subsection*{Conflict of interest}
The authors declare that they have no known competing financial interests or personal relationships that could have appeared to influence the work reported in this paper.

\end{document}